\documentclass[defaultstyle,11pt]{thesis}

\usepackage{amsmath,amsthm,amssymb,amsfonts}
\usepackage{tikz}
\usepackage{color}
\usetikzlibrary{calc}

\usepackage{soul}

\tikzset{every line/.style ={thin}}
\tikzset{linea/.style ={every line, black}, lineb/.style={every line, blue!50}, linec/.style={every line, green!50}, lined/.style={every line, yellow!90!black}}
\tikzset{linee/.style ={gray!60}}  
\tikzset{shade/.style ={gray!30}}  
\tikzset{shade b/.style ={gray!50}} 
\tikzset{shade c/.style ={black}}

\title{Homology Representations Arising from a Hypersimplex }

\author {Jacob Tyler} {Harper}

\otherdegrees{B.S., University of Denver, 2005}

\degree{Doctor of Philosophy}
  {Ph.D., Mathematics}
  
\dept{Department of}			
	{Mathematics}		

\advisor{Prof.}				
	{Richard M. Green}			
	
\reader{Nathaniel Thiem}		
	
\abstract{\OnePageChapter
We present a complete acyclic matching of the Hasse diagram associated with the face lattice of a hypersimplex. Since a hypersimplex is a convex polytope, there is a natural way to form a CW complex from its faces. We will then utilize this matching along with discrete Morse theory and some topological techniques to classify every subcomplex whose reduced homology groups are concentrated in a single degree. These reduced homology groups support a natural action of the symmetric group and a description of the characters that this action produces is given. 
  }
  
\acknowledgements{	\OnePageChapter	
I would like to thank my advisor Richard M. Green for helping me become a better mathematician. Without his guidance this thesis would not have been possible. I thank my second reader Nathaniel Thiem for providing me with useful comments. I would also like to thank Marty Walter, Markus Pflaum, and Rob Maier for being on my committee. I am grateful for the support I received from NSF grant DMS-0905768. Lastly, I would like to thank my Mom and Dad for everything they have done for me and my wife Megan for all her support and encouragement over the years.
	}
  
\ToCisShort	

\LoFisShort	

	\emptyLoT	

\begin{document}

\theoremstyle{plain}
\newtheorem{thm}{Theorem}[section]
\newtheorem{lem}[thm]{Lemma}
\newtheorem{prop}[thm]{Proposition}
\newtheorem{cor}[thm]{Corollary}

\theoremstyle{definition}
\newtheorem{defn}[thm]{Definition}
\newtheorem{conj}[thm]{Conjecture}
\newtheorem{eg}[thm]{Example}

\theoremstyle{remark}
\newtheorem{rem}[thm]{Remark}
\newtheorem{noteU}[thm]{Note}
\newtheorem{case}[thm]{Case}

\newcommand*{\longerrightarrow}{\ensuremath{\relbar\joinrel\relbar\joinrel\relbar\joinrel\rightarrow}}

\chapter{Introduction}
\label{c:intro}

\section{Coxeter Groups}
\label{s:Coxeter}

We start by recalling the definition and some key properties of Coxeter groups. 

\begin{defn}
A \textbf{Coxeter system} is a pair $(W, S)$ consisting of a group $W$ and a set of generators $S\subset W$, subject only to relations of the form $(ss')^{m(s,s')}=1$, where $m(s,s)=1$ and $m(s,s') = m(s',s) \geq 2$ for $s\neq s'$ in $S$. In case no such relation occurs for a pair $s,s'$, we make the convention that $m(s,s')=\infty$. The group $W$ itself is called a \textbf{Coxeter group} and the subgroups given by $W_I = \langle s\in I \ | \ I\subset S \rangle$ are called the \textbf{parabolic subgroups}.
\end{defn}

\begin{eg}
\label{eg:SnPres}
The symmetric group has the following presentation: 
\begin{center}
$S_n = \langle s_1,s_2,\ldots,s_{n-1}\ | \ s_i^2=1, (s_is_{i+1})^3=1,$ and $(s_is_j)^2=1$ if $|i-j|\geq 2\rangle$
\end{center}
where the generator $s_i$ can be identified with the transposition $(i, i+1)$. The focus of this thesis will be on this particular Coxeter group.
\end{eg}

It is usually very difficult to say much about a group given only by generators and relations, but in the case of Coxeter groups there is a nice result that tells us the order of the elements mentioned in the presentation.

\begin{prop}
Let $(W, S)$ be a Coxeter system. 
\begin{enumerate}
\item [\rm(i)] There is a unique epimorphism $\epsilon:W\rightarrow\{1,-1\}$ sending each generator $s\in S$ to $-1$. In particular, each $s$ has order 2 in $W$.
\item [\rm(ii)] The order of $ss'\in W$ for any $s,s'\in S$ is precisely $m(s,s')$.
\end{enumerate}
\end{prop}
\begin{proof}
Part (i) is immediate from the definition and is also \cite[Proposition 5.1]{Humphreys}. Part (ii) is \cite[Proposition 5.3]{Humphreys}
\end{proof}

\begin{defn}
Consider a Coxeter system $(W,S)$. Since $s^{-1}=s$ for every $s\in S$, every $w\neq1$ in $W$ can be written in the form $w=s_{i_1}s_{i_2}\ldots s_{i_r}$ for some $s_{i_j}$ (not necessarily distinct) in $S$. If $r$ is as small as possible, we call $r$ the \textbf{length} of $w$, written $\ell(w)$.
\end{defn}

\begin{lem}
\label{lem:Coxeter}
Consider a Coxeter group $W$ and some parabolic subgroup $W_I$. Define $$W^I = \{w\in W \ | \ \ell(ws) > \ell(w) \ \mathrm{for\ all}\ s\in I\}.$$ Given $w\in W$, there is a unique $u\in W^I$ and a unique $v\in W_I$ such that $w=uv$. Their lengths satisfy $\ell(w)=\ell(u)+\ell(v)$. Moreover, $u$ is the unique element of smallest length in the coset $wW_I$.
\end{lem}
\begin{proof}
This follows from \cite[Proposition 1.10(c)]{Humphreys}.
\end{proof}

All of the information about a Coxeter system $(W, S)$ can be encoded into a graph $\Gamma_W$. The vertex set of $\Gamma_W$ is in one-to-one correspondence with $S$ and there exists an edge joining the vertices corresponding to $s\neq s'$ whenever $m(s,s')\geq3$. The edges are labeled by $m(s,s')$, although it is common to omit the label 3 because it occurs so often. Note that this implies if the two vertices corresponding to $s$ and $s'$ are not joined by an edge, then $m(s,s')=2$.

\begin{defn}
Given an arbitrary Coxeter system $(W, S)$, the graph $\Gamma_W$ described above is called the \textbf{Coxeter graph} of $W$.
\end{defn}

\begin{eg}
Figure \ref{fig:CgraphA} below shows the Coxeter graph for the Coxeter group of type $A_{n-1}$. It has exactly $n-1$ vertices and all of the edges shown could be labeled by 3, but by convention we omit the labels. Notice if we translate the information from the graph into a presentation we arrive at the presentation shown in Example \ref{eg:SnPres} which shows that the Coxeter group of type $A_{n-1}$ is isomorphic to the symmetric group $S_n$.

\vspace{.15in}
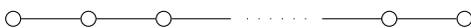
\begin{figure}[htbp]
\begin{center}
\begin{tikzpicture}[scale=1]
\node (a) at (0,0) [circle, draw, inner sep=2pt] {};
\node (b) at (1,0) [circle, draw, inner sep=2pt] {};
\draw (a) to (b);
\node (c) at (2,0) [circle, draw, inner sep=2pt] {};
\draw (b) to (c);
\node (d) at (3,0) [circle, inner sep=2pt] {};
\draw (c) to (d);
\node (e) at (4,0) [circle, inner sep=2pt] {};
\draw[loosely dotted] (d) to (e);
\node (f) at (5,0) [circle, draw, inner sep=2pt] {};
\draw (e) to (f);
\node (g) at (6,0) [circle, draw, inner sep=2pt] {};
\draw (f) to (g);
\end{tikzpicture}
\caption{Coxeter graph of type $A_{n-1}$}
\label{fig:CgraphA}
\end{center}
\end{figure}
\end{eg}

It will be important to understand how Coxeter groups act on vector spaces. One of the most common representations, sometimes known as the geometric representation, deals with the following definition and theorem.

\begin{defn}
\label{def:refgp}
Let $V$ be a real euclidean space endowed with a positive definite symmetric bilinear form. A \textbf{reflection} is a linear operator $s$ on $V$ that sends some nonzero vector $\alpha$ to its negative while fixing pointwise the hyperplane $H_\alpha$ orthogonal to $\alpha$. We may write $s = s_\alpha$ to denote such a reflection. A \textbf{finite reflection group} is a finite group generated by reflections.
\end{defn}

\begin{thm}
Every finite Coxeter group has a representation as a finite reflection group and every finite reflection group has a presentation as a Coxeter group.
\end{thm}

\begin{proof}
This is a well-known result. The reader is directed towards \cite[Theorem 6.4]{Humphreys} for one version of this proof.
\end{proof}

More details of this representation can be found in \cite[\S 5.3]{Humphreys}, but one way to arrive at such a representation is the following. If we are given a Coxeter system $(W,S)$, then we begin with a vector space $V$ over $\mathbb{R}$, having basis $\{\alpha_s \ | \ s\in S\}$ in one-to-one correspondence with $S$. Next we define a symmetric bilinear form $B$ on $V$ by requiring $B(\alpha_s, \alpha_{s'}) = -\mathrm{cos}\ \frac{\pi}{m(s,s')}$. We can now define the reflection $\sigma_s:V\rightarrow V$ by the rule $\sigma_s \lambda = \lambda - 2B(\alpha_s,\lambda)\alpha_s$. Finally, there exists a unique homomorphism $\sigma :W\rightarrow GL(V)$ sending $s$ to $\sigma_s$. It is worth noting that \cite[Theorem 6.4]{Humphreys} proves that $B$ is positive definite, or in other words $B(\lambda, \lambda)>0$ for all nonzero $\lambda\in V$,  if and only if $W$ is finite.

While the above description does work for both finite and infinite Coxeter groups, from now on we will only talk about finite Coxeter groups and hence also consider them as finite reflection groups. We proceed by giving a description of two sets of vectors that arise when talking about these types of groups, and note that the definition of a root system given here is different from the root system of a Lie algebra.

\begin{defn}
A \textbf{root system} $\Phi$ is a finite set of nonzero vectors in $V$ such that $\Phi\cap\mathbb{R}\alpha = \{\alpha, -\alpha\}$ and $s_{\alpha}\Phi = \Phi$ for all $\alpha\in\Phi$ where $s_\alpha$ is the reflection associated with $\alpha$. We say that a subset $\Delta$ of $\Phi$ is a \textbf{simple system} (and call its elements \textbf{simple roots}) if $\Delta$ is a vector space basis for the $\mathbb{R}$-span of $\Phi$ in $V$ and if moreover each $\alpha\in\Phi$ is a linear combination of $\Delta$ with coefficients all of the same sign (all nonnegative or all nonpositive).
\end{defn}

Every finite reflection group can be generated by the reflections corresponding to a root system (see \cite[\S 1.2]{Humphreys}).
Given a root system, \cite[Theorem 1.3]{Humphreys} proves that there exists a corresponding simple system and \cite[Theorem 1.5]{Humphreys} proves that the reflections corresponding to a set of simple roots also generate the group. 
Finally, and not too surprising given the relationship between finite Coxeter groups and finite reflection groups, \cite[\S 5.4]{Humphreys} shows that the elements of a simple system $\Delta$ are in one-to-one correspondence with the elements of $S$.

Now with $W$ acting on $V$ as a finite reflection group, the following theorem shows us one reason why simple systems are helpful.

\begin{thm}
Every vector in $V$ is conjugate under $W$ to one and only one point in $$D=\{\lambda\in V \ | \ (\lambda,\alpha)\geq0 \ \mathrm{for\ all} \  \alpha\in\Delta\}.$$ The set $D$ is known as a \textbf{fundamental domain} for the action of $W$ on $V$.
\end{thm}

\begin{proof}
This is \cite[Theorem 1.12(b)]{Humphreys}.
\end{proof}

\begin{eg}
\label{eg:fundom}
Let $V=\mathbb{R}^n$ with bilinear form equal to the dot product and let $W=S_n$ act on $V$ by permuting the coordinates. Define 
$$\Phi = \{\epsilon_i-\epsilon_j \ | \ 1\leq i \neq j \leq n\}$$ and  $$\Delta = \{\epsilon_i-\epsilon_{i+1} \ | \ 1\leq i \leq n-1 \},$$ 
where the $\epsilon_i$ are the standard basis vectors for $\mathbb{R}^n$. 
It follows from the definitions that $\Phi$ is a root system and $\Delta$ is a simple system for the finite reflection group associated with the Coxeter group $W$; for details see \cite[\S 2.10]{Humphreys}. In this example, 
$$D = \{(x_1, x_2, \ldots, x_n)\in\mathbb{R}^n \ | \ x_i\geq x_{i+1} \ \mathrm{for} \ \mathrm{all} \ i\}$$ 
since $(x_1, x_2, \ldots, x_n)\cdot (\epsilon_i-\epsilon_{i+1}) = x_i - x_{i+1}$.
\end{eg}

\section{Representation Theory of $S_n$}

Later on in this thesis we will study how the symmetric group acts on certain subsets of the hypersimplex. Fortunately, the representation theory and character theory of $S_n$ is well understood. In this section we shall summarize some of the key points of the theory and direct the reader to Geck and Pfeiffer's book \cite{CharactersCoxeter} and Fulton's book \cite{Fulton} for more details on the subject.

\begin{defn}
Given a group $G$, written multiplicatively, the \textbf{group algebra} $\mathbb{C}G$ is the $\mathbb{C}$-vector space with basis $G$ with multiplication given by the multiplication of $G$ and extended linearly.
\end{defn}

\begin{defn}
Let $n\in\mathbb{Z}$. A \textbf{representation} of a group $G$ over $\mathbb{C}$ is a group homomorphism $\rho:G\rightarrow GL_n(\mathbb{C})$.
\end{defn}

Notice that a representation $\rho:G\rightarrow GL_n(\mathbb{C})$ defines a group action of $G$ on an $n$-dimensional vector space $V$ over $\mathbb{C}$ in the following way. 
Fix a basis $\beta=\{b_1,\ldots, b_n\}$ for $V$ and define the action $G\times V\rightarrow V$ by letting $g\cdot v=\displaystyle\sum_{i=1}^{n}\left(\displaystyle\sum_{j=1}^{n}a_{ij} v_j\right)b_i$ where $v=\displaystyle\sum_{i=1}^{n}v_i b_i$ and $(\rho(g))_{i,j}=a_{ij}$.

Moreover this action completely determines the representation since we can recover $\rho$ given $V$ as long as we have the same basis. Hence to specify a representation it is enough to specify how it acts on its representing vector space. 
Alternatively, the action of a group $G$ on a complex vector space $V$ induces a left action of the group algebra $\mathbb{C}G$ on the vector space $V$, and vice versa which means that these group actions are equivalent to left $\mathbb{C}G$-modules. It is common to refer to $V$ itself as the representation when the map $\rho$ is clear from context even though this is an abuse of language. Also note that we can define representations over any field $K$ but in this thesis we will always use $\mathbb{C}$.

\begin{defn}
Let $V$ be the representation of $G$ corresponding to a map $\rho:G\rightarrow GL_n(\mathbb{C})$. The dimension of $V$ is also called the \textbf{dimension}, or \textbf{degree}, of the representation. A \textbf{subrepresentation} of $V$ is a subspace $U$ that is $G$-invariant, i.e., $g\cdot u\in U$ for all $g\in G$ and $u\in U$. If $V$ is nonzero and has no subrepresentations other than $\{0\}$ and $V$ then we say that $V$ is \textbf{irreducible}. If $V'$ is another representation of $G$ corresponding to a map $\rho':G\rightarrow GL_n(\mathbb{C})$, then we say that $V$ and $V'$ are \textbf{isomorphic} or \textbf{equivalent} if there exists a vector space isomorphism $\theta:V\rightarrow V'$ such that for all $g\in G$, $\theta\circ\rho(g)\circ\theta^{-1} = \rho'(g)$.
\end{defn}

The one dimensional representations of $S_n$ are called the trivial and sign representations. The trivial representation, denoted by id$_n$, is the unique homomorphism id$_n:S_n\rightarrow GL_1(\mathbb{C})$ that sends every element of $S_n$ to the identity. The sign representation, denoted by sgn$_n$, is the unique homomorphism sgn$_n:S_n\rightarrow GL_1(\mathbb{C})$ that sends every transposition $(i,i+1)$ of $S_n$ to $-1$. We will later rely heavily on these two representations to construct the remaining representations that we are interested in.

\begin{defn}
Given a representation $\rho:G\rightarrow GL_n(\mathbb{C})$, the corresponding \textbf{character} of $\rho$ is the function $\chi_\rho:G\rightarrow \mathbb{C}$ given by $\chi_\rho (g) = \mathrm{tr}(\rho(g))$ where $\mathrm{tr}$ is the trace function. The character of an irreducible representation is an \textbf{irreducible character}.
\end{defn}

Now before we can describe the irreducible representations, and hence describe the irreducible characters of the symmetric group, it will be beneficial to introduce some more terminology.

\begin{defn}
A \textbf{partition} of a positive integer $n$ is a way of writing $n$ as a sum of positive integers where the order of those integers does not matter. If $n=n_1+n_2+\cdots+n_i$ such that $n_1\geq n_2\geq\cdots\geq n_i\in\mathbb{Z}_{\geq0}$ then we denote the corresponding partition by $[n_1,n_2,\ldots,n_i]$.
\end{defn}

\begin{defn}
A \textbf{Young diagram} is a collection of boxes, or cells, arranged in left-justified rows, with a (weakly) decreasing number of boxes in each row from top to bottom. 
\end{defn}

\begin{eg}
Figure \ref{fig:Young} below shows a few Young diagrams. Notice that listing the number of boxes in each row from top to bottom gives a partition of $n$ where $n$ is the total number of boxes. Likewise, every partition corresponds to a Young diagram. We can label the Young diagrams below from left to right using the corresponding partitions in the following way: $[6,4,4,2]$, $[4,4,3,3,1,1]$, $[5]$, and $[1,1,1,1,1]$ which can also be denoted by $[1^5]$.

\vspace{.15in}
\begin{figure}[htbp]
\begin{center}
\begin{tikzpicture}[scale=.55]
\def \labelh{-6.5}

\draw [step=1,thin] (0,0) grid (6, 1);
\draw [step=1,thin] (0,-1) grid (4, 0);
\draw [step=1,thin] (0,-2) grid (4, -1);
\draw [step=1,thin] (0,-3) grid (2, -2);
\draw node at (3,\labelh) {$[6,4,4,2]$};

\draw [step=1,thin] (8,-1) grid (12, 1);
\draw [step=1,thin] (8,-3) grid (11, -1);
\draw [step=1,thin] (8,-5) grid (9, -3);
\draw node at (10,\labelh) {$[4,4,3,3,1,1]$};

\draw [step=1,thin] (14,0) grid (19, 1);
\draw node at (16.5,\labelh) {$[5]$};

\draw [step=1,thin] (21,-4) grid (22, 1);
\draw node at (21.5,\labelh) {$[1^5]$};

\end{tikzpicture}
\caption{Young diagrams}
\label{fig:Young}
\end{center}
\end{figure}
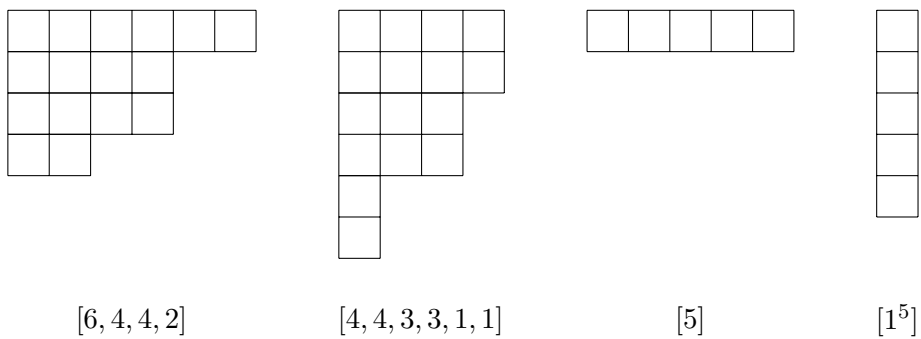
\end{eg}

\begin{defn}
Any way of putting a positive integer in each box of a Young diagram $\lambda$ is called a \textbf{filling}. A \textbf{tableau} (plural \textbf{tableaux}) is a filling that is
\begin{enumerate}
\item [\rm(i)] weakly increasing across each row and
\item [\rm(ii)] strictly increasing down each column.
\end{enumerate}
(Some authors use the term ``column-strict tableau" to mean what we call a ``tableau".)
A \textbf{standard tableau} is a tableau in which the entries are the numbers from $1$ to $n$, each occurring once. The partition $\lambda$ is sometimes referred to as the \textbf{shape} of the tableau. 
\end{defn}

\begin{eg}
Figure \ref{fig:tab} below shows two different ways to fill the Young diagram $[6,4,4,2]$. Each filling is a tableau, but only the filling on the right is a standard tableau.

\vspace{.15in}
\begin{figure}[htbp]
\begin{center}
\begin{tikzpicture}[scale=.7]
\def \labelh{-3.8}
\def \split{3}
\draw node at (-1.5,0) {};

\draw [step=1,thin] (0,0) grid (6, 1);
\draw [step=1,thin] (0,-1) grid (4, 0);
\draw [step=1,thin] (0,-2) grid (4, -1);
\draw [step=1,thin] (0,-3) grid (2, -2);

\draw [step=1,thin] (6+\split,0) grid (12+\split, 1);
\draw [step=1,thin] (6+\split,-1) grid (10+\split, 0);
\draw [step=1,thin] (6+\split,-2) grid (10+\split, -1);
\draw [step=1,thin] (6+\split,-3) grid (8+\split, -2);

\draw node at (.5,.5) {1};
\draw node at (1.5,.5) {2};
\draw node at (2.5,.5) {2};
\draw node at (3.5,.5) {3};
\draw node at (4.5,.5) {3};
\draw node at (5.5,.5) {5};
\draw node at (.5,-.5) {2};
\draw node at (1.5,-.5) {3};
\draw node at (2.5,-.5) {5};
\draw node at (3.5,-.5) {5};
\draw node at (.5,-1.5) {4};
\draw node at (1.5,-1.5) {4};
\draw node at (2.5,-1.5) {6};
\draw node at (3.5,-1.5) {6};
\draw node at (.5,-2.5) {5};
\draw node at (1.5,-2.5) {6};

\draw node at (6+\split+.5,.5) {1};
\draw node at (6+\split+1.5,.5) {3};
\draw node at (6+\split+2.5,.5) {7};
\draw node at (6+\split+3.5,.5) {12};
\draw node at (6+\split+4.5,.5) {13};
\draw node at (6+\split+5.5,.5) {15};
\draw node at (6+\split+.5,-.5) {2};
\draw node at (6+\split+1.5,-.5) {5};
\draw node at (6+\split+2.5,-.5) {10};
\draw node at (6+\split+3.5,-.5) {14};
\draw node at (6+\split+.5,-1.5) {4};
\draw node at (6+\split+1.5,-1.5) {8};
\draw node at (6+\split+2.5,-1.5) {11};
\draw node at (6+\split+3.5,-1.5) {16};
\draw node at (6+\split+.5,-2.5) {6};
\draw node at (6+\split+1.5,-2.5) {9};

\end{tikzpicture}
\caption{Tableaux}
\label{fig:tab}
\end{center}
\end{figure}
\end{eg}

Using the notion of tableaux we can now start to describe the vector spaces on which the symmetric group will act. First of all, let $T$ denote a filling for a Young diagram with $n$ boxes with the numbers from 1 to $n$, with no repeats allowed. Notice that the symmetric group $S_n$ acts on the set of such fillings, with $\sigma\cdot T$ being the filling that puts $\sigma(i)$ in the box in which $T$ puts $i$. For a filling $T$ we have a subgroup $R(T)$ of $S_n$, the row group of $T$, which consists of the permutations that permute the entries of each row among themselves. Similarly we have the column group $C(T)$ of permutations preserving the columns.

\begin{defn}
A \textbf{tabloid} is an equivalence class of fillings of a Young diagram (with distinct numbers $1,\ldots,n$), two being equivalent if corresponding rows contain the same entries. The tabloid determined by a filling $T$ is denoted $\{T\}$. If we let $\lambda$ be a partition of $n$ then define $M^\lambda$ to be the complex vector space with basis the tabloids $\{T\}$ of shape $\lambda$ and let $v_T \in M^\lambda$ be equal to $\displaystyle\sum_{q\in C(T)}\mathrm{sgn}(q)\{q\cdot T\}$. Let $S^\lambda$ be the subspace of $M^\lambda$ spanned by the elements $v_T$ as $T$ varies over all fillings of $\lambda$.
\end{defn}

\begin{prop}
\label{prop:tabs}
Consider the symmetric group $S_n$ and let $T$ be a filling for a Young diagram with n boxes with the numbers from 1 to $n$, with no repeats allowed.
\begin{enumerate}
\item [\rm(i)] The action $\sigma\cdot v_T = v_{\sigma\cdot T}$ for all $T$ and all $\sigma\in S_n$.
\item [\rm(ii)] The elements $v_T$, as $T$ varies over the standard tableaux of shape $\lambda$, form a basis for $S^\lambda$.
\item [\rm(iii)] For each partition $\lambda$ of $n$, $S^\lambda$ is an irreducible representation of $S_n$. Every irreducible representation of $S_n$ is isomorphic to exactly one $S^\lambda$.
\end{enumerate}
\end{prop}

\begin{proof}
This is shown in \cite[\S 7.2]{Fulton}.
\end{proof}

It should be noted that Proposition \ref{prop:tabs} (i) does not tell us explicitly how $S_n$ acts on the basis elements of $S^{\lambda}$, since $\sigma\cdot T$ may not be a standard tableau whenever $T$ is.

Given a partition $\lambda$ of $n$, we will write the corresponding character of $S_n$ as $\chi^\lambda$. In particular, the trivial character corresponds to the partition $\lambda=[n]$ and hence is denoted by $\chi^{[n]}$ while the sign character corresponds to the partition $\lambda=[1,1,\ldots,1]=[1^n]$ and hence is denoted by $\chi^{[1^n]}$.

\begin{defn}
Let $\rho_m:S_m\times U\rightarrow U$ and $\rho_n:S_n\times V\rightarrow V$ be two representations. Define the representation $\rho_m\otimes\rho_n$ of $(S_m \times S_n)$ in the following way: $$\rho_m\otimes\rho_n:(S_m \times S_n)\times(U\otimes V)\rightarrow U\otimes V$$ $$\rho_m\otimes\rho_n((g,h),(u\otimes v))\mapsto\rho_m(g,u)\otimes\rho_n(h,v)$$
and extend linearly for all $g\in S_m$, $h\in S_n$, $u\in U$, and $v\in V$. Denote the corresponding character by $\chi_{\rho_m}\times\chi_{\rho_n}$.
\end{defn}

Next notice that if we let $V$ be a representation of a group $G$ and $H$ be a subgroup of $G$ then $V$ is also a representation of $H$ by restriction. This representation is denoted by $\mathrm{Res}^G_H V$ or $(V)\downarrow^G_H$. 
Likewise, we can start with a representation of $H$ and define a representation of $G$ as follows.

\begin{defn}
Let $H$ be a subgroup of a finite group $G$ and let $\rho:H\times V\rightarrow V$ be a representation of $H$. We can define a representation on $G$, known as the \textbf{induced representation} and denoted by $\mathrm{Ind}^G_H V$ or $(V)\uparrow_H^G$, by letting $(V)\uparrow_H^G \  =\mathbb{C}G\otimes_{\mathbb{C}H}V$ where $G$ acts on $\mathbb{C}G\otimes_{\mathbb{C}H}V$ by letting $g_1(g_2\otimes v) = (g_1g_2)\otimes v$ and extending linearly for every $g_1,g_2\in G$ and $v\in V$. The corresponding character is denoted by $(\chi_\rho)\uparrow^G_H$ and the degree of this new representation is equal to the product of $|G:H|$ with the dimension of $V$.
\end{defn}

\begin{defn}
Let $V$ be a vector space over a field $K$. For any nonnegative integer $k$, we define the \textbf{$k$-th tensor power of $V$} to be the tensor product of $V$ with itself $k$ times: $T^k V = V^{\otimes k} = V \otimes V \otimes \cdots \otimes V$.
The \textbf{tensor algebra} of a vector space $V$, denoted by $T(V)$, is the algebra of tensors on $V$ (of any rank) with multiplication being the tensor product.
We can construct $T(V)$ as the direct sum of $T^k V$ for $k=0,1,2,\ldots$
$$T(V)=\displaystyle\bigoplus_{k=0}^{\infty} T^k V = K \oplus V \oplus (V\otimes V) \oplus (V\otimes V \otimes V) \oplus \cdots.$$
\end{defn}

\begin{defn}
The \textbf{exterior algebra} $\bigwedge V$ over a vector space $V$ over a field $K$ is defined as the quotient algebra of the tensor algebra by the two-sided ideal $I$ generated by all elements of the form $x\otimes x$ such that $x\in V$.
The \textbf{$k$-th exterior power of $V$}, denoted $\displaystyle\bigwedge^{k} V$, is the vector subspace of $\bigwedge V$ spanned by all elements of the form $x_1\wedge x_2\wedge\cdots\wedge x_k$ where $x_i\in V$ and $x_1\wedge x_2\wedge\cdots\wedge x_k$ is the image of $x_1\otimes x_2\otimes\cdots\otimes x_k$ in the quotient.
\end{defn}

Note that if a group $G$ acts on $V$, then this action extends to an action on $T^k V$ by letting 
$$g\cdot (x_1\otimes x_2\otimes\cdots\otimes x_k) = (g\cdot x_1)\otimes (g\cdot x_2)\otimes\cdots\otimes (g\cdot x_k).$$
It follows that this action will induce an action on $\displaystyle\bigwedge^{k} V$ by letting 
$$g\cdot (x_1\wedge x_2\wedge\cdots\wedge x_k) = (g\cdot x_1)\wedge (g\cdot x_2)\wedge\cdots\wedge (g\cdot x_k).$$

\begin{defn}
Let $E_n$ be the representation of $S_n$ constructed by first taking the $n$-dimensional representation of $S_n$ corresponding to the natural action of the group on $n$ letters, and then quotienting by the 1-dimensional submodule spanned by the all-ones vector.
This representation is known as the \textbf{reflection representation} and corresponds to the partition $[n-1,1]$. Note that this is the same representation referenced in Section \ref{s:Coxeter}.
\end{defn}

\begin{prop}
\label{prop:ext}
The character of the $S_n$-module $\displaystyle\bigwedge^{d} E_{n}$ for $0\leq d\leq n-1$ is given by the partition
$$\lambda=[n-d,\ 1^{d}].$$
\end{prop}
\begin{proof}
This is \cite[Proposition 5.4.12]{CharactersCoxeter}.
\end{proof}

\section{The Hypersimplex}

Polytopes are important geometric objects that often have nice combinatorial properties. The subject of study in this thesis is focused around a family of polytopes known as hypersimplices. In this section we will define these objects and explore some of their properties that will be needed in the subsequent chapters.

\begin{defn}
A subset $K\subset\mathbb{R}^n$ is said to be \textbf{convex} if the straight line segment between any two points in $K$ is also contained in $K$.
For any $K\subset\mathbb{R}^n$, the smallest convex set containing $K$ can be constructed as the intersection of all convex sets that contain $K$ and is known as the \textbf{convex hull} of $K$. A \textbf{polytope} is the convex hull of a finite set of points in $\mathbb{R}^n$.
\end{defn}

Note that if $H$ is a hyperplane in $\mathbb{R}^n$, then the complement $\mathbb{R}^n-H$ has two open components. A closed halfspace is the union of one of those two components with the hyperplane. A polytope can also be defined as the bounded intersection of finitely many closed halfspaces. It is nontrivial that these two definitions for a polytope are equivalent, but it is shown in the following theorem to be true.

\begin{thm}[\textbf{The Main Theorem for Polytopes}]
A subset $P\subset\mathbb{R}^n$ is the convex hull of a finite subset if and only if it is a bounded intersection of halfspaces.
\end{thm}
\begin{proof}
This is \cite[Theorem 1.1]{Ziegler95}. 
\end{proof}

Both definitions are helpful as some properties of polytopes are easier to understand when viewing them one way versus the other. For example, the following definition relates more to the halfspace definition of a polytope and can be found in Ziegler's book \cite{Ziegler95}.

\begin{defn}
Let $P\subseteq \mathbb{R}^n$ be a polytope. A \textbf{face} $F$ of $P$ is either $P$ itself or the intersection of $P$ with a hyperplane such that $P$ is contained in one of the two closed halfspaces determined by the hyperplane. The faces of dimension 0, 1, and $i$ are called the \textbf{vertices}, \textbf{edges}, and \textbf{$i$-faces} of $P$ respectively.
\end{defn}

In this thesis we will investigate the properties of the following polytope.

\begin{defn}
Let $J(n,k)$ be the polytope equal to the convex hull of the points in $\mathbb{R}^n$ with exactly $k$ 1's and $n-k$ 0's. For any value of $n$ and $k$ such that $1\leq k\leq n-1$, this polytope is called a \textbf{hypersimplex}.
\end{defn}

We will show later on that this set of points is equal to the set of vertices of $J(n,k)$. In fact, the purpose of the remainder of this section is to fully describe the face lattice of this polytope. 

\begin{prop}
Let $P\subseteq \mathbb{R}^n$ be a polytope. Every polytope is the convex hull of its vertices, $\mathrm{vert}(P)$, and every face $F$ of $P$ is a polytope with $\mathrm{vert}(F) = F \cap \mathrm{vert}(P)$.
\end{prop}
\begin{proof}
This is \cite[Propositions 2.2 and 2.3]{Ziegler95}.
\end{proof}

It follows that every face of a polytope $P$ corresponds to a subset of the vertices of $P$ and is also itself a polytope. However not every subset of the vertices of $P$ corresponds to a face, as we see in the following example.

\begin{eg}
Consider the space $\mathbb{R}^2$ and let $P$ be equal to the convex hull of the finite subset given by $V = \{(0,0), (1,0), (0,1), (1,1)\}$. Clearly $P$ is just a square (with its interior). The faces of $P$ are as follows: the four points (0,0), (1,0), (0,1), (1,1) make up the vertices of $P$, the four edges of $P$ are equal to the convex hull of the sets $\{(0,0), (0,1)\}$, $\{(0,0), (1,0)\}$, $\{(1,0), (1,1)\}$, and $\{(0,1), (1,1)\}$, and the convex hull of $V$ makes up the unique $2$-face of $P$.

Next consider the convex hull of the subset $\{(0,0), (1,1)\} \subset V$. The only hyperplane containing this set of points is the line in $\mathbb{R}^2$ that contains both $(0,0)$ and $(1,1)$. However, $P$ is not contained in one of the two closed halfspaces determined by this line as $P$ contains points on either side of the line. Therefore the convex hull of $\{(0,0), (1,1)\}$ is not a face of $P$.
\end{eg}

Now in order to describe the face lattice of $J(n,k)$ we first make the observation that if we let $v_0 = (1_1,\ldots, 1_k,0_{k+1},\ldots,0_n)\in\mathbb{R}^n$ and let the symmetric group $S_n$ act on $v_0$ by permuting the coordinates, then $J(n,k)$ is equal to the convex hull of this group orbit. Since $S_n$ is the Coxeter group of type $A_{n-1}$ we can use Theorem \ref{thm:Casselman} below to tell us exactly which subsets of the vertices correspond to faces of $J(n,k)$.

For the remainder of this section, let $(W,S)$ be an arbitrary Coxeter system such that $W$ is a finite Coxeter group, let $V$ be a finite dimensional real vector space such that the elements of $s\in S$ act on it by reflections, and let $\Delta = \{\alpha_s\in V \ | \ s\in S\}$ be a set of vectors corresponding to the reflections as in Section \ref{s:Coxeter}.

\begin{defn}
Let $\delta \subset \Delta$ and consider the Coxeter graph $\Gamma_W$. A subset $\kappa \subset \Delta$ is said to be \textbf{$\delta$-connected} if every one of the vertices in $\Gamma_W$ corresponding to an element of $\kappa$ can be connected to a vertex corresponding to an element of $\delta$ by a path inside itself.
\end{defn}

\begin{eg}
Let $W$ be the Coxeter group of type $A_6$. For this example we will label the vertices of $\Gamma_W$ to help illustrate the above definition. In Figure \ref{fig:deltaconnected} below we have two cases and in both of them the circled vertices indicate the set $\kappa$ while the vertices colored black indicate the set $\delta$. 

On the left we have $\kappa=\{\alpha_1,\alpha_2\}$ and $\delta=\{\alpha_2,\alpha_4\}$. Notice that $\alpha_2\in\delta$ and $\alpha_1$ is connected to $\alpha_2$ by a path inside $\kappa$. Therefore in this case, we see that $\kappa$ is $\delta$-connected.

On the right we have $\kappa=\{\alpha_1,\alpha_2,\alpha_5,\alpha_6\}$ and $\delta=\{\alpha_2,\alpha_4\}$. Notice that $\alpha_5$ is only connected to itself and $\alpha_6$ by a path inside $\kappa$, and since neither of these are elements of $\delta$, it follows that $\kappa$ is not $\delta$-connected.

\vspace{.15in}
\begin{figure}[htbp]
\begin{center}
\begin{tikzpicture}[scale=1]
\def \height{.1}

\node (a) at (0,0) [circle, draw, inner sep=2pt] {};
\node (b) at (1,0) [circle, draw, inner sep=2pt] {};
\draw (a) to (b);
\node (c) at (2,0) [circle, draw, inner sep=2pt] {};
\draw (b) to (c);
\node (d) at (3,0) [circle, draw, inner sep=2pt] {};
\draw (c) to (d);
\node (e) at (4,0) [circle, draw, inner sep=2pt] {};
\draw (d) to (e);
\node (f) at (5,0) [circle, draw, inner sep=2pt] {};
\draw (e) to (f);

\node at (0,\height) [above] {$\alpha_1$};
\node at (1,\height) [above] {$\alpha_2$};
\node at (2,\height) [above] {$\alpha_3$};
\node at (3,\height) [above] {$\alpha_4$};
\node at (4,\height) [above] {$\alpha_5$};
\node at (5,\height) [above] {$\alpha_6$};

\draw (.5,0) ellipse (1 and .2);
\fill [black] (1,0) circle (3pt);
\fill [black] (3,0) circle (3pt);

\end{tikzpicture}
%
%
%
%
\hspace{.6in}
\begin{tikzpicture}[scale=1]
\def \height{.1}

\node (a) at (0,0) [circle, draw, inner sep=2pt] {};
\node (b) at (1,0) [circle, draw, inner sep=2pt] {};
\draw (a) to (b);
\node (c) at (2,0) [circle, draw, inner sep=2pt] {};
\draw (b) to (c);
\node (d) at (3,0) [circle, draw, inner sep=2pt] {};
\draw (c) to (d);
\node (e) at (4,0) [circle, draw, inner sep=2pt] {};
\draw (d) to (e);
\node (f) at (5,0) [circle, draw, inner sep=2pt] {};
\draw (e) to (f);

\node at (0,\height) [above] {$\alpha_1$};
\node at (1,\height) [above] {$\alpha_2$};
\node at (2,\height) [above] {$\alpha_3$};
\node at (3,\height) [above] {$\alpha_4$};
\node at (4,\height) [above] {$\alpha_5$};
\node at (5,\height) [above] {$\alpha_6$};

\draw (.5,0) ellipse (1 and .2);
\draw (4.5,0) ellipse (1 and .2);
\fill [black] (1,0) circle (3pt);
\fill [black] (3,0) circle (3pt);

\end{tikzpicture}
\caption{$\delta$-connected}
\label{fig:deltaconnected}
\end{center}
\end{figure}
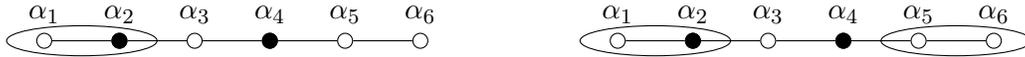
\end{eg}

\begin{thm}
\label{thm:Casselman}
Let $\delta = \{\alpha\in\Delta \ | \ (\alpha, v) \neq 0\}$ for some $v\in V$ and let $W_\kappa = \langle s\in S \ | \ \alpha_s\in\kappa \rangle$ for some $\kappa \subset \Delta$. If $v\in V$ lies in the fundamental domain 
$$D=\{\lambda\in V \ | \ (\lambda,\alpha)\geq0 \ for \ all \ \alpha\in\Delta\},$$ 
then the map taking $\kappa$ to the convex hull $F_\kappa$ of $W_\kappa\cdot v$ is a bijection between the $\delta$-connected subsets of $\Delta$ and the set of W-representatives of the faces of the convex hull of $W\cdot v$.
\end{thm}

\begin{proof}
This can be found in \cite[Theorem 3.1]{Casselman}, although Casselman states that the proof is implicit in the work of Satake in \cite[Lemma 5]{Satake} and Borel--Tits in \cite[\S 12.16]{BT}.
\end{proof}

This theorem tells us several important facts. It says that every face is $W$-conjugate to exactly one of the $F_\kappa$. It also implies that $F_\kappa$, and hence each of its $W$-conjugates, is of dimension equal to the cardinality of $\kappa$. Furthermore, it says that an increasing chain of $\delta$-connected subsets corresponds to an increasing chain of faces. This information is sufficient to describe the face lattice of $J(n,k)$. However, it should be noted that while the description given in Proposition \ref{prop:faces} is assumed to be known, a reference could not be found.

\begin{prop}
\label{prop:faces}
Let $W = S_n$, $v_0 = (1_1,\ldots, 1_k,0_{k+1},\ldots,0_n)\in \mathbb{R}^n$, and $F_I = \{W_ I \cdot v_0\}$ for some $I\subset S$. The faces of $J(n,k)$ are as follows:
\begin{enumerate}
\item [\rm(i)] $\binom{n}{k}$ 0-faces (vertices) given by $\{(x_1, x_2,\ldots,x_n)\in\mathbb{R}^n \ | \ x_j\in\{0,1\}$ and $\displaystyle\sum_{j=1}^{n} x_j=k\}$;
\item [\rm(ii)] $\binom{n}{i+1}$$\binom{n-i-1}{j-1}$ i-faces for every set $I = \{s_j,s_{j+1},\ldots, s_{j+i-1}\}\subset S$ such that $1\leq j\leq k\leq j+i-1\leq n-1$, given by the convex hull of $w\cdot F_I$ for some $w\in W$.

\end{enumerate}
\end{prop}

\begin{proof}
Let $V$ be equal to $\mathbb{R}^n$ equipped with the dot product. We start by letting $W$ act on $V$ by permuting the coordinates of each vector and noting that $J(n,k)$ is equal to the convex hull of $W\cdot v_0$.  As we saw in Example \ref{eg:fundom}, the set $\Delta = \{\epsilon_j-\epsilon_{j+1}$ $|$ $1\leq j \leq n-1 \}$ is a simple system for the finite reflection group associated with $W$ where the $\epsilon_j$ are the standard basis vectors for $\mathbb{R}^n$.

We also saw in Example \ref{eg:fundom} that the set $D=\{\lambda\in V \ | \ (\lambda,\alpha)\geq0 \ \mathrm{for\ all}\ \alpha\in\Delta\}$ is equal to $\{(x_1, x_2, \ldots, x_n)\in\mathbb{R}^n$ $|$ $x_i\geq x_{i+1}$ for all $i\}$ and so $v_0\in D$. Therefore the hypotheses of Theorem \ref{thm:Casselman} are satisfied and some quick calculations (for example, $(\epsilon_j - \epsilon_{j+1}, v_0) = 1-1=0$ when $j<k$) show that the set $\delta = \{\alpha\in\Delta \ | \ (\alpha, v_0) \neq 0\}$ is equal to the single element set $\{\epsilon_k - \epsilon_{k+1}\}$.

The $0$-faces can be found by first identifying all of the $\delta$-connected subsets of $\Delta$ of cardinality 0. The empty set is clearly the only set that satisfies this and so by Theorem \ref{thm:Casselman} and the fact that $W_\emptyset = 1$, there is a single $W$-representative of the 0-faces given by $\{W_\emptyset\cdot v_0\} = \{v_0\}$. Hence, every 0-face is of the form $ w\cdot\{v_0\}$ for some $w\in W$ which gives us the $\binom{n}{k}$ vectors described in (i).

Showing (ii) is only different in that there are more possibilities for $\delta$-connected subsets of $\Delta$ of cardinality $i>0$. There are two key ideas that help us finish the proof. First of all, it follows from the definition that every $\delta$-connected set $\kappa$ must contain at least one element of $\delta$ and since $\delta$ is a single element set, $\kappa$ must contain $\{\epsilon_k - \epsilon_{k+1}\}$. Second of all, if we consider the subgraph $\Gamma_\kappa$ of the Coxeter graph $\Gamma_W$ created by only using the vertices corresponding to elements of $\kappa$ and the edges between them, then $\Gamma_\kappa$ must be connected.

Using these two ideas and looking at $\Gamma_W$ shown in Figure \ref{fig:CgraphA} above, we see that any $\delta$-connected subset $\kappa$ must correspond to a set of generators $I = \{s_j,s_{j+1},\ldots, s_{j+i-1}\}$ such that $1\leq j\leq k\leq j+i-1\leq n-1$ as desired. For every set $I$ there exists a $W$-representative given by the convex hull of $\{W_I \cdot v_0\}$. 

The size of this orbit can be found by first realizing that every face in this orbit is uniquely defined by the set of vertices of which the face is the convex hull and so we just need to count the number of unique sets of vertices we can get by acting on the set of vertices in the representative, $\{W_I \cdot v_0\}$. Notice that in the set $\{W_I \cdot v_0\}$, the first $j-1$ coordinates of every vertex are 1's, the next $i+1$ coordinates are the ones being permuted, and the remaining coordinates of every vertex are 0's. If we act on $\{W_I \cdot v_0\}$ by $W$, there are $\binom{n}{i+1}$ possible locations for the coordinates that are being permuted and $\binom{n-i-1}{j-1}$ for the 1's.
\end{proof}

\section{CW complexes and Cellular Homology}
\label{s:CW}

In this section we show how to naturally construct a CW complex from a hypersimplex in a way that generalizes to any convex polytope. The action of $S_n$ on the hypersimplex will extend to this CW complex and then will induce an action onto the resulting homology groups that arise from any subcomplex. Identifying the corresponding characters later on is one of the main goals of this thesis.

We start by recalling several definitions.

\begin{defn}
\
\begin{enumerate}
\item [(i)] An \textbf{$n$-cell}, $e = e^n$ is a homeomorphic copy of the open $n$-disk $D^n - S^{n-1}$, where $D^n$ is the closed unit ball in Euclidean $n$-space and $S^{n-1}$ is its boundary, the unit $(n-1)$-sphere. We call $e$ a \textbf{cell} if it is an $n$-cell for some $n$ and define dim$(e)=n$.
\item [(ii)] If a topological space $X$ is a disjoint union of cells $X=\bigcup\{e \ | \ e\in E\}$, then for each $k\geq 0 $, we define the \textbf{$k$-skeleton} $X^{(k)}$ of $X$ by $X^{(k)}=\bigcup\{e\in E \ | \ $dim$(e)\leq k \}$.
\end{enumerate}
\end{defn}

The basic idea behind CW complexes is to start with a discrete set $X^0$, whose points are regarded as $0$-cells, and then to inductively form $X^n$ from $X^{n-1}$ by attaching $n$-cells via attaching maps $\phi:S^{n-1}\rightarrow X^{n-1}$.

\begin{defn}
A \textbf{finite CW complex} is an ordered triple $(X,E,\Phi)$, where $X$ is a Hausdorff space, $E$ is a family of cells in $X$, and $\{\Phi_e \ | \ e\in E\}$ is a family of maps, such that
\begin{enumerate}
\item [(i)] $X=\bigcup\{e \ | \ e\in E\}$ is a disjoint union;
\item [(ii)] for each $k$-cell $e\in E$, the map $\Phi_e : D^k \longrightarrow e\cup X^{(k-1)}$ is a continuous map such that $\Phi_e(S^{k-1})\subseteq X^{(k-1)}$ and $\Phi_e|_{D^k-S^{k-1}}:D^k-S^{k-1} \longrightarrow e$ is a homeomorphism.
\end{enumerate}
If the maps $\Phi_e$ are all homeomorphisms, the CW complex is called \textbf{regular}. In this thesis, we will only consider CW complexes that are regular.
\end{defn}

\begin{defn}
\label{def:subcomplex}
A \textbf{subcomplex} of the CW complex $(X,E,\Phi)$ is a triple $(|E'|,E',\Phi')$,  where $E'\subset E$, 
$|E'|:=\bigcup\{e\ | \ e\in E'\}\subset X,$
$\Phi'=\{\Phi_e \ | \ e\in E'\}$, and Im $\Phi_e \subset |E'|$ for every $e\in E'$. 
\end{defn}

\begin{thm}
Let $K$ be the hypersimplex $J(n,k)$ regarded as a subspace of $\mathbb{R}^n$, and let $E$ be the union of the following two sets:
\begin{enumerate}
\item [\rm(i)] the set of vertices of $J(n,k)$;
\item [\rm(ii)] the set of interiors of all $i$-faces of $J(n,k)$ for all $0<i<n$.
\end{enumerate}
Then $(K,E,\Phi)$ is a regular CW complex, where the maps $\Phi_e$ are the natural identifications.
\end{thm}

\begin{proof}
It is a standard result that the faces of a convex polytope form a regular CW complex; more details can be found in \cite[\S 1.3]{Forman04}.
\end{proof}

Later on it will be necessary to compute the homology groups that arise from subcomplexes of the CW complex corresponding to $J(n,k)$. 
Cellular homology is a convenient theory for doing exactly that.
We do not recall the full definition here, but instead direct the reader to \cite[\S 2.2]{Hatcher} and \cite[Proposition 5.3.10]{Geoghegan} for more details.

\begin{defn}
A \textbf{chain complex} is a sequence of abelian groups or modules $(C_i)_{i\in \mathbb{Z}}$, connected by homomorphisms (called boundary operators) $\partial_i : C_i\rightarrow C_{i-1}$, such that the composition of any two consecutive maps is zero: $\partial_i \circ \partial_{i+1} = 0$ for all $i$.
A \textbf{chain map} between two chain complexes $(C_i)_{i\in \mathbb{Z}}$ and $(C'_i)_{i\in \mathbb{Z}}$ is a sequence $(f_i)_{i\in \mathbb{Z}}$ of module homomorphisms $f_i:C_i\rightarrow C'_i$ for each $i$ such that $f_{i-1}\circ\partial_i = \partial'_{i}\circ f_i$.
\end{defn}

Let $X$ be a regular CW complex.
The basic idea of cellular homology is to introduce a chain complex $(C_i)_{i\in \mathbb{Z}}$, called the cellular chain complex of $X$.
The groups $C_i$ are all free abelian and have basis in one-to-one correspondence with the $i$-cells of $X$. The maps $\partial_i$ have the form 
$$\partial_i(e_\tau) = \displaystyle\sum_{\sigma\in X^{(i-1)}}[\tau:\sigma]e_\sigma$$ 
where $[\tau:\sigma] = 0$ if $\sigma$ is not a face of $\tau$ and $[\tau:\sigma] = \pm 1$ if $\sigma$ is a face of $\tau$, dependent on an orientation of the faces. The homology groups $H_i(X) =$ ker$(\partial_i) / $Im$(\partial_{i+1})$ are called the cellular homology groups and are equivalent to the homology groups obtained using singular homology. 
It should also be noted that we can extend $X$ by considering $\emptyset$ as a unique cell of dimension $-1$. The resulting cellular chain complex then leads to the reduced homology groups which are denoted by $\widetilde{H}_i(X)$.

It turns out that many CW complexes that arise in combinatorics are homotopy equivalent to a wedge of spheres of the same dimension. The following well-known proposition will allow us to state which $S_n$-invariant subcomplexes of the CW complex corresponding to $J(n,k)$ have this property.

\begin{prop}
\label{prop:wedge}
If $X$ is a topological space that has a CW decomposition consisting of exactly one 0-cell and $k$ $d$-cells such that $d>0$, then $X$ is homotopy equivalent to a wedge of $k$ $d$-spheres.
\end{prop}
\begin{proof}
One proof of this can be seen in \cite[Example 1.2]{Forman02}.
\end{proof}

\section{Discrete Morse Theory}

Even though cellular homology simplifies finding the homology groups by some amount, it would make things much easier if our CW complex always had the property that no two of its cells were in adjacent dimensions. 
In general this does not happen and we can not just modify a CW complex and expect it to produce the same homology groups. However the techniques introduced in this section, which were invented by Forman \cite{Forman02}, will give us a CW complex that is homotopic to ours and that have this desired property in most of the cases that we focus on later in the thesis.

\begin{defn}
Let $K$ be a finite regular CW complex. A \textbf{discrete vector field} on $K$ is a collection of pairs of cells $(K_1, K_2)$ such that
\begin{enumerate}
\item [(i)] $K_1$ is a face of $K_2$ of codimension 1 and
\item [(ii)] every cell of $K$ lies in at most one such pair.
\end{enumerate}
We call a cell of $K$ \textbf{matched} if it lies in one of the above pairs and \textbf{unmatched} otherwise.
\end{defn}

\begin{defn}
If $V$ is a discrete vector field on a regular CW complex $K$, a \textbf{$V$-path} is a sequence of cells 
\begin{center}
$a_0, b_0, a_1, b_1, a_2,\ldots,b_r,a_{r+1}$
\end{center}
such that for each $i=0,\ldots, r$, $a_i$ and $a_{i+1}$ are each a codimension 1 face of $b_i$, each of the pairs $(a_i, b_i)$ belongs to $V$ (hence $a_i$ is matched with $b_i$), and $a_i \neq a_{i+1}$ for all $0\leq i \leq r$. If $r\geq 0$, we call the $V$-path \textbf{nontrivial} and if $a_0 = a_{r+1}$, we call the $V$-path \textbf{closed}.
\end{defn}

Note that all of the faces $a_i$ in the sequence above have the same dimension, $p$ say, and all of the faces $b_i$ have dimension $p+1$.

\begin{rem}
\label{rem:hasse}
Let $V$ be a discrete vector field on a regular CW complex $K$. Consider the set of cells of $K$ together with the empty cell $\emptyset$, which we consider as a cell of dimension $-1$. This gives us a partially ordered set ordered under inclusion. We can create a directed Hasse diagram, $H(V)$, from this by pointing all of the edges towards the larger cell and then reversing the direction of any edge in which the smaller cell is matched with the larger cell.
\end{rem}

\begin{eg}
\label{eg:Hasse0}
Consider the hypersimplex $J(3,1)$. This is just a $2$-dimensional simplex equal to the convex hull of the points $(1,0,0), (0,1,0),$ and $(0,0,1)$ in $\mathbb{R}^3$. Label the vertices by $A, B,$ and $C$, label the edges as $AB, AC,$ and $AB$ where the edge $AB$ is the edge joining $A$ and $B$ and likewise for the other two, and label the unique 2-face of the simplex by $ABC$ as shown on the left of Figure \ref{fig:Hasse0} below.

Let $V_1$ be the discrete vector field made up of the following pairs of faces:
\begin{center}
$(A, AB), (B, BC),$ and $(C, AC).$
\end{center}
The corresponding directed Hasse diagram $H(V_1)$ can also be seen in Figure \ref{fig:Hasse0}. Note that this is only a partial matching since the faces $ABC$ and $\emptyset$ are unmatched. The $V_1$-path $A, AB, B$ is a nontrivial $V_1$-path, but is not closed while the $V_1$-path $A, AB, B, BC, C, AC, A$ is a nontrivial closed $V_1$-path.
\end{eg}

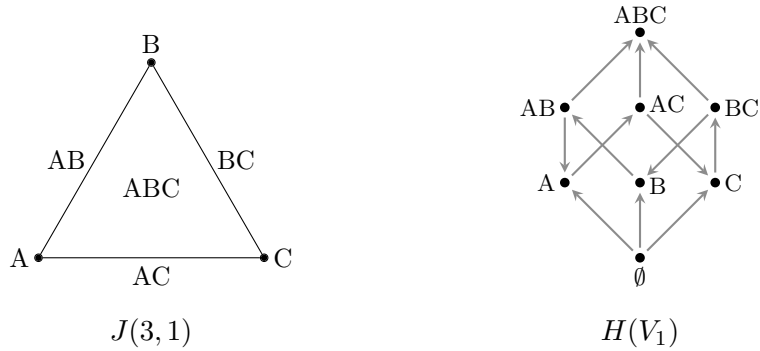
\begin{figure}[htbp]
\begin{center}
\begin{tikzpicture}[scale=1]

\def \tscale{3}

\small

\node (A) at (0,0) [circle, draw, inner sep=1pt] {};
\node (B) at (\tscale* .5, \tscale * 3^.5/2) [circle, draw, inner sep=1pt] {};
\node (C) at (\tscale,0) [circle, draw, inner sep=1pt] {};
\node at (A) [left] {A};
\node at (B) [above] {B};
\node at (C) [right] {C};

\node (D) at (\tscale* .5, \tscale * 3^.5/4 * .9) [below] {ABC};

\fill [black] (A) circle (1.2pt);
\fill [black] (B) circle (1.2pt);
\fill [black] (C) circle (1.2pt);

\draw (A) -- node [left] {AB} (B);
\draw (B) -- node [right] {BC} (C);
\draw (A) -- node [below] {AC} (C);

\normalsize

\node at (\tscale/2,-1) {$J(3,1)$};

\end{tikzpicture}
\hspace{1 in}
\begin{tikzpicture}[scale=1, >=stealth, thick, shorten >=2pt, shorten <=2pt]
\tikzset{shade t/.style ={gray!90}}
\tikzset{nodecirc/.style ={circle, draw, inner sep=1pt, fill=black}}
\node (e) at (0,0) [nodecirc] {};
\node (A) at (-1,1) [nodecirc] {};
\node (B) at (0,1) [nodecirc] {};
\node (C) at (1,1) [nodecirc] {};
\node (AB) at (-1,2) [nodecirc] {};
\node (AC) at (0,2) [nodecirc] {};
\node (BC) at (1,2) [nodecirc] {};
\node (ABC) at (0,3) [nodecirc] {};

\footnotesize
\node at (e) [below] {$\emptyset$};
\node at (A) [left] {A};
\node at (0,.95) [right] {B};
\node at (C) [right] {C};
\node at (AB) [left] {AB};
\node at (0,2.05) [right] {AC};
\node at (BC) [right] {BC};
\node at(ABC) [above] {ABC};
\normalsize

\draw [->] [shade t] (e) -- (A);
\draw [->] [shade t] (e) -- (B);
\draw [->] [shade t] (e) -- (C);
\draw [->] [shade t] (A) -- (AC);
\draw [<-] [shade t] (A) -- (AB);
\draw [->] [shade t] (B) -- (AB);
\draw [<-] [shade t] (B) -- (BC);
\draw [<-] [shade t] (C) -- (AC);
\draw [->] [shade t] (C) -- (BC);
\draw [->] [shade t] (AC) -- (ABC);
\draw [->] [shade t] (AB) -- (ABC);
\draw [->] [shade t] (BC) -- (ABC);

\node at (0,-1) {$H(V_1)$};

\end{tikzpicture}
\caption{Example of a Hasse diagram}
\label{fig:Hasse0}
\end{center}
\end{figure}

\begin{defn}
Consider the directed Hasse diagram $H(V)$ described in Remark \ref{rem:hasse}. If $H(V)$ has no directed cycles then we say that $V$ is an \textbf{acyclic matching} of the Hasse diagram of $K$. We call such a matching a \textbf{partial} matching if not every cell is paired with another and say it is a \textbf{complete} matching otherwise.
\end{defn}

\begin{thm}[\textbf{Forman}]
\label{thm:Forman}
Let $V$ be a discrete vector field on a regular CW complex $K$.
\begin{enumerate}
\item [\rm(i)] There are no nontrivial closed $V$-paths if and only if $V$ is an acyclic matching of the Hasse diagram of $K$.
\item [\rm(ii)] Suppose that $V$ is an acyclic partial matching of the Hasse diagram of $K$ in which the empty set is unpaired. Let $u_p$ denote the number of unpaired $p$-cells. Then $K$ is homotopic to a CW complex with exactly $u_p$ cells of dimension $p$ for each $p\geq 0$.
\end{enumerate}
\end{thm}

\begin{proof}
Part (i) is \cite[Theorem 6.2]{Forman02} and part (ii) is \cite[Theorem 6.3]{Forman02}.
\end{proof}

Finding a complete acyclic matching for the CW complex associated with the hypersimplex will be the focus of Section \ref{s:Construction}. Theorem \ref{thm:Forman} (i) makes it much easier to show that this matching will be acyclic and this will be the focus of Section \ref{s:acyclicproof}.

\begin{eg}
Consider the hypersimplex $J(3,1)$ and let $V_1$ be the discrete vector field as in Example \ref{eg:Hasse0}. It can be seen by using brute force that this is not an acyclic partial matching because there exist several directed cycles, for example, there is one that follows the path $A,AC,C,BC,B,AB,A$. We could also prove this by using Theorem \ref{thm:Forman} (i) since $A$, $AB$, $B$, $BC$, $C$, $AC$, $A$ is a nontrivial closed $V_1$-path. Notice that the second path is the same as the first except in reverse order. This is because a directed cycle follows the arrows, while a $V$-path goes against the arrows by definition.

On the other hand let $V_2$ be the discrete vector field made up of the following pairs of faces:
\begin{center}
$(\emptyset, A), (B, AB), (C, AC),$ and $(BC, ABC).$
\end{center}
This does turn out to be an acyclic matching and in this case it is even a complete matching. We can see that this is acyclic by once again checking for any cycles in the diagram in Figure \ref{fig:Hasse} through brute force. Proving this otherwise takes a little bit more work, as will be seen in Section \ref{s:acyclicproof}.

\end{eg}

\begin{figure}[htbp]
\begin{center}
\begin{tikzpicture}[scale=1, >=stealth, thick, shorten >=2pt, shorten <=2pt]
\tikzset{shade t/.style ={gray!90}}
\tikzset{nodecirc/.style ={circle, draw, inner sep=1pt, fill=black}}
\node (e) at (0,0) [nodecirc] {};
\node (A) at (-1,1) [nodecirc] {};
\node (B) at (0,1) [nodecirc] {};
\node (C) at (1,1) [nodecirc] {};
\node (AB) at (-1,2) [nodecirc] {};
\node (AC) at (0,2) [nodecirc] {};
\node (BC) at (1,2) [nodecirc] {};
\node (ABC) at (0,3) [nodecirc] {};

\footnotesize
\node at (e) [below] {$\emptyset$};
\node at (A) [left] {A};
\node at (0,.95) [right] {B};
\node at (C) [right] {C};
\node at (AB) [left] {AB};
\node at (0,2.05) [right] {AC};
\node at (BC) [right] {BC};
\node at(ABC) [above] {ABC};
\normalsize

\draw [->] [shade t] (e) -- (A);
\draw [->] [shade t] (e) -- (B);
\draw [->] [shade t] (e) -- (C);
\draw [->] [shade t] (A) -- (AC);
\draw [<-] [shade t] (A) -- (AB);
\draw [->] [shade t] (B) -- (AB);
\draw [<-] [shade t] (B) -- (BC);
\draw [<-] [shade t] (C) -- (AC);
\draw [->] [shade t] (C) -- (BC);
\draw [->] [shade t] (AC) -- (ABC);
\draw [->] [shade t] (AB) -- (ABC);
\draw [->] [shade t] (BC) -- (ABC);

\node at (0,-1) {$H(V_1)$};

\end{tikzpicture}
\hspace{1.3 in}
\begin{tikzpicture}[scale=1, >=stealth, thick, shorten >=2pt, shorten <=2pt]
\tikzset{shade t/.style ={gray!90}}
\tikzset{nodecirc/.style ={circle, draw, inner sep=1pt, fill=black}}
\node (e) at (0,0) [nodecirc] {};
\node (A) at (-1,1) [nodecirc] {};
\node (B) at (0,1) [nodecirc] {};
\node (C) at (1,1) [nodecirc] {};
\node (AB) at (-1,2) [nodecirc] {};
\node (AC) at (0,2) [nodecirc] {};
\node (BC) at (1,2) [nodecirc] {};
\node (ABC) at (0,3) [nodecirc] {};

\footnotesize
\node at (e) [below] {$\emptyset$};
\node at (A) [left] {A};
\node at (0,.95) [right] {B};
\node at (C) [right] {C};
\node at (AB) [left] {AB};
\node at (0,2.05) [right] {AC};
\node at (BC) [right] {BC};
\node at(ABC) [above] {ABC};
\normalsize

\draw [<-] [shade t] (e) -- (A);
\draw [->] [shade t] (e) -- (B);
\draw [->] [shade t] (e) -- (C);
\draw [->] [shade t] (A) -- (AC);
\draw [->] [shade t] (A) -- (AB);
\draw [<-] [shade t] (B) -- (AB);
\draw [->] [shade t] (B) -- (BC);
\draw [<-] [shade t] (C) -- (AC);
\draw [->] [shade t] (C) -- (BC);
\draw [->] [shade t] (AC) -- (ABC);
\draw [->] [shade t] (AB) -- (ABC);
\draw [<-] [shade t] (BC) -- (ABC);

\node at (0,-1) {$H(V_2)$};

\end{tikzpicture}
\vspace{.3 in}
\caption{Example of a cyclic and an acyclic matching}
\label{fig:Hasse}
\end{center}
\end{figure}
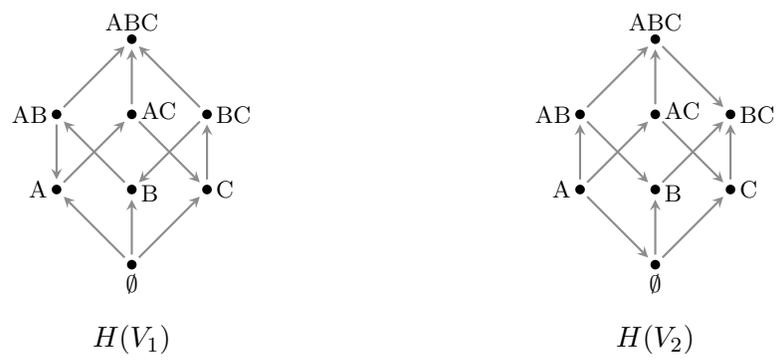


\chapter{A Family of Complete Acyclic Matchings for $J(n,k)$}
\label{chap:matching} 

Consider a hypersimplex $J(n,k)$ and the corresponding CW complex $K$ associated with the faces of $J(n,k)$ as described in Chapter \ref{c:intro}. In Section \ref{s:Construction} we will construct a family of discrete vector fields on $K$ or in other words describe how to pair the faces of $J(n,k)$ together in a ``nice" way. 
In Section \ref{s:matchingprop} we will then show that these matchings are complete matchings.
The purpose of Section \ref{s:acyclicproof} is to prove that this matching is acyclic so that we can use Theorem \ref{thm:Forman} to get information about the corresponding homology groups of the subcomplexes of $K$. 
We finish the chapter by giving more insight into the matchings and introducing further notation.

\section{Construction of Matchings}
\label{s:Construction}

We start by introducing a more convenient way to refer to the faces of $J(n,k)$.

\begin{defn}
Let $S$ be a sequence of length $n$ made up of 0's, 1's, and~*'s. Define $F(S)$ to be the face of $J(n,k)$ equal to the convex hull of the set of vertices of this polytope whose coordinates written out in a sequence differ from $S$ only where $S$ has a *. Define $S(0)$ and $S(1)$ to be the number of 0's and 1's in the sequence $S$ respectively.
\end{defn}

\begin{eg}
Let $S$ = 11***00 with $n = 7$ and $k = 3$. Then $F(S)$ is equal to the convex hull of $(1,1,1,0,0,0,0)$, $(1,1,0,1,0,0,0)$, and $(1,1,0,0,1,0,0)$.
\end{eg}

At this point we will define a family of matchings dependent on two integers $m_0$ and $m_1$, where $0\leq m_0\leq n-k-1$ and $1\leq m_1\leq k-1$. It will not be completely clear what their purpose is until we introduce some diagrams in Section \ref{s:subdiagrams}, but essentially they will allow us to make adjustments to the matching relative to the faces $F(S)$ where $S(0)=m_0$ and $S(1)=m_1$.

We start by matching the vertex $v_0 = (1,\ldots,1,0,\ldots,0)$ with $\emptyset$. 
Next we fix $m_0$ and $m_1$ and then match $F(S)$ with $F(S')$ where $S'$ is obtained from $S$ by doing one of the following replacements to $S$:

\begin{enumerate}
  \item If $S(1)$ $\leq k-1$ and there is a 1 to the right of the rightmost * and one of the following is true:
  \begin{enumerate}
    \item $S(1) \neq m_1$;
    \item $S(1) = m_1$ and $S(0)>m_0$; 
    \item $S(1) = m_1$, $S(0)=m_0$, and there is not a 0 to the left of the leftmost *;
  \end{enumerate}
then replace the rightmost 1 with a *.

  \item If $S(1)$ $\leq k-2$ and there is no 1 right of the rightmost *, and one of the following is true:
  \begin{enumerate}
    \item $S(1) \neq m_1 - 1$;
    \item $S(1) = m_1 - 1$ and $S(0)>m_0$;
    \item $S(1) = m_1 - 1$, $S(0)=m_0$, and there is not a 0 to the left of the leftmost *;
  \end{enumerate}
then replace the rightmost * with a 1. 
  \item If $S(1) = k-1$, $S(0)\leq n-k-1$, there is no 1 to the right of the rightmost *, and there is a 0 to the left of the leftmost *, then replace the leftmost 0 with a *.
  \item If $S(1) = k-1$, $S(0)\leq n-k-2$, there is no 1 to the right of the rightmost *, and there is no 0 to the left of the leftmost *, then replace the leftmost * with a 0. 
  \item If $S(1) = m_1$, $S(0)\leq m_0$, there is a 1 to the right of the rightmost *, and there is a 0 to the left of the leftmost *, then replace the leftmost 0 with a *.
  \item If $S(1) = m_1$, $S(0)<m_0$, there is a 1 to the right of the rightmost *, and there is no 0 to the left of the leftmost *, then replace the leftmost * with a 0. 
  \item If $S(1) = m_1 - 1$, $S(0)\leq m_0$, there is no 1 to the right of the rightmost *, and there is a 0 to the left of the leftmost *, then replace the leftmost 0 with a *.
  \item If $S(1) = m_1 - 1$, $S(0)<m_0$, there is no 1 to the right of the rightmost *, and there is no 0 to the left of the leftmost *, then replace the leftmost * with a 0. 
  \item If $S(1)=k$, $S(0)=n-k$, and $F(S)\neq v_0= (1,\ldots,1,0,\ldots,0)$, then replace the leftmost 0 and the rightmost 1 with a *.
  \item If $S(1)=k-1$, $S(0)=n-k-1$, there is no 1 to the right of rightmost *, and there is no 0 to the left of the leftmost *, then replace the leftmost * with a 0 and the remaining * with a 1.
 
\end{enumerate}

\begin{eg}
Let $n=8$, $k=3$, and consider the matching with $m_0=2$ and $m_1=1$, then:
\begin{center}
$F$(0100*0*1) and $F$(0100*0**) are matched by (1)(a) and (2)(a);
\\
$F$(0*00*0*1) and $F$(0*00*0**) are matched by (1)(b) and (2)(b);
\\
$F$(***0*0*1) and $F$(***0*0**) are matched by (1)(c) and (2)(c);
\\
$F$(0100*1*0) and $F$(*100*1*0) are matched by (3) and (4);
\\
$F$(0****0*1) and $F$(*****0*1) are matched by (5) and (6);
\\
$F$(0****0**) and $F$(*****0**) are matched by (7) and (8);
\\
$F$(10010100) and $F$(1*010*00) are matched by (9) and (10).
\end{center}
\end{eg}

We will later on refer back to these rules to describe faces. A face $F(S)$ of type 1 is one such that $S$ satisfies the conditions of (1)(a), (b), or (c), a face of type 3 is one such that $S$ satisfies the conditions of (3), and so on. 

\begin{rem}
\label{rem:matching}
There are a few things here worth pointing out. 
First of all, if $F(S)$ is a face of type 1--8 or 10, then $S(0)<n-k$ and $S(1)<k$ and hence $S$ contains at least two *'s. Therefore $F(S)$ is the convex hull of more than one point and cannot be a vertex.
Also notice that if we are given a face $F(S)$ then (1), (3), (5), and (7) replace a 0 or 1 in $S$ with a * and hence match this face to one of higher dimension. On the other hand (2), (4), (6), and (8) replace a * in $S$ with a 0 or 1 which matches the given face to one of lower dimension. Similarly (9) matches a vertex with an edge, which is a face of higher dimension, and (10) matches an edge to a vertex.

\end{rem}

\begin{lem}
\label{l:wd}
For any $m_0$ and $m_1$, the ten rules above partition the set of faces of $J(n,k)$, other than $\{v_0\}$, into subsets of faces of type $i$ where $1\leq i\leq10$.
\end{lem}
\begin{proof}
Let $F(S)$ be a face of $J(n,k)$ other than $\{v_0\}$.
Notice that if $S$ contains no *'s then it only satisfies the conditions of (9). If we assume that $S$ has a 1 to the right of the rightmost * then it either satisfies (1), (5), or (6). 
If $S(1)\neq m_1$ then $S$ must satisfy (1)(a). If $S(1)=m_1$ and $S(0)>m_0$ then $S$ must satisfy (1)(b). If $S(1)= m_1$, $S(0)\leq m_0$, and there is not a 0 to the left of the leftmost * then $S$ satisfies (1)(c) when $S(0)=m_0$ and (6) otherwise. If $S(1)= m_1$, $S(0)\leq m_0$, and there is a 0 to the left of the leftmost * then $S$ satisfies (5).

Next we assume that $S$ does not have a 1 to the right of the rightmost * and that $S(1)=k-1$ and so it either satisfies (3), (4), or (10). If there
is no 0 to the left of the leftmost * in $S$ then $S$ satisfies (4) when $S(0)\leq n-k-2$ and satisfies (10) when $S(0)=n-k-1$. If there is a 0 to the left of the leftmost * in $S$ then $S$ only satisfies (3).

Finally, we assume that $S$ does not have a 1 to the right of the rightmost * and that $S(1)<k-1$ and so it either satisfies (2), (7), or (8) since $m_1-1 < k-1$. If $S(1)\neq m_1-1$ then $S$ must satisfy (2)(a). If $S(1)= m_1-1$ and $S(0)>m_0$ then $S$ must satisfy (2)(b). If $S(1)= m_1-1$, $S(0)\leq m_0$, and there is not a 0 to the left of the leftmost * then $S$ satisfies (2)(c) when $S(0)=m_0$ and (8) otherwise. If $S(1)= m_1-1$, $S(0)\leq m_0$, and there is a 0 to the left of the leftmost * then $S$ satisfies (7).
\end{proof}

\section{Properties of the Matchings}
\label{s:matchingprop}

Now that we have a description of the matchings that we will use, we need to show that they satisfy the conditions of being a complete acyclic matching. If we let $V$ be a collection of pairs of faces matched together by the rules in Section \ref{s:Construction}, then showing that $V$ is complete amounts to showing that every face appears in exactly one pair in $V$. Showing that $V$ is acyclic will be shown in Section \ref{s:acyclicproof}. For the remainder of this section and the next we will assume that we have a fixed matching, or in other words, that $m_0$ and $m_1$ are fixed.

\begin{defn}
Let $i,j\in\mathbb{Z}$ such that $1\leq i,j\leq 10$ and suppose that rule $(i)$ and rule $(j)$ are two of the ten rules from Section \ref{s:Construction}. 
We say that rule $(i)$ and rule $(j)$ are \textbf{inverses} of each other if both of the following are true:
\begin{enumerate}
\item [(i)] If $F(S)$ is matched with $F(S')$ by $(i)$, then $F(S')$ is matched with $F(S)$ by $(j)$.
\item [(ii)] If $F(S')$ is matched with $F(S)$ by $(j)$ , then $F(S)$ is matched with $F(S')$ by $(i)$.
\end{enumerate}
\end{defn}

\begin{lem} Consider the ten rules from Section \ref{s:Construction}.
\label{l:inv1}
\begin{enumerate}
\item [\rm{(i)}] The rules (1)(a) and (2)(a) are inverses of each other.
\item [\rm{(ii)}] The rules (1)(b) and (2)(b) are inverses of each other.
\item [\rm{(iii)}] The rules (1)(c) and (2)(c) are inverses of each other.
\end{enumerate}
\end{lem}

\begin{proof}
First suppose that $S$ satisfies the conditions of (1)(a) and so $S(1)\neq m_1$, $S(1)\leq k-1$, and there is a 1 to the right of the rightmost * in $S$. Then by (1), $S'$ is obtained from $S$ by replacing the rightmost 1 with a *, so $S'(1) = S(1)-1 \neq m_1-1$, $S'(1) = S(1)-1 \leq k-2$, and there is no 1 to the right of the rightmost * in $S'$. Therefore $S'$ satisfies the conditions of (2)(a) and hence will be matched to the face $F(S'')$ where $S''$ is obtained from $S'$ by replacing the rightmost * with a 1. This gives us $S'' = S$ as desired.

Next suppose that $S$ satisfies the conditions of (1)(b) and so $S(1) = m_1$, $S(1)\leq k-1$, $S(0)>m_0$, and there is a 1 to the right of the rightmost * in $S$. Then by (1), $S'$ is obtained from $S$ by replacing the rightmost 1 with a *, so $S'(1) = S(1)-1 = m_1-1$, $S'(1) = S(1)-1 \leq k-2$, $S'(0) = S(0)>m_0$, and there is no 1 to the right of the rightmost * in $S'$. Therefore $S'$ satisfies the conditions of (2)(b) and hence will be matched to the face $F(S'')$ where $S''$ is obtained from $S'$ by replacing the rightmost * with a 1. This gives us $S'' = S$ as desired.

Next suppose that $S$ satisfies the conditions of (1)(c) and so $S(1) = m_1$, $S(1)\leq k-1$, $S(0) = m_0$, there is a 1 to the right of the rightmost * in $S$, and there is not a 0 to the left of the leftmost *. Then by (1), $S'$ is obtained from $S$ by replacing the rightmost 1 with a *, so $S'(1) = S(1)-1 = m_1-1$, $S'(1) = S(1)-1 \leq k-2$, $S'(0) = S(0) = m_0$, there is no 1 to the right of the rightmost * in $S'$, and there is not a 0 to the left of the leftmost *. Therefore $S'$ satisfies the conditions of (2)(c) and hence will be matched to the face $F(S'')$ where $S''$ is obtained from $S'$ by replacing the rightmost * with a 1. This gives us $S'' = S$ as desired.

What remains to be shown in this case is that if we have a face $F(S')$ of type 2 then there is some face $F(S)$ of type 1 that matches with it. Notice that $S'$ has no 1 to the right of the rightmost * and hence when we consider the sequence $S$ obtained by replacing the rightmost * in $S'$ with a 1, we see that $S$ satisfies the conditions of (1) and that $F(S)$ will be matched with $F(S')$ as desired.
\end{proof}

\begin{lem}
\label{l:inv3}
The rules (3) and (4) are inverses of each other.
\end{lem}
\begin{proof}
Suppose first that $S$ satisfies the conditions of (3) and so $S(1)=k-1$, $S(0)\leq n-k-1$, there is no 1 to the right of the rightmost *, and there is a 0 to the left of the leftmost *. Then by (3), $S'$ is obtained from $S$ by replacing the leftmost 0 with a *, so $S'(1) = S(1)=k-1$, $S'(0)= S(0)-1\leq n-k-2$, there is no 1 to the right of the rightmost *, and there is no 0 to the left of the leftmost *. Therefore $S'$ satisfies the conditions of (4) and hence will be matched to the face $F(S'')$ where $S''$ is obtained from $S'$ by replacing the leftmost * with a 0. This gives us $S'' = S$ as desired.

Now let $S$ satisfy the conditions of (4) and so $S(1)=k-1$, $S(0)\leq n-k-2$, there is no 1 to the right of the rightmost *, and there is no 0 to the left of the leftmost *. Then by (4), $S'$ is obtained from $S$ by replacing the leftmost * with a 0, so $S'(1) = S(1)=k-1$, $S'(0)= S(0)+1\leq n-k-1$, there is no 1 to the right of the rightmost *, and there is a 0 to the left of the leftmost *. Therefore $S'$ satisfies the conditions of (3) and hence will be matched to the face $F(S'')$ where $S''$ is obtained from $S'$ by replacing the leftmost 0 with a *. This gives us $S'' = S$ as desired.
\end{proof}

\begin{lem}
\label{l:inv5}
The rules (5) and (6) are inverses of each other and the rules (7) and (8) are inverses of each other.
\end{lem}
\begin{proof}
This proof is similar to that of Lemma \ref{l:inv3}.
\end{proof}

\begin{lem}
\label{l:inv9}
The rules (9) and (10) are inverses of each other.
\end{lem}
\begin{proof}
Suppose first that $S$ satisfies the conditions of (9) and so $S(1)=k$, $S(0)=n-k$, and $S\neq 1\ldots 10\ldots 0$. Then by (9), $S'$ is obtained from $S$ by replacing the leftmost 0 and the rightmost 1 with a *, so $S'(1)=S(1)-1=k-1$, $S'(0)=S(0)-1=n-k-1$, there is no 1 to the right of rightmost *, there is no 0 to the left of the leftmost *. Therefore $S'$ satisfies the conditions of (10) and hence will be matched to the face $F(S'')$ where $S''$ is obtained from $S'$ by replacing the leftmost * with a 0 and the remaining * with a 1. This gives us $S'' = S$ as desired.

Next let $S$ satisfy the conditions of (10) and so $S(1) = k-1$, $S(0) = n-k-1$, there is no 1 to the right of rightmost *, and there is no 0 to the left of the leftmost *. Then by (10), $S'$ is obtained from $S$ by replacing the leftmost * with a 0 and the rightmost * with a 1. 
If $F(S')=F(1\ldots 10\ldots 0)=\{v_0\}$ then that would mean the leftmost * in $S$ was to the right of the rightmost * in $S$ which is a contradiction. So $S'(1)=S(1)+1=k$, $S'(0)=S(0)+1=n-k$, and $F(S')\neq \{v_0\}$. Therefore $S'$ satisfies the conditions of (9) and hence will be matched to the face $F(S'')$ where $S''$ is obtained from $S'$ by replacing the leftmost 0 and the rightmost 1 with a *. This gives us $S''=S$ as desired.
\end{proof}

\begin{prop}
Fix $m_0$ and $m_1$ and let $V$ be the corresponding collection of pairs of matched faces. Every face appears in exactly one pair in $V$ and hence $V$ is a discrete vector field and furthermore, $V$ is a complete matching.
\end{prop}
\begin{proof}
Lemma \ref{l:wd} shows that every faces is matched with at least one other face and
Lemmas \ref{l:inv1}--\ref{l:inv9} show that every face is matched to at most one other face. Remark \ref{rem:matching} points out that for every pair of faces, one face is a codimension 1 face of the other and the assertions follow.
\end{proof}

\section{Proof that the Matchings are Acyclic}
\label{s:acyclicproof}

First recall that $S$ is a sequence of 0's, 1's, and *'s. For notational purposes it will be helpful to break up $S$ into subsequences of only 0's and 1's and subsequences of only *'s. From now let $f_i$ denote either a sequence of 0's and 1's or $\emptyset$ where $1\leq i\leq n$. Also let $f'_i$ be a subsequence of $f_i$, let $0 \cdots 0$ denote a sequence of 0's or possibly $\emptyset$, and let $1 \cdots 1$ denote a sequence of 1's or possibly~$\emptyset$. 

\begin{eg}
Suppose we have an edge given by $F(S)$. Since $F(S)$ is an edge, $S$ has exactly two *'s and therefore we can write $F(S)=F(f_1*f_2*f_3)$. Suppose also that $1\in f_3$, then in order to help show the location of the rightmost 1 in $f_3$ we could also write $f_3 = f'_310 \cdots 0$ since it is only possible for either a sequence of 0's or $\emptyset$ to be to the right of the rightmost 1 in $f_3$.
\end{eg}

Next, following the language of Forman \cite{Forman02} introduced in Chapter \ref{c:intro}, if $K$ is the CW complex formed by the faces of $J(n,k)$ then let $V$ be the discrete vector field on $K$ defined by the collection of pairs of matched faces. Recall that a $V$-path is a sequence of cells
\begin{center}
$a_0,b_0,a_1,b_1,a_2,\ldots,b_r,a_{r+1}$
\end{center}
such that for each $i = 0,\ldots,r,$ each of $a_i$ and $a_{i+1}$ is a codimension 1 face of $b_i$, each $(a_i,b_i)$ belongs to $V$ (hence $a_i$ is matched with $b_i$), and $a_i\neq a_{i+1}$ for all $0\leq i\leq r$. If $r\geq 0$, we call the $V$-path nontrivial, and if $a_0 = a_{r+1}$, we call the $V$-path closed.

\begin{defn}
\label{def:concat}
Given two $V$-paths $a_0,b_0,a_1,\ldots,b_r,a_{r+1}$ and $a'_0,b'_0,a'_1,\ldots,b'_s,a'_{s+1}$ such that $a_{r+1} = a'_0$, define their \emph{concatenation} to be the following $V$-path: 
\begin{center}
$a_0,b_0,a_1,\ldots,b_r,a'_0,b'_0,a'_1,\ldots,b'_s,a'_{s+1}$.
\end{center}
\end{defn}

We will now show in Lemmas \ref{lem:acyc1}--\ref{lem:acyc7} that there are no nontrivial, closed $V$-paths when $a_0$ is not a vertex and deal with the case when it is a vertex in Lemma \ref{lem:acyc8}.

\begin{lem}
\label{lem:acyc1}
Given a $V$-path $a_0,b_0,a_1$ where $a_0 = F(S) = F(f_1*\cdots f_i*f_{i+1})$ is of type 1, either $a_1 = F(S')$ is of type 1, 2, 5, 6, or 10 and going from $S$ to $S'$ replaces the rightmost 1 with a 0 or $a_1 = F(S')$ is of type 2, 3, 4, 7, 8, or 10, going from $S$ to $S'$ replaces the rightmost 1 with a *, and this new * is the rightmost * in $S'$.
\end{lem}
\begin{proof}
In this case $f_{i+1} = f'_{i+1}10 \cdots 0$ and so 
\begin{center}
$a_0 = F(f_1*\cdots f_i*f'_{i+1}10 \cdots 0$) and $b_0 = F(f_1*\cdots f_i*f'_{i+1}*0 \cdots 0)$.
\end{center}
In order to choose $a_1$, we can replace any * except the rightmost with 0 or 1 or replace the rightmost * with 0 since $a_1\neq a_0$. 
Either way, $S'(1) \leq S(1) \leq k-1$ and so $a_1$ cannot be of type 9.
First consider the case when we replace the rightmost * with 0 making 
\begin{center}
$a_1 = F(f_1*\cdots f_i*f'_{i+1}00 \cdots 0) = F(S')$.
\end{center}

If we assume that $a_1$ is of type 7 or 8 then $S(0)=S'(0)-1<m_0$ and $S(1)=S'(1)+1=m_1$ and hence $a_0$ would not be of type 1 which is a contradiction. Also, $a_1$ cannot be of type 3 or 4 since $S'(1) < S(1) \leq k-1$ and so $a_1$ must be of type 1, 2, 5, 6, or 10. Either way the result from $S$ to $S'$ was that we replaced the rightmost 1 with a 0.

Next consider the case when we replace any * except the rightmost with a 0 or 1, then we will write $f_1*\cdots f_i*f'_{i+1}$ as $f_1\cdots f'_{i+1}$ (since we do not know and it will not matter which * was replaced) and hence
\begin{center}
$a_1 = F(f_1\cdots f'_{i+1}*0 \cdots 0)$.
\end{center}

Notice that there is no 1 to the right of the rightmost * and so $a_1$ is of type 2, 3, 4, 7, 8, or 10. Either way the result from $S$ to $S'$ was that we replaced the rightmost 1 with a *.
\end{proof}

\begin{lem}
\label{lem:acyc2}
Given a $V$-path $a_0,b_0,a_1$ where $a_0 = F(S) = F(f_1*\cdots f_i*f_{i+1})$ is of type 3, either $a_1 = F(S')$ is of type 4 or 10 and going from $S$ to $S'$ does not change the rightmost * or $a_1 = F(S')$ is of type 1, 4, 6, or 10 and going from $S$ to $S'$ replaces the rightmost * with a 0.
\end{lem}
\begin{proof}
In this case $f_1 = 1 \cdots 10f'_1$, $f_{i+1} = 0\cdots0$, and $S(1) = k-1$ and so 
\begin{center}
$a_0 = F(1 \cdots 10f'_1*\cdots f_i*0\cdots0)$ and $b_0 = F(1 \cdots 1*f'_1*\cdots f_i*0\cdots0)$.
\end{center}
In order to choose $a_1 = F(S')$ we can replace any * except the leftmost only with a 0 since $S(1) = k-1$. 
No matter what is replaced, $S(1) = S'(1) = k-1$ and so $a_1$ cannot be of type 2, 7, 8, or 9.

First consider the case when any * but the leftmost or rightmost is replaced by a 0, then 
\begin{center}
$a_1 = F(1 \cdots 1*f'_1\cdots f_i*0\cdots0)$.
\end{center}
In this case there is no 1 to the right of the rightmost~* in $S'$ and so $a_1$ cannot be of type 1, 5, or 6. Similarly there is no 0 to the left of the leftmost~* in $S'$ and so $a_1$ cannot be of type 3. Therefore $a_1$ is of type 4 or 10.
Also, going from $S$ to $S'$ does not change the rightmost~*.

Next consider the case when the rightmost * is replaced by a 0 and so
\begin{center}
$a_1 = F(1 \cdots 1*f'_1*\cdots *f_i00\cdots0)$.
\end{center}
There is no 0 to the left of the leftmost * in $S'$ and so $a_1$ is not of type 3 or 5. Therefore $a_1$ can only be of type 1, 4, 6, or 10. Either way going from $S$ to $S'$ replaces the rightmost * with a 0.
\end{proof}

\begin{lem}
\label{lem:acyc3}
Given a $V$-path $a_0,b_0,a_1$ where $a_0 = F(S) = F(f_1*\cdots f_i*f_{i+1})$ is of type 5, $a_1 = F(S')$ is of type 1 or 6.
\end{lem}
\begin{proof} 
In this case $f_1 = 1\cdots 10f'_1$ and $f_{i+1} = f'_{i+1}10\cdots 0$ and so we have
\begin{center}
$a_0 = F(1\cdots 10f'_1*\cdots*f'_{i+1}10\cdots 0)$ and $b_0 = F(1\cdots 1*f'_1*\cdots*f'_{i+1}10\cdots 0)$.
\end{center}
There are again two cases for choosing $a_1$, but either way $S'(0) \leq S(0) \leq m_0 \leq n-k-1$ and so $a_1$ cannot be of type 9. 
For the first case, replace any * except the leftmost by a 0. Since there is a 1 to the right of the rightmost *, $a_1$ cannot be of type 2, 3, 4, 7, 8, or 10. Also since there is not a 0 to the left of the leftmost *, $a_1$ cannot be of type 5 and so $a_1$ can be of type 1 or 6.

On the other hand if we replace any * by a 1 we again have that there is a 1 to the right of the rightmost * and so $a_1$ cannot be of type 2, 3, 4, 7, 8, or 10. However this time $S'(1) = S(1)+1 = m_1+1$ and so $a_1$ cannot be of type 5 or 6 either. Therefore $a_1$ must be of type 1.
\end{proof}

\begin{lem}
\label{lem:acyc4}
Given a $V$-path $a_0,b_0,a_1$ where $a_0 = F(S) = F(f_1*\cdots f_i*f_{i+1})$ is of type 7, one of three things can happen:
\begin{enumerate}
\item [\rm{(i)}] $a_1 = F(S')$ is of type 1, 2, or 8 and going from $S$ to $S'$ replaces the rightmost * with a 0; 
\item [\rm{(ii)}] $a_1 = F(S')$ is of type 2, 3, 4, 8, or 10 and going from $S$ to $S$ leaves the rightmost * unchanged; or
\item [\rm{(iii)}] $a_1 = F(S')$ is of type 6 and going from $S$ to $S'$ replaces the rightmost * with a 1.
\end{enumerate}

\end{lem}
\begin{proof} 
In this case $f_1 = 1\cdots 10f'_1$ and $f_{i+1} = 0\cdots 0$ and so we have
\begin{center}
$a_0 = F(1\cdots 10f'_1*\cdots f_i*0\cdots 0)$ and $b_0 = F(1\cdots 1*f'_1*\cdots f_i*0\cdots 0)$.
\end{center}
This time choosing $a_1$ will be broken up into three cases, but in all of them $S'(0) \leq S(0) \leq m_0 \leq n-k-1$ and so $a_1$ cannot be of type 9. In the first case replace the rightmost * by 0 and so $S'(1) = S(1) = m_1-1 \leq k-2$. This means $a_1$ cannot be of type 3, 4, 5, 6, or 10. Also there is no 0 to the left of the leftmost * and so $a_1$ cannot be of type 7 either. Therefore $a_1$ can be of type 1, 2, or 8.

For the second case either replace any * but the leftmost or rightmost by a 0 or replace any * but the rightmost by a 1. Either way going from $S$ to $S'$ leaves the rightmost * unchanged. Next note that there is no 1 to the right of the rightmost * and hence $a_1$ cannot be of type 1, 5, or 6. Also if any * but the leftmost or rightmost is replaced by a 0 then there is no 0 to the left of the leftmost *, whereas if any * but the rightmost is replaced by a 1 then $S'(1) = m_1$, so in either case $a_1$ cannot be of type 7. Therefore $a_1$ can be of type 2, 3, 4, 8, or 10.

Last of all, we can choose $a_1$ by replacing the rightmost * by a 1.
In this case $S'(1) = S(1)+1 = m_1$ and $S'(0)=S(1)-1<m_0$ so $a_1$ cannot be of type 1, 7, or 8. 
Since there is a 1 to the right of the rightmost * in $S'$, $a_1$ cannot be of type 2, 3, 4, or 10.
Since there is not a 0 to the left of the leftmost * in $S'$, $a_1$ cannot be of type 5 either. 
Therefore $a_1$ can only be of type 6.
\end{proof}

\begin{lem}
\label{lem:acyc5}
Let $a_0,b_0,a_1,\ldots,b_r,a_{r+1}$ be a $V$-path such that $a_0 = F(S)$ is of type 1 and $a_j = F(S^{(j)})$ for each $j\geq0$. If, while going from $S$ to $S^{(j)}$ for any $j\geq1$, the rightmost 1 of $S$ is replaced with a 0, then $a_0\neq a_{r+1}$.
\end{lem}
\begin{proof}
First notice that in order to get from $S^{(j)}$ to $S^{(j+1)}$ for any $j\geq0$ either a 0 or 1 in $S^{(j)}$ is changed to a * based on the ten rules for matching faces and then a * is replaced by a 0 or 1. 
Furthermore, every $a_i$ has the same dimension as $a_0$ by definition of a $V$-path and since $a_0$ is a face of type 1 and hence not a vertex, none of the $a_i$ are vertices. Therefore we will never utilize rule (9).

Recall that a face of type 1 looks like
\begin{center}
$a_0 = F(f_1*\cdots f_i*f'_{i+1}10\cdots 0)$.
\end{center}
If there are any 0's to the right of the rightmost 1 in $S$ then they remain unchanged while going from $S^{(j)}$ to $S^{(j+1)}$ for any $j\geq0$, since the only rules other than (9) that involve replacing a 0 with a * require that the 0 be to the left of the leftmost *. 

If the rightmost 1 of $S$ is replaced with a 0 while going from $S$ to $S^{(j)}$ for any $j\geq1$, then there are no *'s to the right of that 0 in $S^{(j)}$. Therefore that 0 will remain fixed while going from $S^{(j)}$ to $S^{(r+1)}$ because again, the only rules other than (9) that involve replacing a 0 with a * require that the 0 be to the left of the leftmost~* and so $a_0\neq a_{r+1}$.
\end{proof}

\begin{lem}
\label{lem:acyc6}
There are no nontrivial, closed $V$-paths $a_0,b_0,a_1,\ldots,b_r,a_{r+1}$ such that $a_0 = F(S)$ is of type 1.
\end{lem}
\begin{proof}
Suppose there is a nontrivial, closed $V$-path $a_0,b_0,a_1,\ldots,b_r,a_{r+1}$ with $a_0$ being of type 1. Then $a_{r+1} = a_0$ and hence is also of type 1. Note that if for any $j$, $a_j$ is of type 2, 4, 6, 8, or 10 then the $V$-path is not closed since faces of those types are paired with faces of lower dimension.

By Lemmas \ref{lem:acyc1}--\ref{lem:acyc4}, this $V$-path must look like one of the following: 
\begin{enumerate}
\item [(i)] $a_0,b_0,a_1$ where $a_1$ is of type 1,
\item [(ii)] $a_0,b_0,a_1,b_1,a_2$ where $a_1$ is of type 3 and $a_2$ is of type 1,
\item [(iii)] $a_0,b_0,a_1,b_1,a_2$ where $a_1$ is of type 5 and $a_2$ is of type 1,
\item [(iv)] $a_0,b_0,a_1,b_1,a_2$ where $a_1$ is of type 7 and $a_2$ is of type 1,
\item [(v)] $a_0,b_0,a_1,b_1,a_2,b_2,a_3$ where $a_1$ is of type 7, $a_2$ is of type 3, and $a_3$ is of type 1,
\end{enumerate}
or a concatenation of several of these in the sense of Definition \ref{def:concat}. 

If we let $F(S')$ be the last face in each $V$-path above, then 
we finish this proof by showing case by case that the rightmost 1 in $S$ is replaced by a 0 while going from $S$ to $S'$.
This implies $a_0\neq a_{r+1}$ by Lemma \ref{lem:acyc5}, contradicting the assumption that the $V$-path is closed.

If our $V$-path starts as in (i) or (iii), then Lemma \ref{lem:acyc1} shows that the rightmost 1 in $S$ is replaced by a 0 while going from $a_0$ to $a_1$. 

If our $V$-path starts as in (ii), then Lemma \ref{lem:acyc1} shows that the rightmost 1 in $S$ is replaced by a * while going from $a_0$ to $a_1$ and Lemma \ref{lem:acyc2} shows that this * is then replaced by a 0 while going from $a_1$ to $a_2$. 

If our $V$-path starts as in (iv), then Lemma \ref{lem:acyc1} shows that the rightmost 1 in $S$ is replaced by a * while going from $a_0$ to $a_1$ and Lemma \ref{lem:acyc4} shows that this * is then replaced by a 0 while going from $a_1$ to $a_2$. 

If our $V$-path starts as in (v), then Lemma \ref{lem:acyc1} shows that the rightmost 1 in $S$ is replaced by a * while going from $a_0$ to $a_1$; Lemma \ref{lem:acyc4} shows that this * is unchanged while going from $a_1$ to $a_2$; and finally Lemma \ref{lem:acyc2} shows that this * is then replaced by a 0 while going from $a_2$ to $a_3$.
\end{proof}

\begin{lem}
\label{lem:acyc7}
There are no nontrivial, closed $V$-paths $a_0,b_0,a_1,\ldots,b_r,a_{r+1}$ such that $a_0 = F(S)$ is of type 3, 5, or 7.
\end{lem}
\begin{proof}
Suppose there is a nontrivial, closed $V$-path $a_0,b_0,a_1,\ldots,b_r,a_{r+1}$ with $a_0$ being of type 3. Again note that if for any $j$, $a_j$ is of type 2, 4, 6, or 8 then the $V$-path is not closed since faces of those types are paired with faces of lower dimension. 
By Lemma \ref{lem:acyc2} we must have that $a_1$ is of type 1 and hence $a_1,\ldots,b_r,a_0,b_0,a_1$ is a nontrivial, closed $V$-path with $a_1$ being of type 1. 
This contradicts Lemma \ref{lem:acyc6}. 
A similar argument deals with the cases when $a_0$ is of type 5 (using Lemma \ref{lem:acyc3}) and when $a_0$ is of type 7 (using Lemma \ref{lem:acyc4}).
\end{proof}

\begin{lem}
\label{lem:acyc8}
There are no nontrivial, closed $V$-paths $a_0,b_0,a_1,\ldots,b_r,a_{r+1}$ such that $a_0 = F(S)$ is a vertex.
\end{lem}
\begin{proof}
If $a_0 = \{v_0\} = \{(1,\ldots,1,0,\ldots,0)\}$ then we are done since $\{v_0\}$ is not matched with an edge. If $a_0 = F(S)$ is any other vertex, then the rightmost 1 has at least one 0 to its left in $S$ and so
$$a_0 = F(1\cdots10f_110\cdots0), \ b_0 = F(1\cdots1*f_1*0\cdots0),$$
and
$$a_1 = F(1\cdots11f_100\cdots0) = F(S').$$
Notice the net result from $S$ to $S'$ was that the rightmost 1 was moved to its left.

Assume that $a_0,b_0,a_1,\ldots,b_r,a_{r+1}$ is a closed $V$-path such that $a_0$ is a vertex and that $a_i = F(S^{(i)})$ for $0\leq i\leq r+1$. For the same reasons as above none of the $a_i$ can be equal to $\{v_0\}$ and going from $S^{(i)}$ to $S^{(i+1)}$ always moves the rightmost 1 to its left and therefore $a_0 \neq a_{r+1}$, which is a contradiction.
\end{proof}

\begin{thm}
\label{thm:acyc}
The matchings described in Section \ref{s:Construction} are acyclic.
\end{thm}
\begin{proof}
This follows from Theorem \ref{thm:Forman} (i) which says that there are no nontrivial, closed $V$-paths if and only if $V$ is an acyclic matching of the Hasse diagram of $K$. We have already shown there are no nontrivial closed $V$-paths in Lemmas \ref{lem:acyc6}--\ref{lem:acyc8} and so we are done.
\end{proof}


\section{Subcomplex Diagrams}
\label{s:subdiagrams}

In this section we will be interested in the action of $S_n$ on $\mathbb{R}^n$ defined by permutation of the coordinates. This action will induce an action of $S_n$ on the set of faces of $J(n,k)$ defined by $g\cdot F(S) = F(g\cdot S)$ where $F(S)$ is a face of $J(n,k)$, $g\in S_n$, and $S_n$ acts on sequences $S$ of length $n$ by permuting the elements of the sequence.

\begin{rem}
\label{rem:subcomp}
If $g\cdot F(S) = F(S')$ for some $g\in S_n$, then we have $S(0)=S'(0)$ and $S(1)=S'(1)$, and hence the orbit of faces containing $F(S)$ can be uniquely defined by $S(0)$ and $S(1)$. Also note that a face $F(S)$ has dimension equal to $n-S(0)-S(1)-1$ by Theorem \ref{thm:Casselman} unless $F(S)$ is a vertex in which case it has dimension 0. We will also be interested in the $S_n$-invariant subcomplexes that arise from this action. Recall from Definition \ref{def:subcomplex}
that a subset of faces of $J(n,k)$ will correspond to a subcomplex $K'$ if and only if for every face $F(S)$ in $K'$, $K'$ also contains the faces of $J(n,k)$ that are contained in $F(S)$.
Therefore if $K'$ is an $S_n$-invariant subcomplex and $F(S)$ is in $K'$, then so is every face $F(S')$ such that $S(0)=S'(0)$ and $S(1)=S'(1)$ and so is every face $F(S')$ such that $S'$ can be obtained from $S$ by replacing any number of *'s with 0's or 1's.
\end{rem}

Given the polytope $J(n,k)$, let the ($n-k$) $\times$ ($k$) grid of squares represent this polytope in the following way. First, we will ignore the vertices and the $\emptyset$ that was added because every nontrivial $S_n$-invariant subcomplex contains these anyway. Next, each square will represent an orbit of faces starting with the bottom left square which will represent the faces $F(S)$ such that $S(0)=(n-k-1)$ and $S(1)=(k-1)$, or in other words the edges.

The other squares represent the remaining orbits as follows. Consider a square representing the faces $F(S)$ such that $S(0)=i$ and $S(1)=j$. If there is a square immediately above it then this new square represents all of the faces $F(S')$ such that $S'(0)=i$ and $S'(1)=j-1$.
On the other hand if there is a square immediately to its right then this new square represents all of the faces $F(S'')$ such that $S''(0)=i-1$ and $S''(1)=j$.

\begin{eg}
Consider the polytope $J(18,8)$ and so in this example, $n-k = 10$ and $k = 8$.
Figure \ref{fig:labeledgrid} shows the 10 $\times$ 8 grid such that the square labeled by the pair $(i, j)$ represents the faces $F(S)$ such that 
$S(0)=i$ and $S(1)=j$.
It is also worth pointing out that each diagonal of faces labeled by $(i, j)$ where $i+j$ is constant represents faces of the same dimension, since the dimension of $F(S)$ is equal to $n-S(0)-S(1)-1$. For example, the diagonal made up of the squares labeled (9,5), (8,6), and (7,7) represent every face of dimension 3.
\end{eg}

\begin{figure}[htbp]
\begin{center}
\begin{tikzpicture}[scale=1]
\def \length{10}
\def \height{8}
\def \lengthb{9} 
\def \heightb{7} 

\draw [step=1,thin,gray!40] (0,0) grid (\length, \height);
\draw [gray!40] (0,0) -- (0,\height);
\draw [gray!40] (0,0) -- (\length,0); 
\draw [gray!40] (\length, \height) -- (0,\height);
\draw [gray!40] (\length, \height) -- (\length,0);

\foreach \x in {0,...,\lengthb}
  \foreach \y in {0,...,\heightb}
  {
  \draw [xshift=.5cm, yshift=.5cm] ($(\length,\height) - (\x,\y) -(1,1)$) node {\x,\y}; 
  }
\end{tikzpicture}
\caption{Grid representing $J(18,8)$.}
\label{fig:labeledgrid}
\end{center}
\end{figure}

It will also sometimes be helpful to add lines over the top of the grid to represent one of the matchings described in Section \ref{s:Construction}. If a line crosses from one square to another one immediately above, below, or to its sides it means that there are faces from one of the squares matched with faces of the other.

\begin{eg}
In Figure \ref{fig:gridmatchlines} we see two different matchings on $J(18,8)$. On the left diagram the matching with $m_0=0$ is seen and hence has no faces of types 5--8. The vertical lines indicate faces of types 1 and 2 while the single horizontal line indicates faces of type 3 and 4.

On the right diagram the matching with $m_0=4$ and $m_1=5$ is seen. Once again the vertical lines indicate faces of types 1 and 2. However this time there are three horizontal lines. The lowest one indicates faces of type 3 and 4, the middle one indicates faces of types 5 and 6, and the highest one indicates faces of types 7 and 8. A circle has been added to the diagram just to help point out which square is the one that represents all the faces $F(S)$ such that $S(0)=m_0$ and $S(1)=m_1$.

\end{eg}

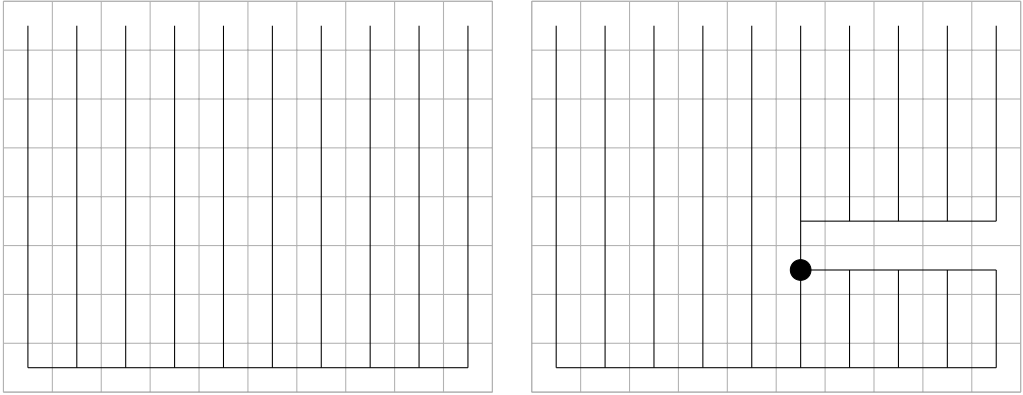
\begin{figure}[htbp]
\begin{center}
\begin{tikzpicture}[scale=.65]
\def \length{10}
\def \height{8}
\def \lengthb{9} 
\def \heightb{7} 

\draw [step=1,thin,linee] (0,0) grid (\length, \height);
\draw [linee] (0,0) -- (0,\height);
\draw [linee] (0,0) -- (\length,0);
\draw [linee] (\length, \height) -- (0,\height);
\draw [linee] (\length, \height) -- (\length,0);

\foreach \x in {0,...,\lengthb}
\draw ($(.5,.5)+(\x,0)$) -- ($(.5,7.5)+(\x,0)$);
\draw (.5,.5)--(9.5,.5);

\end{tikzpicture}
\hspace{.1in}
\begin{tikzpicture}[scale=.65]

\def \length{10}
\def \height{8}
\def \lengthb{9} 
\def \heightb{7} 

\filldraw (5.5,2.5) circle (6pt);

\draw [step=1,thin,linee] (0,0) grid (\length, \height);
\draw [linee] (0,0) -- (0,\height);
\draw [linee] (0,0) -- (\length,0);
\draw [linee] (\length, \height) -- (0,\height);
\draw [linee] (\length, \height) -- (\length,0);

\foreach \x in {0,...,5}
\draw ($(.5,.5)+(\x,0)$) -- ($(.5,7.5)+(\x,0)$);
\foreach \x in {6,...,9}
\draw ($(.5,.5)+(\x,0)$) -- ($(.5,2.5)+(\x,0)$);
\foreach \x in {6,...,9}
\draw ($(.5,3.5)+(\x,0)$) -- ($(.5,7.5)+(\x,0)$);
\draw (.5,.5)--(9.5,.5);
\draw (5.5,2.5)--(9.5,2.5);
\draw (5.5,3.5)--(9.5,3.5);

\end{tikzpicture}
\caption{Two different matchings on $J(18,8)$.}
\label{fig:gridmatchlines}
\end{center}
\end{figure}

The last addition to these diagrams allows us to point out a specific subset $K'$ of faces of $J(n,k)$ by shading the squares representing faces in $K'$. Notice that each set $K'$ is trivially $S_n$-invariant since it is defined by a set of orbits. Also, the squares shaded slightly darker will still indicate that the faces represented by this square are in $K'$, but will also mean that some of the faces are unmatched. When we refer to a square as ``shaded", we will mean either type of shading.

We will only be concerned with diagrams that represent subcomplexes of $K$.
Furthermore, for the remainder of this thesis we will only consider $S_n$-invariant subcomplexes and so for the sake of convenience we will refer to $S_n$-invariant subcomplexes simply as subcomplexes.

\begin{rem}
\label{rem:downleft}
A diagram will correspond to an ($S_n$-invariant) subcomplex precisely when, for each shaded square in the diagram, every square below it and to its left is also shaded in by Remark \ref{rem:subcomp}. 
This implies that if $Q$ denotes the highest shaded square in its column, then the highest shaded square in any column to the right of the column containing $Q$ must be below or in the row containing $Q$.
If a square in the subcomplex diagram of $K'$ is shaded in, then we will say that this square is in $K'$.
\end{rem}

\begin{eg}
\label{eg:difmatch}
Consider the polytope $J(18,8)$ again except this time also consider the subcomplex indicated by the shaded squares shown on the left diagram in Figure \ref{fig:subcomplex}.
On the right diagram of Figure \ref{fig:subcomplex}, the same subcomplex is shown except we see the matching with $m_0 = 4$ and $m_1 = 5$ represented as well. Notice how with this matching all of the unmatched faces are along the same diagonal and hence all of the same dimension. We will later see how Theorem \ref{thm:Forman} (ii) implies that this subcomplex has its reduced homology groups concentrated in a single degree. 

Unfortunately, the alternative in which a subcomplex has unmatched faces of several different dimensions does not necessarily imply that its reduced homology groups are not concentrated in a single degree. For instance, if we had used the matching with $m_0=0$ as in Figure \ref{fig:subcomplex2} instead, we see that there are unmatched faces in several dimensions, but as we have already said, this subcomplex has its reduced homology groups concentrated in a single degree.

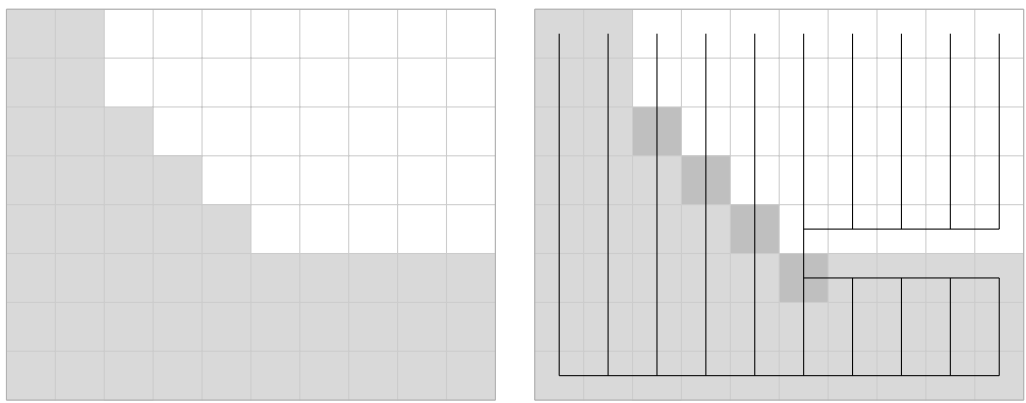
\begin{figure}[htbp]
\begin{center}
\begin{tikzpicture}[scale=.65]
\def \length{10}
\def \height{8}
\def \lengthb{9} 
\def \heightb{7} 

\fill [shade] (0,0) rectangle (1,8);
\fill [shade] (1,0) rectangle (2,8);
\fill [shade] (2,0) rectangle (3,6);
\fill [shade] (3,0) rectangle (4,5);
\fill [shade] (4,0) rectangle (5,4);
\fill [shade] (5,0) rectangle (6,3);
\fill [shade] (6,0) rectangle (7,3);
\fill [shade] (7,0) rectangle (8,3);
\fill [shade] (8,0) rectangle (9,3);
\fill [shade] (9,0) rectangle (10,3);

\draw [step=1,thin,gray!40] (0,0) grid (\length, \height);
\draw [linee] (0,0) -- (0,\height);
\draw [linee] (0,0) -- (\length,0);
\draw [linee] (\length, \height) -- (0,\height);
\draw [linee] (\length, \height) -- (\length,0);

\end{tikzpicture}
\hspace{.1in}
\begin{tikzpicture}[scale=.65]
\def \length{10}
\def \height{8}
\def \lengthb{9} 
\def \heightb{7} 

\fill [shade] (0,0) rectangle (1,8);
\fill [shade] (1,0) rectangle (2,8);
\fill [shade] (2,0) rectangle (3,6);
\fill [shade] (3,0) rectangle (4,5);
\fill [shade] (4,0) rectangle (5,4);
\fill [shade] (5,0) rectangle (6,3);
\fill [shade] (6,0) rectangle (7,3);
\fill [shade] (7,0) rectangle (8,3);
\fill [shade] (8,0) rectangle (9,3);
\fill [shade] (9,0) rectangle (10,3);

\fill [shade b] (2,5) rectangle (3,6);
\fill [shade b] (3,4) rectangle (4,5);
\fill [shade b] (4,3) rectangle (5,4);
\fill [shade b] (5,2) rectangle (6,3);

\draw [step=1,thin,gray!40] (0,0) grid (\length, \height);
\draw [linee] (0,0) -- (0,\height);
\draw [linee] (0,0) -- (\length,0);
\draw [linee] (\length, \height) -- (0,\height);
\draw [linee] (\length, \height) -- (\length,0);

\foreach \x in {0,...,5}
\draw ($(.5,.5)+(\x,0)$) -- ($(.5,7.5)+(\x,0)$);
\foreach \x in {6,...,9}
\draw ($(.5,.5)+(\x,0)$) -- ($(.5,2.5)+(\x,0)$);
\foreach \x in {6,...,9}
\draw ($(.5,3.5)+(\x,0)$) -- ($(.5,7.5)+(\x,0)$);
\draw (.5,.5)--(9.5,.5);
\draw (5.5,2.5)--(9.5,2.5);
\draw (5.5,3.5)--(9.5,3.5);

\end{tikzpicture}
\caption{Subcomplex diagram along with a matching.}
\label{fig:subcomplex}
\end{center}
\end{figure}

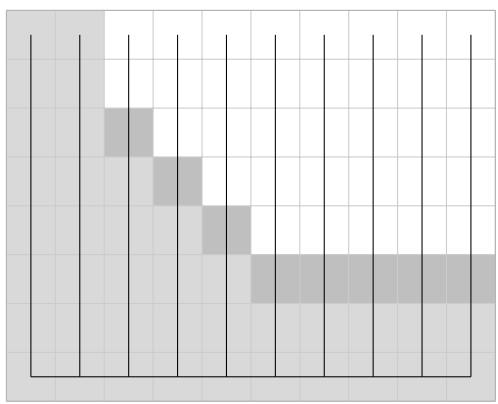
\begin{figure}[htbp!]
\begin{center}
\begin{tikzpicture}[scale=.65]
\def \length{10}
\def \height{8}
\def \lengthb{9} 
\def \heightb{7} 

\fill [shade] (0,0) rectangle (1,8);
\fill [shade] (1,0) rectangle (2,8);
\fill [shade] (2,0) rectangle (3,6);
\fill [shade] (3,0) rectangle (4,5);
\fill [shade] (4,0) rectangle (5,4);
\fill [shade] (5,0) rectangle (6,3);
\fill [shade] (6,0) rectangle (7,3);
\fill [shade] (7,0) rectangle (8,3);
\fill [shade] (8,0) rectangle (9,3);
\fill [shade] (9,0) rectangle (10,3);

\fill [shade b] (2,5) rectangle (3,6);
\fill [shade b] (3,4) rectangle (4,5);
\fill [shade b] (4,3) rectangle (5,4);
\fill [shade b] (5,2) rectangle (6,3);
\fill [shade b] (6,2) rectangle (7,3);
\fill [shade b] (7,2) rectangle (8,3);
\fill [shade b] (8,2) rectangle (9,3);
\fill [shade b] (9,2) rectangle (10,3);

\draw [step=1,thin,gray!40] (0,0) grid (\length, \height);
\draw [linee] (0,0) -- (0,\height);
\draw [linee] (0,0) -- (\length,0);
\draw [linee] (\length, \height) -- (0,\height);
\draw [linee] (\length, \height) -- (\length,0);

\foreach \x in {0,...,9}
\draw ($(.5,.5)+(\x,0)$) -- ($(.5,7.5)+(\x,0)$);
\draw (.5,.5)--(9.5,.5);

\end{tikzpicture}
\caption{Subcomplex diagram with a different matching.}
\label{fig:subcomplex2}
\end{center}
\end{figure}

\end{eg}

The last thing that needs to be decided in this section is which matching should be used for each subcomplex, since there are several choices. 
As we saw in Example \ref{eg:difmatch}, choosing the appropriate matching can lead us to the conclusion that a subcomplex has its reduced homology groups concentrated in a single degree much faster than a matching that is poorly chosen.

\begin{defn}
\label{defn:canon}
Consider the subcomplex diagram for some subcomplex $K'$.
The \textbf{canonical matching} is chosen in the following way.
If there are no rows in this diagram that are completely shaded in or the entire diagram is shaded in, then we will let $m_0=0$ and $m_1=0$. Otherwise, let $m_0=i$ and $m_1=j$ where the square labeled by $(i,j)$ is in the highest row that is completely shaded in and is the leftmost square such that the square above it is not shaded.
\end{defn}

For the remainder of this thesis if we are given a subcomplex $K'$, then we will only consider it with the canonical matching unless otherwise stated.

\begin{eg}
Consider the polytope $J(18,8)$ and the same subcomplex as in the previous example. We see in Figure \ref{fig:howto} why the matching with $m_0 = 4$ and $m_1 = 5$ is the appropriate one to use according to the rules above.
\end{eg}

\begin{figure}[htbp]
\begin{center}
\begin{tikzpicture}[scale=.65]
\def \length{10}
\def \height{8}
\def \lengthb{9} 
\def \heightb{7} 
\fill [shade] (0,0) rectangle (1,8);
\fill [shade] (1,0) rectangle (2,8);
\fill [shade] (2,0) rectangle (3,6);
\fill [shade] (3,0) rectangle (4,5);
\fill [shade] (4,0) rectangle (5,4);
\fill [shade] (5,0) rectangle (6,3);
\fill [shade] (6,0) rectangle (7,3);
\fill [shade] (7,0) rectangle (8,3);
\fill [shade] (8,0) rectangle (9,3);
\fill [shade] (9,0) rectangle (10,3);
\draw [step=1,thin,gray!40] (0,0) grid (\length, \height);
\draw [linee] (0,0) -- (0,\height);
\draw [linee] (0,0) -- (\length,0);
\draw [linee] (\length, \height) -- (0,\height);
\draw [linee] (\length, \height) -- (\length,0);
\draw (0,2) rectangle (10,3);
\draw (10,2.5) -- (11.5,3.5) node [right] {highest completely shaded in row};
\draw [<-] (5.5,2.5) -- (11.5,.5) node [right, text width=6cm] {leftmost square in the highest \\completely shaded in row such that the square above it is not in the \\subcomplex};
\end{tikzpicture}
\caption{How to choose a matching.}
\label{fig:howto}
\end{center}
\end{figure}
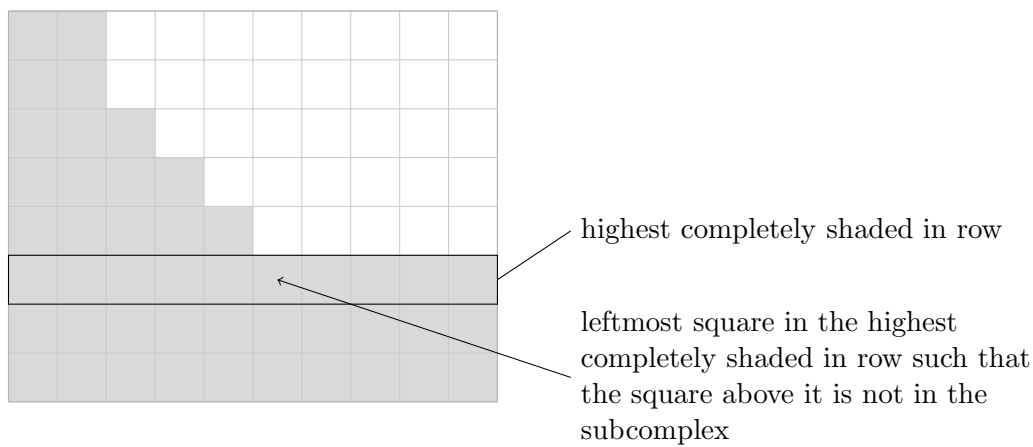

Next consider the polytope $J(35,15)$ and the subcomplex corresponding to the diagram in Figure \ref{fig:howto2}. In this case we have $m_0=2$ and $m_1=13$. The square containing the faces $F(S)$ such that $S(0)=2$ and $S(1)=13$ is marked with a circle and the corresponding matching is also shown. However, if we consider the subcomplex corresponding to the diagram in Figure \ref{fig:howto3}, we see that there are no rows that are completely shaded in and hence $m_0=0$ and $m_1=0$. 

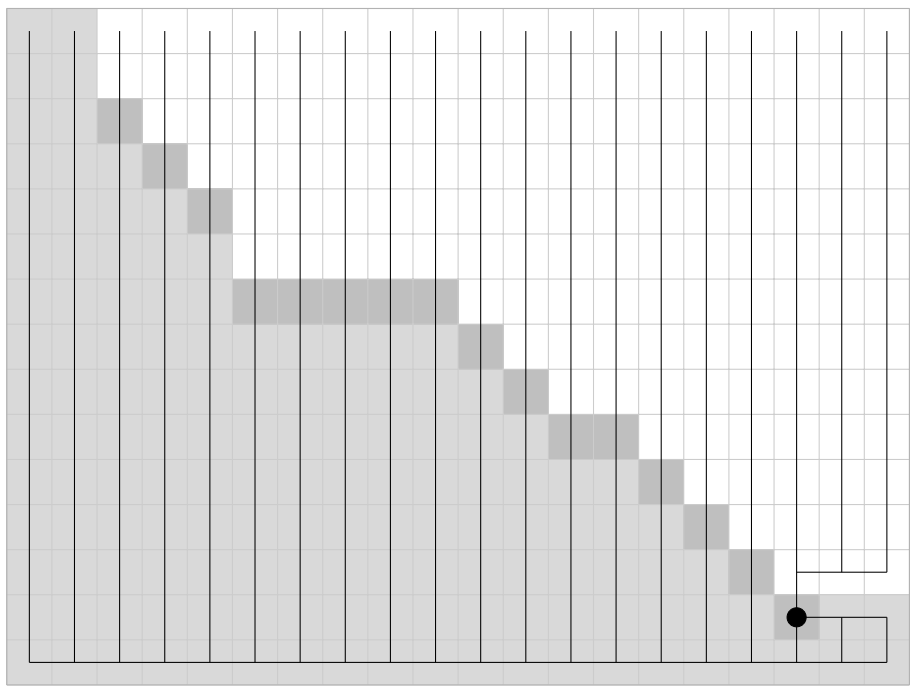
\begin{figure}[htbp]
\begin{center}
\begin{tikzpicture}[scale=.6]
\def \length{20}
\def \height{15}
\def \lengthb{9} 
\def \heightb{7} 
\fill [shade] (0,0) rectangle (1,15);
\fill [shade] (1,0) rectangle (2,15);
\fill [shade] (2,0) rectangle (3,13);
\fill [shade] (3,0) rectangle (4,12);
\fill [shade] (4,0) rectangle (5,11);
\fill [shade] (5,0) rectangle (6,9);
\fill [shade] (6,0) rectangle (7,9);
\fill [shade] (7,0) rectangle (8,9);
\fill [shade] (8,0) rectangle (9,9);
\fill [shade] (9,0) rectangle (10,9);
\fill [shade] (10,0) rectangle (11,8);
\fill [shade] (11,0) rectangle (12,7);
\fill [shade] (12,0) rectangle (13,6);
\fill [shade] (13,0) rectangle (14,6);
\fill [shade] (14,0) rectangle (15,5);
\fill [shade] (15,0) rectangle (16,4);
\fill [shade] (16,0) rectangle (17,3);
\fill [shade] (17,0) rectangle (18,2);
\fill [shade] (18,0) rectangle (19,2);
\fill [shade] (19,0) rectangle (20,2);
\fill [shade b] (2,12) rectangle (3,13);
\fill [shade b] (3,11) rectangle (4,12);
\fill [shade b] (4,10) rectangle (5,11);
\fill [shade b] (5,8) rectangle (6,9);
\fill [shade b] (6,8) rectangle (7,9);
\fill [shade b] (7,8) rectangle (8,9);
\fill [shade b] (8,8) rectangle (9,9);
\fill [shade b] (9,8) rectangle (10,9);
\fill [shade b] (10,7) rectangle (11,8);
\fill [shade b] (11,6) rectangle (12,7);
\fill [shade b] (12,5) rectangle (13,6);
\fill [shade b] (13,5) rectangle (14,6);
\fill [shade b] (14,4) rectangle (15,5);
\fill [shade b] (15,3) rectangle (16,4);
\fill [shade b] (16,2) rectangle (17,3);
\fill [shade b] (17,1) rectangle (18,2);
\filldraw [shade c] (17.5,1.5) circle (6pt);
\draw [step=1,thin,gray!40] (0,0) grid (\length, \height);
\draw [linee] (0,0) -- (0,\height);
\draw [linee] (0,0) -- (\length,0);
\draw [linee] (\length, \height) -- (0,\height);
\draw [linee] (\length, \height) -- (\length,0);
\foreach \x in {0,...,17}
\draw ($(.5,.5)+(\x,0)$) -- ($(.5,14.5)+(\x,0)$);
\foreach \x in {18,...,19}
\draw ($(.5,.5)+(\x,0)$) -- ($(.5,1.5)+(\x,0)$);
\foreach \x in {18,...,19}
\draw ($(.5,2.5)+(\x,0)$) -- ($(.5,14.5)+(\x,0)$);
\draw (17.5,2.5)--(19.5,2.5); 
\draw (17.5,1.5)--(19.5,1.5); 
\draw (.5,.5)--(19.5,.5); 
\end{tikzpicture}
\caption{How to choose a matching II.}
\label{fig:howto2}
\end{center}
\end{figure}

\begin{figure}[htbp]
\begin{center}
\begin{tikzpicture}[scale=.6]
\def \length{20}
\def \height{15}
\def \lengthb{9} 
\def \heightb{7} 
\fill [shade] (0,0) rectangle (1,15);
\fill [shade] (1,0) rectangle (2,15);
\fill [shade] (2,0) rectangle (3,13);
\fill [shade] (3,0) rectangle (4,12);
\fill [shade] (4,0) rectangle (5,11);
\fill [shade] (5,0) rectangle (6,9);
\fill [shade] (6,0) rectangle (7,9);
\fill [shade] (7,0) rectangle (8,9);
\fill [shade] (8,0) rectangle (9,9);
\fill [shade] (9,0) rectangle (10,9);
\fill [shade] (10,0) rectangle (11,8);
\fill [shade] (11,0) rectangle (12,7);
\fill [shade] (12,0) rectangle (13,6);
\fill [shade] (13,0) rectangle (14,6);
\fill [shade b] (2,12) rectangle (3,13);
\fill [shade b] (3,11) rectangle (4,12);
\fill [shade b] (4,10) rectangle (5,11);
\fill [shade b] (5,8) rectangle (6,9);
\fill [shade b] (6,8) rectangle (7,9);
\fill [shade b] (7,8) rectangle (8,9);
\fill [shade b] (8,8) rectangle (9,9);
\fill [shade b] (9,8) rectangle (10,9);
\fill [shade b] (10,7) rectangle (11,8);
\fill [shade b] (11,6) rectangle (12,7);
\fill [shade b] (12,5) rectangle (13,6);
\fill [shade b] (13,5) rectangle (14,6);
\fill [shade b] (13,0) rectangle (14,1);
\filldraw [shade c] (19.5,14.5) circle (6pt);
\draw [step=1,thin,gray!40] (0,0) grid (\length, \height);
\draw [linee] (0,0) -- (0,\height);
\draw [linee] (0,0) -- (\length,0);
\draw [linee] (\length, \height) -- (0,\height);
\draw [linee] (\length, \height) -- (\length,0);
\foreach \x in {0,...,19}
\draw ($(.5,.5)+(\x,0)$) -- ($(.5,14.5)+(\x,0)$);
\draw (.5,.5)--(19.5,.5); 
\end{tikzpicture}
\caption{How to choose a matching III.}
\label{fig:howto3}
\end{center}
\end{figure}
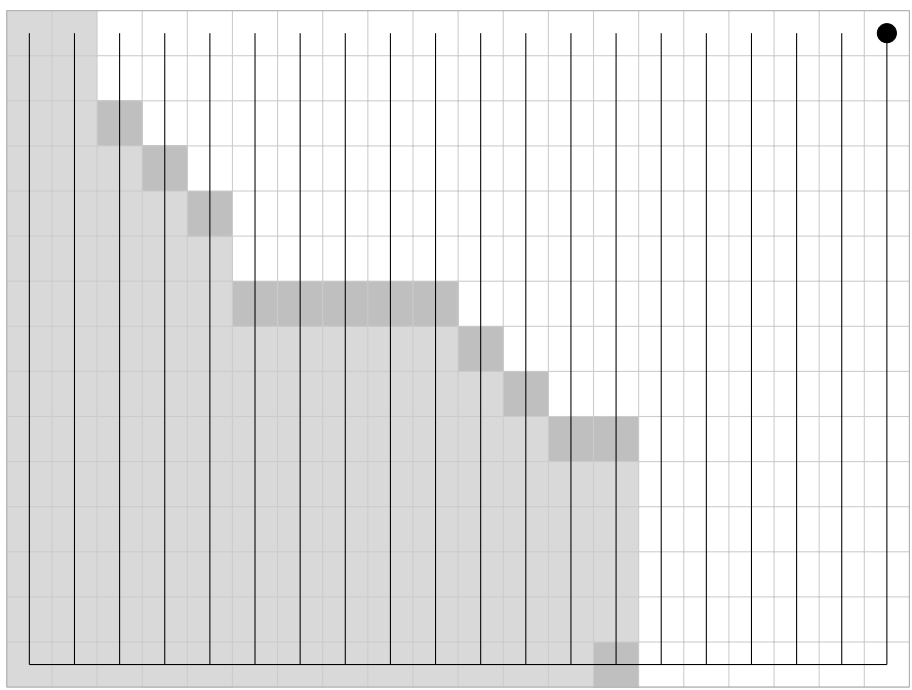


\chapter[Subcomplexes whose Reduced Homology is Concentrated in a Single Degree]{Subcomplexes whose Reduced Homology \\ is Concentrated in a Single Degree}
\label{chap:class}

The goal of this chapter is to consider the CW complex $K$ obtained from the faces of $J(n,k)$ and classify all $S_n$-invariant subcomplexes whose reduced homology groups are concentrated in a single degree.
We will show in Theorem \ref{thm:class} that $S_n$-invariant subcomplexes have this property if and only if they have unmatched faces in a single dimension using the canonical matching.
The ``if" direction is easier, and is accomplished in Proposition \ref{prop:K'reduced}.  The ``only if" direction is harder, and occupies Sections \ref{s:single} and \ref{s:class2}.

\section{Subcomplexes with Unmatched Faces in a Single Dimension}
\label{s:class}
In this section we will identify every subcomplex containing unmatched faces in a single dimension and then show that 
their reduced homology groups are concentrated in a single degree.

\begin{defn}
Given a subcomplex $K'$ of $K$, let $M_{K'}$ denote the square that contains the faces $F(S)$ such that $S(0)=m_0$ and $S(1)=m_1$. When the subcomplex $K'$ is clear from context, we will refer to $M_{K'}$ simply as $M$.
\end{defn}

\begin{lem}
\label{lem:unmatch}
Let $K'$ be a subcomplex of $K$ such that $K'$ contains the edges of $J(n,k)$ and
consider its subcomplex diagram with the canonical matching.
A square in $K'$ contains unmatched faces if and only if that square is one of the following:
\begin{enumerate}
  \item [\rm(i)] the highest shaded square in each column left of and including the column containing $M$ unless that square is in the top row; or
  \item [\rm(ii)] the rightmost shaded square in the bottom row unless the bottom row is completely shaded in.
\end{enumerate}
\end{lem}
\begin{proof}
Since $K'$ contains the edges, it must also contain the vertices and hence every face of type 9 or 10 is matched. Let $F(S)$ be any face in $K'$ of type 1--8 and let $Q$ denote the square containing $F(S)$. If $F(S)$ is of type 1, then it is matched with a face $F(S')$ in the square above $Q$ since we replace a 1 in $S$ with a * to get $S'$. If $F(S)$ is of type 3, 5, or 7 then it is matched with a face $F(S')$ in the square to the right of $Q$ since we replace a 0 in $S$ with a * to get $S'$. If $F(S)$ is of type 2, 4, 6, or 8 then it is matched with a face $F(S')$ in the square below or to the left of $Q$ since we replace a * in $S$ with a 0 or 1 to get $S'$. However $K'$ is a subcomplex and so if $Q$ is in $K'$ then any square below $Q$ or to the left of $Q$ must also be in $K'$. Therefore if $F(S)$ is unmatched, then either $F(S)$ is of type 1 and the square above $Q$ is not in $K'$ or $F(S)$ is of type 3, 5, or 7 and the square to the right of $Q$ is not in $K'$.

Suppose that $Q$ is in any column left of the column containing $M$ and so $S(0)>m_0\geq 0$. 
If $Q$ is in the top row then $S(1)=0$ and so $Q$ only contains faces of type 2. If $Q$ is in any row except the top or bottom then $0<S(1)<k-1$ and so $Q$ contains faces of types 1 and 2 and no others. 
If $Q$ is in the bottom row then $S(1)=k-1$ and so $Q$ contains faces of types 1 and 3, $Q$ contains faces of type 4 if it is not in the leftmost column, and $Q$ does not contain any other types of faces. 
Therefore $Q$ only contains unmatched faces if $Q$ is the highest shaded square in its column and is not in the top row or $Q$ is the rightmost shaded square in the bottom row and the bottom row is not completely shaded in.

Next, suppose that $Q$ is in the same column as $M$ or one to its right and that $m_1\neq 0$.
In this case $M$ is not in the top row. The row containing $M$ is completely shaded by Definition \ref{defn:canon}. This implies every row below the row containing $M$ is completely shaded since $K'$ is a subcomplex.
If $Q$ is in a row above $M$ then the faces of $Q$ are not in $K'$ by Definition \ref{defn:canon}.
If $Q=M$ then $0<S(1)=m_1\leq k-1$. This implies that $Q$ contains unmatched faces of type 1 since the square above $M$ is not in $K'$. 
If $Q$ is in the same row as $M$ but not equal to $M$ then $S(1)=m_1$ and $S(0)<m_0$. This implies that $Q$ does not contain faces of type 1 or 7, but could contain faces of type 3 or 5. 
However, either the square to the right of $Q$ is in $K'$ or $Q$ is in the rightmost column and has no faces of type 3 or 5, which implies $Q$ has no unmatched faces.
If $Q$ is in a row below $M$ then $Q$ cannot contain unmatched faces since the square above $Q$ is in $K'$ and either the square to the right of $Q$ is in $K'$ or $Q$ is in the rightmost column.
 
Finally, suppose that $Q$ is in the same column as $M$ and that $m_1=0$. In this case $M$ is in the top row and so either the entire top row is shaded, meaning $K'=K$, or there are no rows that are completely shaded. If $K'=K$ then every face is matched. If there are no rows that are completely shaded, then $M$ is in the rightmost column and no square in this column is in $K'$ anyway.
\end{proof}

We can see an example of Lemma \ref{lem:unmatch} by looking at Figures \ref{fig:howto2} and \ref{fig:howto3}. The square with the circle in it indicates $M$ and the darker shaded squares show that every square containing unmatched faces was described in this lemma.

\begin{lem}
\label{lem:subtypes1}
Suppose that the diagram of a subcomplex $K'$ using the canonical matching does not have any rows that are completely shaded in. 
If the only unmatched faces are in a single dimension $d>0$, then the leftmost $i$ columns are completely shaded in where $0\leq i\leq d$ and the squares under and including the diagonal containing faces of dimension $d$ are shaded in. Furthermore, no other squares are shaded in.
\end{lem}
\begin{proof}
The square $M$ is in the top right corner since the diagram does not have any rows that are completely shaded in.
Also, $K'$ contains at least the edges of $J(n,k)$ since $d>0$.
Therefore the rightmost shaded square in the bottom row contains unmatched faces by Lemma \ref{lem:unmatch}
and so this square must contain faces of dimension $d$. This implies that only the leftmost $d$ columns contain shaded squares.

Notice that the highest shaded square in each of the leftmost $d$ columns
contains unmatched faces unless that square is in the top row by Lemma \ref{lem:unmatch}.
Therefore the highest shaded square in each of the leftmost $d$ columns is either in the top row or contains faces of dimension $d$.

Since $K'$ is a subcomplex then for each shaded square in the diagram, every square below it and to its left is also shaded in. 
Therefore, the diagram of $K'$ must shade in every square under and including the diagonal containing faces of dimension $d$ and has between zero and $d$ consecutive columns, starting with the leftmost one, completely shaded in as desired.
\end{proof}

\begin{eg}
Consider $J(18,8)$. Figure \ref{fig:subcomplextypes1} below shows some of the types of subcomplexes described in Lemma \ref{lem:subtypes1}.
\end{eg}

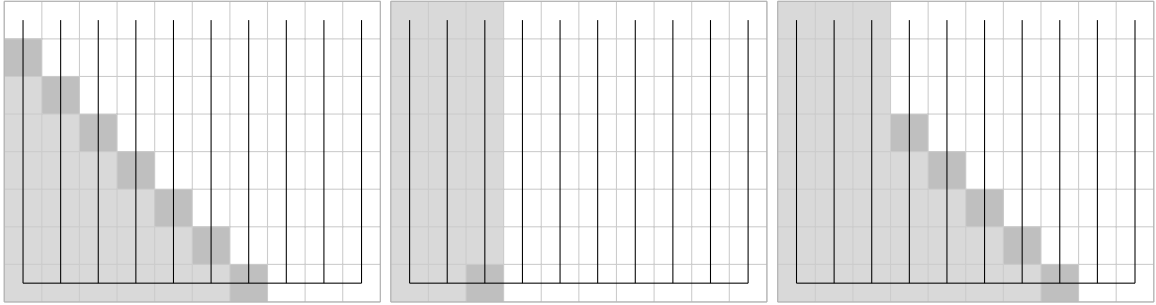
\begin{figure}[htbp]
\def\sca{.5}
\def\split{.4}
\begin{center}
\begin{tikzpicture}[scale=\sca]
\def \length{10}
\def \height{8}
\def \lengthb{9} 
\def \heightb{7} 
\fill [shade] (0,0) rectangle (1,7);
\fill [shade] (1,0) rectangle (2,6);
\fill [shade] (2,0) rectangle (3,5);
\fill [shade] (3,0) rectangle (4,4);
\fill [shade] (4,0) rectangle (5,3);
\fill [shade] (5,0) rectangle (6,2);
\fill [shade] (6,0) rectangle (7,1);
\foreach \x in {0,...,6}
\fill [shade b] ($(\x,7)-(0,\x)$) rectangle ($(\x,6)-(-1,\x)$);
\draw [step=1,thin,gray!40] (0,0) grid (\length, \height);
\draw [linee] (0,0) -- (0,\height);
\draw [linee] (0,0) -- (\length,0);
\draw [linee] (\length, \height) -- (0,\height);
\draw [linee] (\length, \height) -- (\length,0);
\foreach \x in {0,...,9}
\draw [linea]($(.5,.5)+(\x,0)$) -- ($(.5,7.5)+(\x,0)$);
\draw [linea] (.5,.5)--(9.5,.5);
\end{tikzpicture}
\begin{tikzpicture}[scale=\sca]
\def \length{10}
\def \height{8}
\def \lengthb{9} 
\def \heightb{7} 
\fill [shade] (0,0) rectangle (1,8);
\fill [shade] (1,0) rectangle (2,8);
\fill [shade] (2,0) rectangle (3,8);
\fill [shade b] (2,0) rectangle (3,1);
\draw [step=1,thin,gray!40] (0,0) grid (\length, \height);
\draw [linee] (0,0) -- (0,\height);
\draw [linee] (0,0) -- (\length,0);
\draw [linee] (\length, \height) -- (0,\height);
\draw [linee] (\length, \height) -- (\length,0);
\foreach \x in {0,...,9}
\draw [linea]($(.5,.5)+(\x,0)$) -- ($(.5,7.5)+(\x,0)$);
\draw [linea] (.5,.5)--(9.5,.5);
\end{tikzpicture}
\begin{tikzpicture}[scale=\sca]
\def \length{10}
\def \height{8}
\def \lengthb{9} 
\def \heightb{7} 
\foreach \x in {0,1,2}
\fill [shade] (\x,0) rectangle ($(1,8)+(\x,0)$);
\foreach \x in {3,...,7}
\fill [shade] (\x,0) rectangle ($(\x,8)-(-1,\x)$);
\foreach \x in {3,...,7}
\fill [shade b] ($(\x,7)-(0,\x)$) rectangle ($(\x,8)-(-1,\x)$);
\draw [step=1,thin,gray!40] (0,0) grid (\length, \height);
\draw [linee] (0,0) -- (0,\height);
\draw [linee] (0,0) -- (\length,0);
\draw [linee] (\length, \height) -- (0,\height);
\draw [linee] (\length, \height) -- (\length,0);
\foreach \x in {0,...,9}
\draw [linea]($(.5,.5)+(\x,0)$) -- ($(.5,7.5)+(\x,0)$);
\draw [linea] (.5,.5)--(9.5,.5);
\end{tikzpicture}
\caption{Different types of subcomplexes described in Lemma \ref{lem:subtypes1}.}
\label{fig:subcomplextypes1}
\end{center}
\end{figure}

The following lemma is the converse of Lemma \ref{lem:subtypes1}.

\begin{lem}
\label{lem:subtypes1b}
Consider the diagram of a subcomplex $K'$ using the canonical matching that does not have any rows that are completely shaded in. 
Suppose the leftmost $i$ columns are completely shaded in where $0\leq i\leq d$ and the squares under and including the diagonal containing faces of dimension $d$ are also shaded in. If these are the only shaded squares then the only unmatched faces are in a single dimension $d>0$.
\end{lem}
\begin{proof} 
Consider the diagram of $K'$. 
Since there are no rows that are completely shaded in, we have $m_0=m_1=0$. 
This implies $M$ is in the top row and rightmost column. 

Suppose first that $i=d$. In this case, the leftmost $d$ columns are completely shaded in and these are the only shaded squares. The only square in $K'$ containing unmatched faces is then the rightmost shaded square in the bottom row by Lemma \ref{lem:unmatch}. This square contains faces of dimension $d$ as desired.

Next suppose that $i<d$. In this case, the leftmost $i$ columns are completely shaded in, the squares under and including the diagonal containing faces of dimension $d$ are shaded in, and no other squares are shaded in.
By Lemma \ref{lem:unmatch}
the leftmost $i$ columns do not contain any unmatched faces while the highest shaded square in each remaining column and the rightmost shaded square in the bottom row each contain unmatched faces. These squares all contain faces of dimension $d$, as desired.
\end{proof}

\begin{lem}
\label{lem:subtypes2}
Suppose the diagram of a subcomplex $K'$ has precisely the bottom $j$ rows completely shaded in where $0<j<k$.
If $K'$ uses the canonical matching and has unmatched faces in a single dimension $d$ where $d<n-1$,
then the following squares are all shaded in:
\begin{enumerate}
  \item [\rm(i)] the bottom $j$ rows; 
  \item [\rm(ii)] the leftmost $i$ columns for some $0\leq i < n-k-m_0$; and
  \item [\rm(iii)] the squares under and including the diagonal containing faces of dimension $d$.
\end{enumerate}
Furthermore, no other squares are shaded in.
\end{lem}
\begin{proof}
Since $j<k$, the top row is not completely shaded in and so $K'\neq K$.
This implies that $m_1\neq 0$ and so $M$ is in the $(n-k-m_0)$-th column from the left and not in the top row. 
Therefore $M$ contains unmatched faces by Lemma \ref{lem:unmatch} and so $M$ must contain faces of dimension $d$. 

Next consider the set of squares in the same row as $M$ including and to the right of $M$.
None of the squares above the squares in this set are in $K'$ by Definition \ref{defn:canon}. So from now on we only need to consider the columns strictly to the left of $M$. 
The highest shaded square in each of these columns is either in the top row or contains faces of dimension $d$ by Lemma \ref{lem:unmatch}. However, $K'$ is a subcomplex and so if any columns are completely shaded in, then they must be consecutive and start from the left, as desired.
\end{proof}

\begin{eg}
Consider $J(18,8)$. Figure \ref{fig:subcomplextypes2} below shows some of the types of subcomplexes described in Lemma \ref{lem:subtypes2}. In each diagram, the square $M$ is denoted by a circle.
\end{eg}

\begin{figure}[htbp]
\def\sca{.55}
\def\split{.4}
\begin{center}
\begin{tikzpicture}[scale=\sca]
\def \length{10}
\def \height{8}
\def \lengthb{9} 
\def \heightb{7} 
\fill [shade] (0,0) rectangle (10,7);
\foreach \x in {0}
\fill [shade b] ($(\x,7)-(0,\x)$) rectangle ($(\x,6)-(-1,\x)$);
\draw [step=1,thin,gray!40] (0,0) grid (\length, \height);
\draw [linee] (0,0) -- (0,\height);
\draw [linee] (0,0) -- (\length,0);
\draw [linee] (\length, \height) -- (0,\height);
\draw [linee] (\length, \height) -- (\length,0);
\foreach \x in {0,...,9}
\draw [linea]($(.5,.5)+(\x,0)$) -- ($(.5,6.5)+(\x,0)$);
\draw [linea] (.5,6.5)--(.5,7.5);
\draw [linea] (.5,.5)--(9.5,.5);
\draw [linea] (.5,6.5)--(9.5,6.5);
\draw [linea] (.5,7.5)--(9.5,7.5);
\filldraw [shade c] (.5,6.5) circle (6pt);
\end{tikzpicture}
\noindent
\hspace{\split in}
\begin{tikzpicture}[scale=\sca]
\def \length{10}
\def \height{8}
\def \lengthb{9} 
\def \heightb{7} 
\fill [shade] (0,0) rectangle (1,7);
\fill [shade] (1,0) rectangle (2,6);
\fill [shade] (2,0) rectangle (3,5);
\fill [shade] (3,0) rectangle (10,4);
\foreach \x in {0,...,3}
\fill [shade b] ($(\x,7)-(0,\x)$) rectangle ($(\x,6)-(-1,\x)$);
\draw [step=1,thin,gray!40] (0,0) grid (\length, \height);
\draw [linee] (0,0) -- (0,\height);
\draw [linee] (0,0) -- (\length,0);
\draw [linee] (\length, \height) -- (0,\height);
\draw [linee] (\length, \height) -- (\length,0);
\foreach \x in {0,...,3}
\draw [linea]($(.5,.5)+(\x,0)$) -- ($(.5,7.5)+(\x,0)$);
\foreach \x in {4,...,9}
\draw [linea]($(.5,.5)+(\x,0)$) -- ($(.5,3.5)+(\x,0)$);
\foreach \x in {4,...,9}
\draw [linea]($(.5,.5)+(\x,4)$) -- ($(.5,7.5)+(\x,0)$);
\draw [linea] (.5,.5)--(9.5,.5);
\draw [linea] (3.5,3.5)--(9.5,3.5);
\draw [linea] (3.5,4.5)--(9.5,4.5);
\filldraw [shade c] (3.5,3.5) circle (6pt);
\end{tikzpicture}
\end{center}

\begin{center}
\begin{tikzpicture}[scale=\sca]
\def \length{10}
\def \height{8}
\def \lengthb{9} 
\def \heightb{7} 
\fill [shade] (0,0) rectangle (3,8);
\fill [shade] (0,0) rectangle (10,5);
\fill [shade b] (3,4) rectangle (4,5);
\draw [step=1,thin,gray!40] (0,0) grid (\length, \height);
\draw [linee] (0,0) -- (0,\height);
\draw [linee] (0,0) -- (\length,0);
\draw [linee] (\length, \height) -- (0,\height);
\draw [linee] (\length, \height) -- (\length,0);
\def\ma{3} 
\def\mb{4} 
\foreach \x in {0,...,\ma}
\draw [linea]($(.5,.5)+(\x,0)$) -- ($(.5,7.5)+(\x,0)$);
\foreach \x in {\ma,...,9}
\draw [linea]($(.5,.5)+(\x,0)$) -- ($(.5,.5)+(\x,\mb)$);
\foreach \x in {\ma,...,9}
\draw [linea]($(.5,1.5)+(\x,\mb)$) -- ($(.5,7.5)+(\x,0)$);
\draw [linea] ($(\ma,\mb)+(.5,.5)$) -- ($(9.5,.5)+(0,\mb)$);
\draw [linea] ($(\ma,\mb)+(.5,1.5)$) -- ($(9.5,1.5)+(0,\mb)$);
\draw [linea] (.5,.5)--(9.5,.5);
\filldraw [shade c] (3.5,4.5) circle (6pt);
\end{tikzpicture}
\hspace{\split in}
\begin{tikzpicture}[scale=\sca]
\def \length{10}
\def \height{8}
\def \lengthb{9} 
\def \heightb{7} 
\fill [shade] (0,0) rectangle (3,8);
\foreach \x in {3,4,5}
\fill [shade] (\x,0) rectangle ($(\x,8)-(-1,\x)$);
\fill [shade] (6,0) rectangle (10,3);
\foreach \x in {3,4,5}
\fill [shade b] ($(\x,7)-(0,\x)$) rectangle ($(\x,8)-(-1,\x)$);
\draw [step=1,thin,gray!40] (0,0) grid (\length, \height);
\draw [linee] (0,0) -- (0,\height);
\draw [linee] (0,0) -- (\length,0); 
\draw [linee] (\length, \height) -- (0,\height);
\draw [linee] (\length, \height) -- (\length,0);
\def\ma{5} 
\def\mb{2} 
\foreach \x in {0,...,\ma}
\draw [linea]($(.5,.5)+(\x,0)$) -- ($(.5,7.5)+(\x,0)$);
\foreach \x in {\ma,...,9}
\draw [linea]($(.5,.5)+(\x,0)$) -- ($(.5,.5)+(\x,\mb)$);
\foreach \x in {\ma,...,9}
\draw [linea]($(.5,1.5)+(\x,\mb)$) -- ($(.5,7.5)+(\x,0)$);
\draw [linea] ($(\ma,\mb)+(.5,.5)$) -- ($(9.5,.5)+(0,\mb)$);
\draw [linea] ($(\ma,\mb)+(.5,1.5)$) -- ($(9.5,1.5)+(0,\mb)$);
\draw [linea] (.5,.5)--(9.5,.5);
\filldraw [shade c] (5.5,2.5) circle (6pt);
\end{tikzpicture}
\caption{Different types of subcomplexes described in Lemma \ref{lem:subtypes2}.}
\label{fig:subcomplextypes2}
\end{center}
\end{figure}
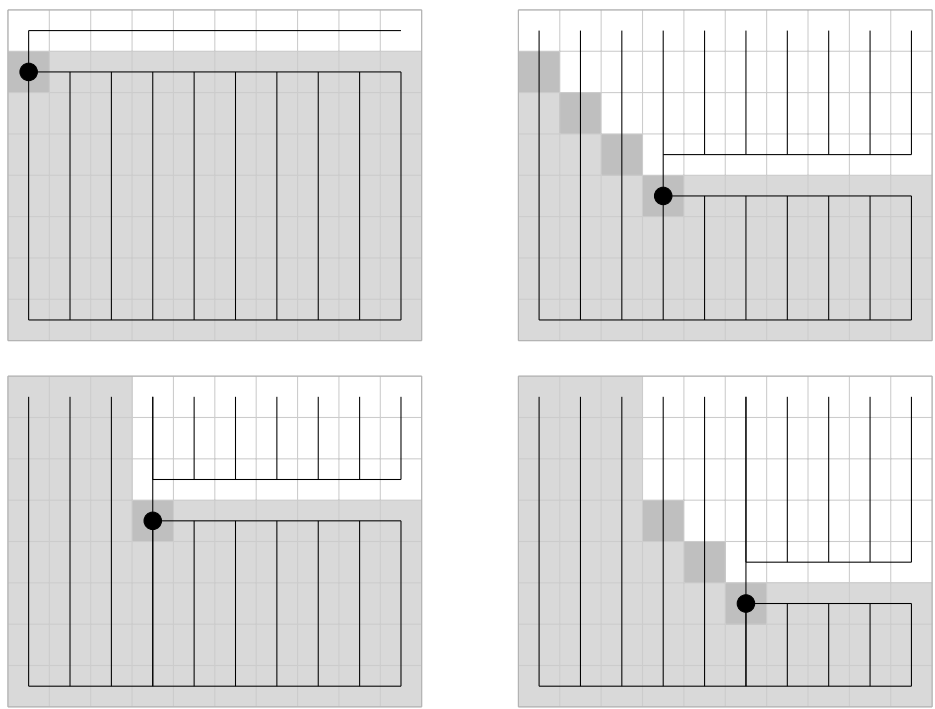

The following lemma is the converse of Lemma \ref{lem:subtypes2}.

\begin{lem}
\label{lem:subtypes2b}
Suppose the diagram of a subcomplex $K'$ has precisely the bottom $j$ rows completely shaded in where $0<j<k$.
Suppose the following squares are shaded in and that they are the only squares that are shaded in:
\begin{enumerate}
  \item [\rm(i)] the bottom $j$ rows;
  \item [\rm(ii)] the leftmost $i$ columns for some $0\leq i < n-k-m_0$; and
  \item [\rm(iii)] the squares under and including the diagonal containing faces of dimension $d$.
\end{enumerate}
If $d<n-1$ and the diagonal containing faces of dimension $d+1$ is not completely shaded in, then $K'$ has unmatched faces in a single dimension $d$ using the canonical matching.
\end{lem}
\begin{proof}
Since $j<k$, the top row is not completely shaded in and so $K'\neq K$.
This implies that $m_1\neq 0$ and so $M$ is in the $(n-k-m_0)$-th column from the left and not in the top row. 

If $i=n-k-m_0-1$, then every column left of $M$ is completely shaded. 
In this case, $M$ is the only square in $K'$ that contains unmatched faces by Lemma \ref{lem:unmatch}.
Since the square above $M$ is not shaded by Definition \ref{defn:canon}, the diagonal containing $M$ is completely shaded while the diagonal containing the square above $M$ is not. This implies that $M$ contains faces of dimension $d$ as desired.

On the other hand, if $i<n-k-m_0-1$ and no other columns are completely shaded, then consider the columns between the leftmost $i$ columns and the column containing $M$. In each of these columns the row above $M$ is shaded by Definition \ref{defn:canon}. None of these columns are completely shaded in by our assumption and so the highest shaded square in each of these columns contains faces of dimension $d$ by (iii). Therefore the square diagonally above and to the left of $M$ is either under or on the diagonal containing faces of dimension $d$. This implies that $M$ is either under or on the diagonal containing faces of dimension $d$, but we already know that the square above $M$ is not shaded in and so $M$ contains faces of dimension $d$. The result now follows by Lemma \ref{lem:unmatch} which says that the only unmatched faces occur in $M$ and in the highest shaded squares in each column left of $M$ that is not in the top row.
\end{proof}

\begin{prop}
\label{prop:subtypes}
A subcomplex $K'$ of $K$ using the canonical matching has unmatched faces in a single dimension if and only if $K'$ can be described completely by one or more of the following properties:
\begin{enumerate}
  \item [\rm(i)]$K'$ contains every face $F(S)$ of dimension at most $d$, $0\leq d\leq n-2$;
  \item [\rm(ii)]$K'$ contains every face $F(S)$ such that $S(1)\geq j$, $1\leq j\leq k-1$; 
  \item [\rm(iii)]$K'$ contains every face $F(S)$ such that $S(0)\geq i$, $1\leq i\leq n-k-1$. 
\end{enumerate}
Figures \ref{fig:subcomplextypes1} and \ref{fig:subcomplextypes2} show the possible subcomplex diagrams that arise from these $K'$.
\end{prop}
\begin{proof}
If $K'$ is made up of only the vertices of $J(n,k)$ then the lemma is trivially true. 
Next notice that (i) is equivalent to saying that the diagram of $K'$ contains every square under and including the diagonal of squares containing faces of dimension $d$, (ii) is equivalent to saying that the diagram of $K'$ contains every square under and including some row, and (iii) equivalent to saying that the diagram of $K'$ contains every square to the left of and including some column. The reason that $d$ cannot be $n-1$ and that $j$ or $i$ cannot be zero is that otherwise $K'$ would be equal to $K$ which has no unmatched faces. 

First assume that $K'$ has unmatched faces in a single dimension. If $K'$ does not have any rows that are completely shaded then the assertion follows from Lemma \ref{lem:subtypes1}. If $K'$ has precisely the bottom $j$ rows completely shaded where $0<j<k$ then the assertion follows from Lemma \ref{lem:subtypes2}.

Next assume that $K'$ can be described completely by one or more of (i), (ii), and (iii). If $K'$ does not have any rows that are completely shaded then the assertion follows from Lemma \ref{lem:subtypes1b}. If $K'$ has precisely the bottom $j$ rows completely shaded where $0<j<k$ then the assertion follows from Lemma \ref{lem:subtypes2b}.
\end{proof}

\begin{prop}
\label{prop:K'reduced}
If $K'$ is a subcomplex of $K$ with unmatched faces in a single dimension $d$, then
$K'$ has its reduced homology concentrated in degree $d$.
Furthermore, if $K'$ has $u_d$ unmatched faces in dimension $d>0$, then $\widetilde{H}_d(K')\cong\mathbb{Z}^{u_d}$.
\end{prop}
\begin{proof}
If we unpair the empty set with the vertex $v_0$ described in Section \ref{s:Construction},
then Theorem \ref{thm:Forman} (ii) implies that $K'$ is homotopic to a CW complex with one cell of dimension 0, $u_d$ cells of dimension $d$, and zero cells in every other dimension.

When $d\neq 1$ the result is obvious since the $d$-th group in the chain complex is nonzero and has zeros on either side of it. Assume that $d=1$ and let $(C(K')_i)_{i\in\mathbb{Z}}$ be the corresponding cellular chain complex. Also, let $(C(K)_i)_{i\in\mathbb{Z}}$ be the corresponding cellular chain complex of $K$. 
Notice that $H_0(K)\cong\mathbb{Z}$ (and so $K$ is connected) since Theorem \ref{thm:Forman} (ii) implies $K$ is homotopic to a CW complex with one cell of dimension 0 and no other cells. This implies $H_0(K')\cong\mathbb{Z}$ (and so $K'$ is also connected) as well since $C(K')_i=C(K)_i$ when $i=0,1$. 

Next consider the CW complex that $K'$ is homotopic to by Theorem \ref{thm:Forman} (ii):

\begin{center}
$\cdots\longerrightarrow 0 \stackrel{\partial_{3}'}{\longerrightarrow} 0 \stackrel{\partial_{2}'}{\longerrightarrow} \mathbb{Z}^{u_1}\stackrel{\partial_{1}'}{\longerrightarrow} \mathbb{Z} \stackrel{\partial_0'}{\longerrightarrow} 0$.
\end{center}
We see that $\mathrm{im}\ \partial_{2}'=0$ which implies that $H_1(K') \cong \mathrm{ker} \ \partial_1'$ and hence is free abelian. In order to figure out its rank we first recall that 
$\mathbb{Z} \cong H_0(K')$ from above and by definition $H_0(K')=\mathrm{ker} \ \partial_0'/\mathrm{im} \ \partial_1'$.
Therefore $\mathbb{Z} \cong\mathrm{ker} \ \partial_0'/\mathrm{im} \ \partial_1'$ and since $\mathrm{ker} \ \partial_0'=\mathbb{Z}$ we have $\mathrm{im}\ \partial_{1}'=0$.
That means $H_1(K') \cong \mathrm{ker} \ \partial_1' = \mathbb{Z}^{u_1}$ as desired.
\end{proof}


\section{Adding a Single Square to a Subcomplex Diagram}
\label{s:single}
Now that we know each of the subcomplexes described in Proposition \ref{prop:subtypes} has their reduced homology concentrated in a single degree, our goal is to show that every other subcomplex does not. We will start by taking one of the subcomplexes described in Proposition \ref{prop:subtypes} and shading in a specific single square to create a new subcomplex $X$ that has unmatched faces in two dimensions. The following two lemmas give us some conditions that such a square must satisfy. Note that each subcomplex described in this section will always use the canonical matching.

\begin{lem}
\label{lem:X}
Let $K'$ be a subcomplex with unmatched faces in a single dimension and
let $X$ be a subcomplex whose diagram is obtained from the diagram of $K'$ by shading in a single square denoted by $Q$.
If $X$ contains unmatched faces in exactly two dimensions, then $Q$ is not in the top row, the rightmost column, or the lower left square.
\end{lem}
\begin{proof}
Suppose that $K'$ contains every face $F(S)$ of dimension at most $d$, every face $F(S)$ such that $S(1)\geq j$ for some $j$, every face $F(S)$ such that $S(0)\geq i$ for some $i$, and no others.
Since $X$ is a subcomplex, the faces in the squares to the left and below $Q$ must be in $K'$ by Remark \ref{rem:downleft}. If $Q$ is in the top row then $X$ contains every face $F(S)$ of dimension at most $d$, every face $F(S)$ such that $S(1)\geq j$, every face $F(S)$ such that $S(0)\geq i-1$, and no others. In other words, the diagram of $X$ has exactly one more completely shaded column than the diagram of $K'$. Therefore $X$ has unmatched faces in a single dimension by Proposition \ref{prop:subtypes}.

The case when $Q$ is the rightmost column is proven in a similar way except that the diagram of $X$ has exactly one more completely shaded row than the diagram of $K'$.
If $Q$ is the lower left square, then $K'$ must be the subcomplex made up of only the vertices of $J(n,k)$. Therefore $X$ contains every face $F(S)$ of dimension at most $1$ and no others which means that $X$ has unmatched faces in a single dimension by Proposition \ref{prop:subtypes}.
\end{proof}

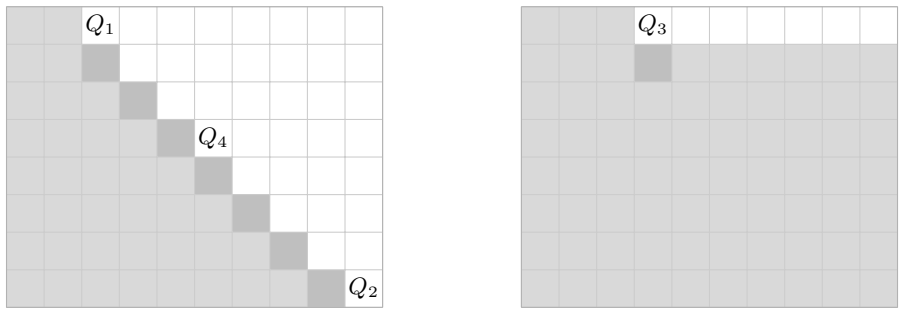
\begin{figure}[htbp]
\def\sca{.5}
\def\split{.4}
\begin{center}
\begin{tikzpicture}[scale=\sca]
\def \length{10}
\def \height{8}
\def \lengthb{9} 
\def \heightb{7} 
\fill [shade] (0,0) rectangle (2,8);
\fill [shade] (2,0) rectangle (3,7);
\fill [shade] (3,0) rectangle (4,6);
\fill [shade] (4,0) rectangle (5,5);
\fill [shade] (5,0) rectangle (6,4);
\fill [shade] (6,0) rectangle (7,3);
\fill [shade] (7,0) rectangle (8,2);
\fill [shade] (8,0) rectangle (9,1);
\foreach \x in {2,...,8}
\fill [shade b] ($(\x,7)-(0,\x-2)$) rectangle ($(\x,6)-(-1,\x-2)$);
\draw [step=1,thin,gray!40] (0,0) grid (\length, \height);
\draw [linee] (0,0) -- (0,\height);
\draw [linee] (0,0) -- (\length,0);
\draw [linee] (\length, \height) -- (0,\height);
\draw [linee] (\length, \height) -- (\length,0);
\footnotesize
\draw (2.5,7.5) node {$Q_1$}; 
\draw (9.5,.5) node {$Q_2$}; 
\draw (5.5,4.5) node {$Q_4$}; 
\end{tikzpicture}
\hspace{.6in}
\begin{tikzpicture}[scale=\sca]
\def \length{10}
\def \height{8}
\def \lengthb{9} 
\def \heightb{7} 
\fill [shade] (0,0) rectangle (3,8);
\fill [shade] (0,0) rectangle (10,7);
\fill [shade b] (3,6) rectangle (4,7);
\draw [step=1,thin,gray!40] (0,0) grid (\length, \height);
\draw [linee] (0,0) -- (0,\height);
\draw [linee] (0,0) -- (\length,0);
\draw [linee] (\length, \height) -- (0,\height);
\draw [linee] (\length, \height) -- (\length,0);
\footnotesize
\draw (3.5,7.5) node {$Q_3$}; 
\end{tikzpicture}
\caption{Possible diagrams for the subcomplex $X$ described in Lemma \ref{lem:X}.}
\label{fig:subcomplextypesX}
\end{center}
\end{figure}

\begin{eg}
Consider $J(18,8)$ and the two subcomplexes given by the diagrams in Figure \ref{fig:subcomplextypesX}. Notice by inspection that if we shade in one of the squares labeled by $Q_1, \ Q_2, \ \mathrm{or} \ Q_3$ then the resulting subcomplex will still have unmatched faces in a single dimension. However if we shade in the square labeled by $Q_4$, then the resulting subcomplex will have unmatched faces in two dimensions as is the case with $X$ in Lemma \ref{lem:X}. 
\end{eg}

\begin{lem}
\label{lem:Xd}
Let $K'$ be a subcomplex with unmatched faces in a single dimension $d$, and let $X$ be a subcomplex whose diagram is obtained from the diagram of $K'$ by shading in a single square denoted by $Q$.
If $X$ contains unmatched faces in exactly two dimensions then $Q$ contains faces of dimension $d+1$ and some of them are unmatched.
\end{lem}
\begin{proof}
Assume first that the diagram of $K'$ does not have any rows that are completely shaded in and so $m_0=m_1=0$. In this case the highest shaded square in each column, unless that square is in the top row, and the rightmost shaded square in the bottom row each contain unmatched faces by Lemma \ref{lem:unmatch}.
If $Q$ contains faces of dimension $d'\neq d+1$ and is not in the bottom row then the square below $Q$ contains faces of dimension $d'-1\neq d$.
Since $X$ is a subcomplex, the square below $Q$ contains faces in $X$ and so by definition, also contain faces in $K'$. However, this would imply that the square below $Q$ is the highest shaded square in its column in $K'$ and so $K'$ contains unmatched faces in dimension $d'-1\neq d$ which is a contradiction.
If $Q$ contains faces of dimension $d'\neq d+1$ and is in the bottom row then the square to the left of $Q$ contains faces of dimension $d'-1\neq d$. Similarly, this leads to a contradiction that the square to the left of $Q$ is the rightmost shaded square in the bottom row in $K'$ and so $K'$ would again contain unmatched faces in dimension $d'-1\neq d$.
We conclude that $Q$ contains faces of dimension $d+1$.

Assume next that the diagram of $K'$ has the bottom $j>0$ rows completely shaded in. 
In this case the highest shaded square in each column left of and including the column containing $M$, unless that square is in the top row, contain unmatched faces by Lemma \ref{lem:unmatch}.
If $Q$ is in one of the columns either containing $M$ or to the right of it, then $Q$ must be the square above $M$ since otherwise the square to the left of $Q$ would not be shaded and that would contradict $X$ being a subcomplex. 
Therefore if $Q$ is not in the bottom row, then the square below $Q$ is the highest shaded square in its column in $K'$.

In either case, $Q$ is not in the top row or rightmost column by Lemma \ref{lem:X}. 
However $Q$ is either the highest shaded square in its column in $X$ or the rightmost shaded square in the bottom row, which is not completely shaded in this case, and so $Q$ contains unmatched faces by Lemma \ref{lem:unmatch}.
\end{proof}

\begin{eg}
Consider $J(18,8)$ and the subcomplex $K'$ given by the diagram in Figure \ref{fig:subcomplextypesXd}. Notice by inspection that the squares labeled by $Q$ are the only ones in which we can add a single square to $K'$ such that this new subcomplex has unmatched faces in two dimensions and every square $Q$ contains faces exactly one dimension higher than the unmatched faces of $K'$.
\end{eg}

\begin{figure}[htbp]
\def\sca{.5}
\def\split{.4}
\begin{center}
\begin{tikzpicture}[scale=\sca]
\def \length{10}
\def \height{8}
\def \lengthb{9} 
\def \heightb{7} 
\fill [shade] (0,0) rectangle (2,8);
\fill [shade] (2,0) rectangle (3,7);
\fill [shade] (3,0) rectangle (4,6);
\fill [shade] (4,0) rectangle (5,5);
\fill [shade] (5,0) rectangle (6,4);
\fill [shade] (6,0) rectangle (7,3);
\fill [shade] (7,0) rectangle (8,3);
\fill [shade] (8,0) rectangle (9,3);
\fill [shade] (9,0) rectangle (10,3);
\foreach \x in {2,...,6}
\fill [shade b] ($(\x,7)-(0,\x-2)$) rectangle ($(\x,6)-(-1,\x-2)$);
\draw [step=1,thin,gray!40] (0,0) grid (\length, \height);
\draw [linee] (0,0) -- (0,\height);
\draw [linee] (0,0) -- (\length,0);
\draw [linee] (\length, \height) -- (0,\height);
\draw [linee] (\length, \height) -- (\length,0);
\footnotesize
\draw (3.5,6.5) node {$Q$}; 
\draw (4.5,5.5) node {$Q$};
\draw (5.5,4.5) node {$Q$};
\draw (6.5,3.5) node {$Q$};
\end{tikzpicture}
\caption{Dimension of faces in $Q$.}
\label{fig:subcomplextypesXd}
\end{center}
\end{figure}
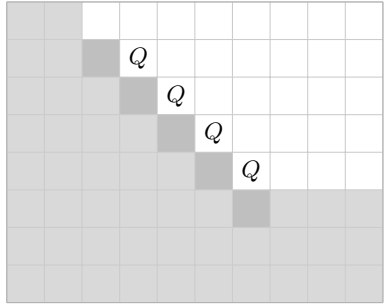

Now that we have some idea what the subcomplex $X$ looks like, we still need to prove that its reduced homology is not concentrated in a single degree. We will go about this by using two more subcomplexes $A$ and $B$ that are obtained from $X$ and do have their reduced homology concentrated in a single degree. The following two lemmas will give us a construction of $A$ and $B$ and a few properties of $A$ and $B$ that follow from this construction.

\begin{defn}
\label{defn:AB}
Let $K'$ be a subcomplex with unmatched faces in a single dimension $d$. Let $X$ be a subcomplex whose diagram is obtained from the diagram of $K'$ by shading in a single square denoted by $Q$ such that $X$ contains unmatched faces in exactly two dimensions. Define $A_{X}$ to be the subcomplex whose diagram is obtained from the diagram of $X$ by adding every square that is either in the column containing $Q$ or in a column to the left of the column containing $Q$. Define $B_{X}$ to be the subcomplex whose diagram is obtained from the diagram of $X$ by adding every square that is either in the row containing $Q$ or in a row below the row containing $Q$. When the subcomplex $X$ is clear from context, we will refer to $A_X$ and $B_X$ as $A$ and $B$ respectively.
\end{defn}

\begin{eg}
Consider the subcomplexes $K'$ and $X_i$ shown in Figure \ref{fig:XAB} below. 
Given these subcomplexes we see the corresponding $A_i$ and $B_i$ as described in Definition \ref{defn:AB}. A black dot has been added to help remind us where the single square $Q$ added to the diagram of $K'$ to obtain the diagram of $X_i$ is.
\end{eg}

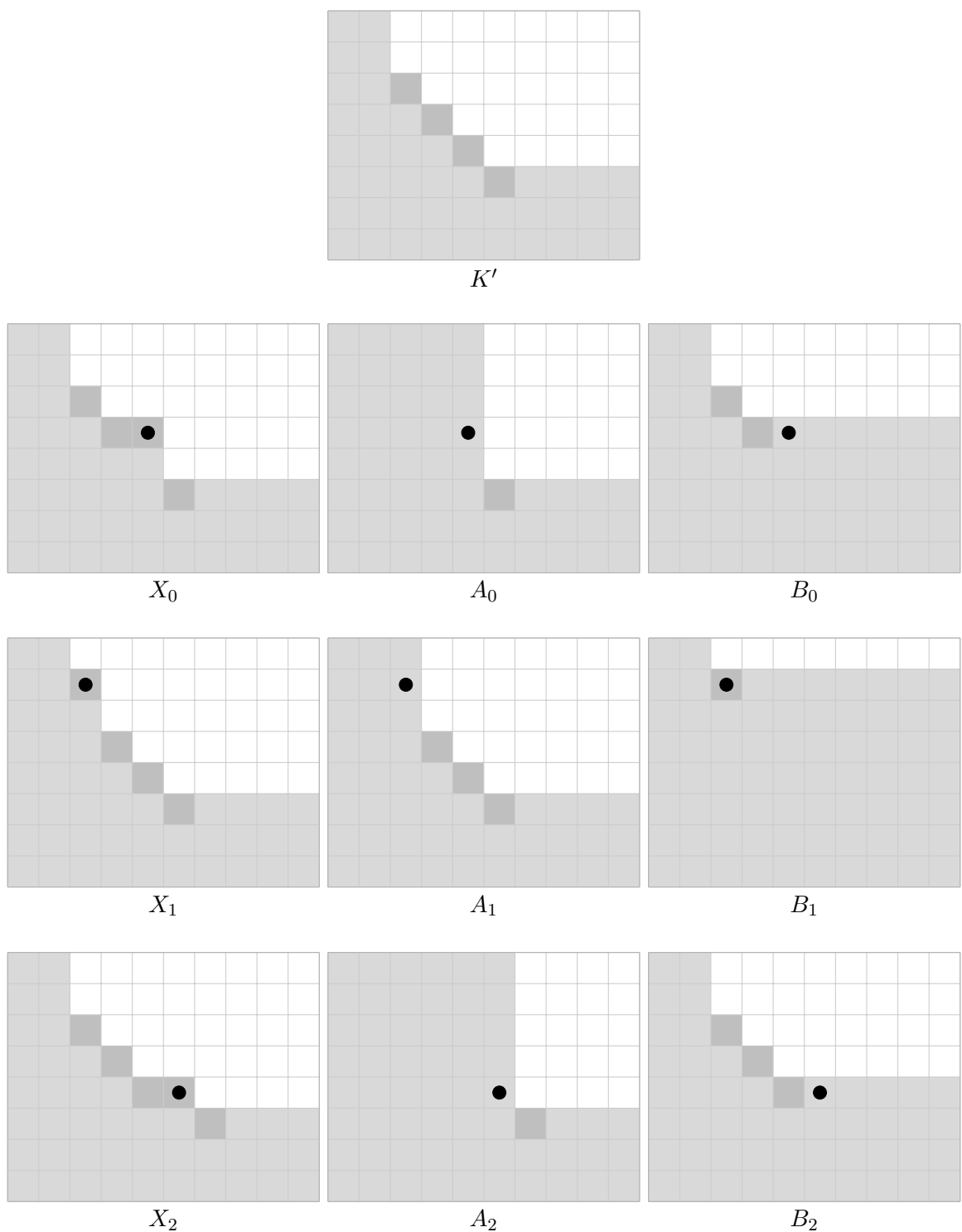
\begin{figure}[htbp]
\def\sca{.5}
\begin{center}
\begin{tikzpicture}[scale=\sca]
\def \length{10}
\def \height{8}
\def \lengthb{9} 
\def \heightb{7} 
\fill [shade] (0,0) rectangle (2,8);
\foreach \x in {2,...,5}
\fill [shade] (\x,0) rectangle ($(\x,8)-(-1,\x)$);
\fill [shade] (6,0) rectangle (10,3);
\foreach \x in {2,...,5}
\fill [shade b] ($(\x,7)-(0,\x)$) rectangle ($(\x,8)-(-1,\x)$);
\draw [step=1,thin,gray!40] (0,0) grid (\length, \height);
\draw [linee] (0,0) -- (0,\height);
\draw [linee] (0,0) -- (\length,0) node [midway, below, black] {$K'$};
\draw [linee] (\length, \height) -- (0,\height);
\draw [linee] (\length, \height) -- (\length,0);
\end{tikzpicture}
\end{center}
\begin{center}
\begin{tikzpicture}[scale=\sca]
\def \length{10}
\def \height{8}
\def \lengthb{9} 
\def \heightb{7} 
\fill [shade] (0,0) rectangle (2,8);
\foreach \x in {2,...,5}
\fill [shade] (\x,0) rectangle ($(\x,8)-(-1,\x)$);
\fill [shade] (6,0) rectangle (10,3);
\foreach \x in {2,3,5}
\fill [shade b] ($(\x,7)-(0,\x)$) rectangle ($(\x,8)-(-1,\x)$);
\fill [shade b] (4,4) rectangle (5,5);
\filldraw [shade c] (4.5,4.5) circle (6pt);
\draw [step=1,thin,gray!40] (0,0) grid (\length, \height);
\draw [linee] (0,0) -- (0,\height);
\draw [linee] (0,0) -- (\length,0) node [midway, below, black] {$X_0$};
\draw [linee] (\length, \height) -- (0,\height);
\draw [linee] (\length, \height) -- (\length,0);
\end{tikzpicture}
\begin{tikzpicture}[scale=\sca]
\def \length{10}
\def \height{8}
\def \lengthb{9} 
\def \heightb{7} 
\fill [shade] (0,0) rectangle (5,8);
\fill [shade] (0,0) rectangle (10,3);
\foreach \x in {5,...,5}
\fill [shade b] ($(\x,7)-(0,\x)$) rectangle ($(\x,8)-(-1,\x)$);
\filldraw [shade c] (4.5,4.5) circle (6pt);
\draw [step=1,thin,gray!40] (0,0) grid (\length, \height);
\draw [linee] (0,0) -- (0,\height);
\draw [linee] (0,0) -- (\length,0) node [midway, below, black] {$A_0$};
\draw [linee] (\length, \height) -- (0,\height);
\draw [linee] (\length, \height) -- (\length,0);
\end{tikzpicture}
\begin{tikzpicture}[scale=\sca]
\def \length{10}
\def \height{8}
\def \lengthb{9} 
\def \heightb{7} 
\fill [shade] (0,0) rectangle (2,8);
\foreach \x in {2,...,5}
\fill [shade] (\x,0) rectangle ($(\x,8)-(-1,\x)$);
\fill [shade] (3,0) rectangle (10,5);
\foreach \x in {2,...,3}
\fill [shade b] ($(\x,7)-(0,\x)$) rectangle ($(\x,8)-(-1,\x)$);
\filldraw [shade c] (4.5,4.5) circle (6pt);
\draw [step=1,thin,gray!40] (0,0) grid (\length, \height);
\draw [linee] (0,0) -- (0,\height);
\draw [linee] (0,0) -- (\length,0) node [midway, below, black] {$B_0$};
\draw [linee] (\length, \height) -- (0,\height);
\draw [linee] (\length, \height) -- (\length,0);
\end{tikzpicture}
\end{center}
%
%
%
%
\begin{center}
\begin{tikzpicture}[scale=\sca]
\def \length{10}
\def \height{8}
\def \lengthb{9} 
\def \heightb{7} 
\fill [shade] (0,0) rectangle (2,8);
\foreach \x in {2,...,5}
\fill [shade] (\x,0) rectangle ($(\x,8)-(-1,\x)$);
\fill [shade] (6,0) rectangle (10,3);
\foreach \x in {3,...,5}
\fill [shade b] ($(\x,7)-(0,\x)$) rectangle ($(\x,8)-(-1,\x)$);
\fill [shade b] (2,6) rectangle (3,7);
\filldraw [shade c] (2.5,6.5) circle (6pt);
\draw [step=1,thin,gray!40] (0,0) grid (\length, \height);
\draw [linee] (0,0) -- (0,\height);
\draw [linee] (0,0) -- (\length,0) node [midway, below, black] {$X_1$};
\draw [linee] (\length, \height) -- (0,\height);
\draw [linee] (\length, \height) -- (\length,0);
\end{tikzpicture}
%
%
\begin{tikzpicture}[scale=\sca]
\def \length{10}
\def \height{8}
\def \lengthb{9} 
\def \heightb{7} 
\fill [shade] (0,0) rectangle (3,8);
\foreach \x in {2,...,5}
\fill [shade] (\x,0) rectangle ($(\x,8)-(-1,\x)$);
\fill [shade] (0,0) rectangle (10,3);
\foreach \x in {3,...,5}
\fill [shade b] ($(\x,7)-(0,\x)$) rectangle ($(\x,8)-(-1,\x)$);
\filldraw [shade c] (2.5,6.5) circle (6pt);
\draw [step=1,thin,gray!40] (0,0) grid (\length, \height);
\draw [linee] (0,0) -- (0,\height);
\draw [linee] (0,0) -- (\length,0) node [midway, below, black] {$A_1$};
\draw [linee] (\length, \height) -- (0,\height);
\draw [linee] (\length, \height) -- (\length,0);
\end{tikzpicture}
%
%
\begin{tikzpicture}[scale=\sca]
\def \length{10}
\def \height{8}
\def \lengthb{9} 
\def \heightb{7} 
\fill [shade] (0,0) rectangle (2,8);
\foreach \x in {2,...,5}
\fill [shade] (\x,0) rectangle ($(\x,8)-(-1,\x)$);
\fill [shade] (2,0) rectangle (10,7);
\foreach \x in {2,...,3}
\fill [shade b] (2,6) rectangle (3,7);
\filldraw [shade c] (2.5,6.5) circle (6pt);
\draw [step=1,thin,gray!40] (0,0) grid (\length, \height);
\draw [linee] (0,0) -- (0,\height);
\draw [linee] (0,0) -- (\length,0) node [midway, below, black] {$B_1$};
\draw [linee] (\length, \height) -- (0,\height);
\draw [linee] (\length, \height) -- (\length,0);
\end{tikzpicture}
\end{center}
%
%
%
%
\begin{center}
\begin{tikzpicture}[scale=\sca]
\def \length{10}
\def \height{8}
\def \lengthb{9} 
\def \heightb{7} 
\fill [shade] (0,0) rectangle (2,8);
\foreach \x in {2,...,5}
\fill [shade] (\x,0) rectangle ($(\x,8)-(-1,\x)$);
\fill [shade] (6,0) rectangle (10,3);
\foreach \x in {2,...,4}
\fill [shade b] ($(\x,7)-(0,\x)$) rectangle ($(\x,8)-(-1,\x)$);
\fill [shade b] (5,3) rectangle (6,4);
\fill [shade b] (6,2) rectangle (7,3);
\filldraw [shade c] (5.5,3.5) circle (6pt);
\draw [step=1,thin,gray!40] (0,0) grid (\length, \height);
\draw [linee] (0,0) -- (0,\height);
\draw [linee] (0,0) -- (\length,0) node [midway, below, black] {$X_2$};
\draw [linee] (\length, \height) -- (0,\height);
\draw [linee] (\length, \height) -- (\length,0);
\end{tikzpicture}
%
%
\begin{tikzpicture}[scale=\sca]
\def \length{10}
\def \height{8}
\def \lengthb{9} 
\def \heightb{7} 
\fill [shade] (0,0) rectangle (6,8);
\foreach \x in {2,...,5}
\fill [shade] (\x,0) rectangle ($(\x,8)-(-1,\x)$);
\fill [shade] (0,0) rectangle (10,3);
\fill [shade b] (6,3) rectangle (7,2);
\filldraw [shade c] (5.5,3.5) circle (6pt);
\draw [step=1,thin,gray!40] (0,0) grid (\length, \height);
\draw [linee] (0,0) -- (0,\height);
\draw [linee] (0,0) -- (\length,0) node [midway, below, black] {$A_2$};
\draw [linee] (\length, \height) -- (0,\height);
\draw [linee] (\length, \height) -- (\length,0);
\end{tikzpicture}
%
%
\begin{tikzpicture}[scale=\sca]
\def \length{10}
\def \height{8}
\def \lengthb{9} 
\def \heightb{7} 
\fill [shade] (0,0) rectangle (2,8);
\foreach \x in {2,...,5}
\fill [shade] (\x,0) rectangle ($(\x,8)-(-1,\x)$);
\fill [shade] (0,0) rectangle (10,4);
\foreach \x in {2,...,4}
\fill [shade b] ($(\x,7)-(0,\x)$) rectangle ($(\x,8)-(-1,\x)$);
\filldraw [shade c] (5.5,3.5) circle (6pt);
\draw [step=1,thin,gray!40] (0,0) grid (\length, \height);
\draw [linee] (0,0) -- (0,\height);
\draw [linee] (0,0) -- (\length,0) node [midway, below, black] {$B_2$};
\draw [linee] (\length, \height) -- (0,\height);
\draw [linee] (\length, \height) -- (\length,0);
\end{tikzpicture}
\caption{Possible diagrams for $A$ and $B$ as described in Definition \ref{defn:AB}.}
\label{fig:XAB}
\end{center}
\end{figure}

\begin{lem}
\label{lem:A}
Let $K'$ be a subcomplex with unmatched faces in a single dimension $d$. Let $X$ be a subcomplex whose diagram is obtained from the diagram of $K'$ by shading in a single square denoted by $Q$ such that $X$ contains unmatched faces in exactly two dimensions.
If $Q$ is in the lowest row in which $K'$ does not contain every square, then the unmatched faces of $A$ are in dimension $d+1$. Otherwise, they are in dimension $d$.
\end{lem}
\begin{proof}
Consider the diagram of $K'$ and assume $Q$ is in the lowest row in which $K'$ does not contain every square.
Notice that there is at least one column to the right of the column containing $Q$ by Lemma \ref{lem:X}.
Consider the diagram of $A$, and note that the entire column that contains $Q$ is shaded in.
If $Q$ is in the bottom row, then $Q$ is the rightmost shaded square in the bottom row and every column containing squares in $A$ is completely shaded. This implies $Q$ contains the only unmatched faces of $A$ by Lemma \ref{lem:unmatch}.
If $Q$ is not in the bottom row, then denote the square diagonally below and to the right of $Q$ by $Q'$. By our assumption $Q'$ is in the highest completely shaded row and every column left of $Q'$ is completely shaded. Therefore $Q'=M_A$ and $Q'$ contains the only unmatched faces of $A$ by Lemma \ref{lem:unmatch}.
Either way the only unmatched faces of are in dimension $d+1$ as desired.

Next assume that there is a row below $Q$ that is not completely shaded in. 
If the diagram of $K'$ has no rows that are completely shaded in, then the rightmost shaded square in the bottom row of $K'$ is still the rightmost shaded square in the bottom row of $A$. Also, the highest shaded squares in each column strictly to the right of $Q$ in the diagram of $K'$ are still the highest shaded squares in each column strictly to the right of $Q$ in the diagram of $A$. 
If the diagram of $K'$ does have rows that are completely shaded in, then $Q$ is not in the row above $M_{K'}$ since the row containing $M_{K'}$ is completely shaded in. 
Therefore $M_A=M_{K'}$ and the only difference between the diagrams of $K'$ and $A$ is that the diagram of $A$ has more columns completely shaded in.
Either way the only unmatched faces of $A$ are also unmatched faces of $K'$. These faces have dimension $d$ as desired.
\end{proof}

\begin{lem}
\label{lem:B}
Let $K'$ be a subcomplex with unmatched faces in a single dimension $d$. Let $X$ be a subcomplex whose diagram is obtained from the diagram of $K'$ by shading in a single square denoted by $Q$ such that $X$ contains unmatched faces in exactly two dimensions.
If $Q$ is in the leftmost column in which $K'$ does not contain every square, then the unmatched faces of $B$ are in dimension $d+1$. Otherwise, they are in dimension $d$.
\end{lem}
\begin{proof}
Recall that $Q$ is not in the top row by Lemma \ref{lem:X}.
This implies that there is a square above $Q$ and it is not in $X$, so by definition, the square above $Q$ is not in $B$ either, and $B\neq K$.

Consider the diagram of $K'$ and assume $Q$ is in the leftmost column that is not completely shaded in.
Now consider the diagram of $B$. 
The entire row that contains $Q$ is shaded in and $Q$ is the leftmost square in this row in which the square above is not shaded.
Therefore $M_B=Q$ and the only unmatched faces of $B$ will be contained in $Q$ by Lemma \ref{lem:unmatch}. These faces have dimension $d+1$ by Lemma \ref{lem:Xd} as desired.

From now on, assume that there is a column left of $Q$ that is not completely shaded in. 
Consider the diagram of $X$. 
Notice that the row containing $Q$ is not completely shaded in since $Q$ cannot be in the rightmost column by Lemma \ref{lem:X}.
This implies that the highest shaded square in every column left of $Q$ contains unmatched faces of dimension $d$ unless it is in the top row by Lemma \ref{lem:unmatch}. Since $Q$ contains faces of dimension $d+1$, the square left of $Q$ contains faces of dimension $d$. Denote the square to the left of $Q$ by $Q_L$. By our assumption, the column containing $Q_L$ is not completely shaded and so $Q_L$ must be the highest shaded square in this column. 

Now consider the diagram of $B$, in which the row containing $Q_L$ and $Q$ is completely shaded in by definition.
Also by definition, the columns left of $Q$ in the diagram of $B$ are the same as those in $X$.
So in the diagram of $B$, the highest shaded square in every column left of $Q$ contains unmatched faces of dimension $d$ unless it is in the top row.
If $Q_L$ is in the leftmost column, then $Q_L=M_B$ by Definition \ref{defn:canon}.

If $Q_L$ is not in the leftmost column, then there exists a shaded square to the left of $Q_L$ containing faces of dimension $d-1$.
It follows that the square to the left of $Q_L$ must not be the highest shaded square in its column.
This implies $Q_L=M_B$ by Definition \ref{defn:canon} as well.

Either way, by Lemma \ref{lem:unmatch}, the only squares in $B$ that contain unmatched faces are $Q_L$ and the highest shaded square in each column left of $Q_L$ that is not in the top row.
These squares contain faces of dimension $d$ as desired.
\end{proof}

\begin{eg}
Consider the subcomplexes $K'$ and $X_i$ shown in Figure \ref{fig:XAB} again. 
Notice that the unmatched faces of $K'$ are all in dimension 8 while the $X_i$ have unmatched faces in dimension 8 and 9.
However the dimensions of the unmatched faces of the $A_i$ and $B_i$ depend on the choice of $Q$.
In $X_0$, $Q$ is in neither the lowest row in which $K'$ does not contain every square nor the leftmost column in which $K'$ does not contain every square. This is why the unmatched faces of $A_0$ and $B_0$ are in dimension 8.
However in $X_1$, $Q$ is in the leftmost column in which $K'$ does not contain every square and we see that the unmatched faces of $B_1$ are in dimension 9. 
Likewise in $X_2$, $Q$ is in the lowest row in which $K'$ does not contain every square and we see that the unmatched faces of $A_2$ are in dimension 9.
\end{eg}

\begin{lem}
\label{lem:ABdim}
Let $K'$ be a subcomplex with unmatched faces in a single dimension $d$. Let $X$ be a subcomplex whose diagram is obtained from the diagram of $K'$ by shading in a single square denoted by $Q$ such that $X$ contains unmatched faces in exactly two dimensions.
If the unmatched faces of $A$ or $B$ are in dimension $d+1$ then the unmatched faces of the other must be in dimension~$d$.
\end{lem}
\begin{proof}
Assume that the unmatched faces of $A$ and $B$ are both in dimension $d+1$ and recall that 
Consider the subcomplex diagram of $X$.

If $Q$ is in the bottom row, then by Lemma \ref{lem:X}, $Q$ cannot also be in the leftmost column. 
Also by Lemma \ref{lem:X}, $Q$ is not in the rightmost column and so the bottom row is not completely shaded, and $M_{X}$ is in the top right corner of the diagram.
By Lemma \ref{lem:B}, $Q$ is in the leftmost column in which $K'$ does not contain every square.
This implies that there exist columns to the left of $Q$, and they are all completely shaded in. 
Since $Q$ is in the bottom row, there are no shaded squares other than $Q$ in the column containing $Q$ or to the right of $Q$ since those squares would also be shaded in the diagram of $K'$, but $Q$ is not in $K'$ and so $K'$ would not be a subcomplex.
It follows that $Q$ is both the rightmost shaded square in the bottom row and the highest shaded square in its column.
Therefore by Lemma \ref{lem:unmatch}, $Q$ is the only square that contains unmatched faces in $X$.
So $X$ only contains unmatched faces in a single dimension, which is a contradiction.

Now suppose that $Q$ is not in the bottom row.
By Lemma \ref{lem:A}, $Q$ is in the lowest row in which $K'$ does not contain every square and this implies that the rows below $Q$ are all completely shaded in.
Notice that the square to the right of $Q$ and the squares above $Q$ are not shaded since those squares would also be shaded in the diagram of $K'$, but $Q$ is not in $K'$ and so $K'$ would not be a subcomplex.
Also, $Q$ is not in the rightmost column which means there exists a square diagonally below and to the right of $Q$. 
Denote this square by $Q'$. 
Since $Q'$ is in the highest completely shaded row and is the leftmost square in that row such that the square above it is not shaded, we have $Q'=M_X$.
By Lemma \ref{lem:B}, $Q$ is in the leftmost column in which $K'$ does not contain every square.
This implies that if there exist columns to the left of $Q$, then they are all completely shaded in. 
So the bottom row of $X$ is completely shaded and $Q$ and $Q'$ are the only highest shaded squares in any column left of and including the column containing $M_X$ that are not in the top row.
Therefore by Lemma \ref{lem:unmatch}, $Q$ and $Q'$ are the only squares that contain unmatched faces in $X$.  
However, this is a contradiction since they both contain faces of the same dimension.

The proof now follows from Lemmas \ref{lem:A} and \ref{lem:B}.
\end{proof}

\begin{eg}
Consider the subcomplexes $X_i$ shown in Figure \ref{fig:ABdim} below. 
In each diagram the square marked with a circle indicates $Q$.
Notice how in each case every column left of $Q$ is completely shaded and every row below $Q$ is completely shaded.
In $X_0$ we see the case when $Q$ is in the bottom row.
In $X_1$ we see the case when $Q$ is in the leftmost column.
In $X_2$ we see the case when $Q$ is in neither the bottom row or the leftmost column.
Also notice that in each case, the unmatched faces of $X_i$ are all in the same dimension.
\end{eg}

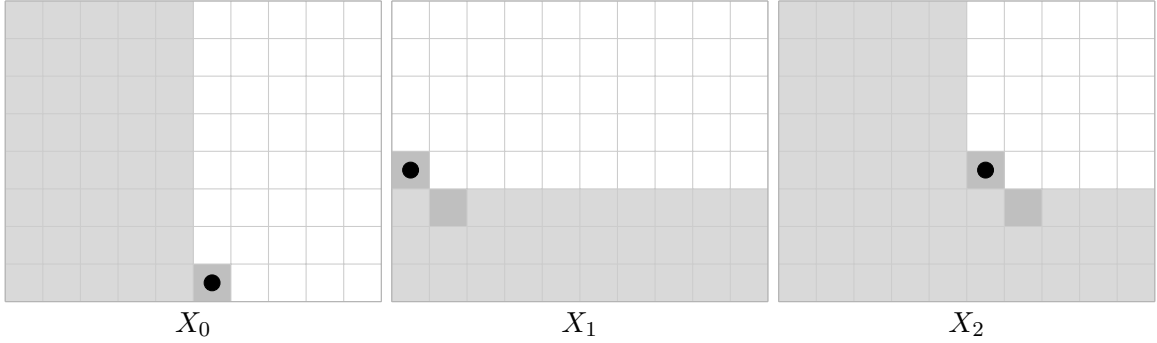
\begin{figure}[htbp]
\def\sca{.5}
\def \length{10}
\def \height{8}
\begin{center}
\begin{tikzpicture}[scale=\sca]
\fill [shade] (0,0) rectangle (5,8);
\fill [shade] (5,0) rectangle (6,1);
\fill [shade b] (5,0) rectangle (6,1);
\filldraw [shade c] (5.5,.5) circle (6pt);
\draw [step=1,thin,gray!40] (0,0) grid (\length, \height);
\draw [linee] (0,0) -- (0,\height);
\draw [linee] (0,0) -- (\length,0) node [midway, below, black] {$X_0$};
\draw [linee] (\length, \height) -- (0,\height);
\draw [linee] (\length, \height) -- (\length,0);
\end{tikzpicture}
\begin{tikzpicture}[scale=\sca]
\fill [shade] (0,0) rectangle (10,3);
\fill [shade b] (0,3) rectangle (1,4);
\filldraw [shade c] (0.5,3.5) circle (6pt);
\fill [shade b] (1,2) rectangle (2,3);
\draw [step=1,thin,gray!40] (0,0) grid (\length, \height);
\draw [linee] (0,0) -- (0,\height);
\draw [linee] (0,0) -- (\length,0) node [midway, below, black] {$X_1$};
\draw [linee] (\length, \height) -- (0,\height);
\draw [linee] (\length, \height) -- (\length,0);
\end{tikzpicture}
\begin{tikzpicture}[scale=\sca]
\fill [shade] (0,0) rectangle (5,8);
\fill [shade] (0,0) rectangle (10,3);
\fill [shade b] (5,3) rectangle (6,4);
\filldraw [shade c] (5.5,3.5) circle (6pt);
\fill [shade b] (6,2) rectangle (7,3);
\draw [step=1,thin,gray!40] (0,0) grid (\length, \height);
\draw [linee] (0,0) -- (0,\height);
\draw [linee] (0,0) -- (\length,0) node [midway, below, black] {$X_2$};
\draw [linee] (\length, \height) -- (0,\height);
\draw [linee] (\length, \height) -- (\length,0);
\end{tikzpicture}
\caption{Possible diagrams arising in proof of Lemma \ref{lem:ABdim}.}
\label{fig:ABdim}
\end{center}
\end{figure}

\begin{lem}
\label{lem:Xnotred}
Let $K'$ be a subcomplex with unmatched faces in a single dimension $d$. Let $X$ be a subcomplex whose diagram is obtained from the diagram of $K'$ by shading in a single square denoted by $Q$ such that $X$ contains unmatched faces in exactly two dimensions.
The reduced homology of $X$ is not concentrated in a single degree. Specifically, we have $\widetilde{H}_d(X)\neq 0$ and $\widetilde{H}_{d+1}(X)\neq0$.
\end{lem}
\begin{proof}
Recall that $Q$ is not in the rightmost column by Lemma \ref{lem:X} and $Q$ contains faces of dimension $d+1$ by Lemma \ref{lem:Xd}.
Also by Lemma \ref{lem:X}, $Q$ is not the lower left square in the diagram of $X$ which implies that $d+1>1$ and hence $d>0$.

Consider the subcomplexes $A$ and $B$ obtained from $X$ as in Definition \ref{defn:AB} and let $A\cup B$ be the subcomplex containing any face contained in either $A$ or $B$. 
In other words, $A\cup B$ is the subcomplex obtained from $X$ by completely shading the row containing $Q$, any row under that, the column containing $Q$, and any column left of that.
Therefore the square to the right of $Q$ in $A\cup B$ is the leftmost shaded square in the highest completely shaded row such that the square above it is not shaded. This implies that the only unmatched faces in $A\cup B$ are in the square to the right of $Q$ by Lemma \ref{lem:unmatch} which contains faces of dimension $d+2$. 
So $\widetilde{H}_l(A\cup B)\cong 0$ if and only if $l\neq d+2$ by Proposition \ref{prop:K'reduced}.

Next consider the diagrams of $A$ and $B$ and notice that the only squares that $A$ and $B$ have in common are the squares of $K'$ and $Q$.
If we let $A\cap B$ be the subcomplex containing only the faces that are in both $A$ and $B$, then $A\cap B=X$.

Since both $A$ and $B$ contain the vertices of $J(n,k)$, $A$ and $B$ have non-empty intersection.
So we can use $X=A\cap B, A, B$ and $A\cup B$ to construct the following long exact sequence of homology groups known as the Mayer--Vietoris sequence for reduced homology:
\begin{center}
$\cdots\longrightarrow \widetilde{H}_{l}(X) \longrightarrow \widetilde{H}_{l}(A)\oplus \widetilde{H}_{l}(B)\longrightarrow \widetilde{H}_{l}(A\cup B)\longrightarrow \widetilde{H}_{l-1}(X)\cdots$ 
\end{center}
\begin{flushright}
$\cdots\longrightarrow \widetilde{H}_0(A\cup B)\longrightarrow 0$
\end{flushright}

Recall that $A$ and $B$ have unmatched faces in a single dimension of either $d$ or $d+1$. Therefore $\widetilde{H}_{l}(A)\oplus \widetilde{H}_{l}(B)\cong 0$ unless $l=d$ or $l=d+1$.
So if we consider the Mayer--Vietoris sequence when $l=d$ or $l=d+1$ we get the following two subsequences:
\begin{center}
$0\longrightarrow \widetilde{H}_{d+2}(A\cup B)\longrightarrow \widetilde{H}_{d+1}(X) \longrightarrow \widetilde{H}_{d+1}(A)\oplus \widetilde{H}_{d+1}(B) \longrightarrow 0$;
\end{center}
\begin{center}
$0\longrightarrow \widetilde{H}_{d}(X)\longrightarrow \widetilde{H}_{d}(A)\oplus \widetilde{H}_{d}(B) \longrightarrow 0$.
\end{center}

Since these are exact sequences, 
$\widetilde{H}_{d}(X)\cong \widetilde{H}_{d}(A)\oplus \widetilde{H}_{d}(B) \neq 0$, since by Lemma \ref{lem:ABdim} at least one of $A$ and $B$ has unmatched faces in dimension $d$. 
Also, if $\widetilde{H}_{d+1}(X)\cong 0$ then $\widetilde{H}_{d+2}(A\cup B)\cong 0$ which is a contradiction. 
Therefore the reduced homology of $X$ is concentrated in degrees $d$ and $d+1$, as desired.
\end{proof}


\section{Classification of Subcomplexes}
\label{s:class2}

We have seen in Lemma \ref{lem:Xnotred} that if we add a single square to a subcomplex containing unmatched faces in a single dimension and we get a subcomplex containing unmatched faces in two dimensions, then the reduced homology of this new subcomplex will not be concentrated in a single degree. Our next goal is to show that if, conversely, we remove some squares, then we end up with the same result.

\begin{lem}
\label{lem:HomContain}
Let $D$ be a CW complex and let $C$ be a subcomplex of $D$. Also let $C$ and $D$ be represented by the following two chain complexes respectively in the standard way:
\begin{center}
$\cdots\longerrightarrow C_{l+1}\stackrel{\partial_{C,l+1}}{\longerrightarrow} C_{l} \stackrel{\partial_{C,l}}{\longerrightarrow} C_{l-1}\stackrel{\partial_{C,l-1}}{\longerrightarrow}\cdots\stackrel{\partial_{C,0}}{\longerrightarrow} C_{-1}\stackrel{\partial_{C,-1}}{\longerrightarrow} 0$;
\end{center}
\begin{center}
$\cdots\longerrightarrow D_{l+1}\stackrel{\partial_{D,l+1}}{\longerrightarrow} D_{l} \stackrel{\partial_{D,l}}{\longerrightarrow} D_{l-1}\stackrel{\partial_{D,l-1}}{\longerrightarrow}\cdots\stackrel{\partial_{D,0}}{\longerrightarrow} D_{-1}\stackrel{\partial_{D,-1}}{\longerrightarrow} 0$.
\end{center}
\begin{enumerate}
  \item [\rm(i)] If $C_{l-1} = D_{l-1}$, $C_{l} \subset D_{l}$, and $H_{l-1}(D)\neq 0$ then $H_{l-1}(C)\neq 0$.
  \item [\rm(ii)] If $C_{l+1} = D_{l+1}$, $C_{l} \subset D_{l}$, and $H_{l}(C)\neq 0$ then $H_{l}(D)\neq 0$.
\end{enumerate}
\end{lem}
\begin{proof}
We first prove part (i) by assuming $C_{l-1} = D_{l-1}$, $C_{l} \subset D_{l}$, and $H_{l-1}(D)\neq 0$ and recalling that $H_{l-1}(D) = $ ker $\partial_{D,l-1} / $im $\partial_{D,l}$. 
Notice that $\mathrm{ker} \ \partial_{C,l-1}=\mathrm{ker} \ \partial_{D,l-1}$ and im $\partial_{C,l}\subset$ im $\partial_{D,l}$.
Therefore, if $H_{l-1}(C) = $ ker $\partial_{C,l-1} / $im $\partial_{C,l}=0$, then $H_{l-1}(D) = $ ker $\partial_{C,l-1} / $im $\partial_{D,l} = 0$ which is a contradiction.

Similarly, we prove part (ii) by assuming $C_{l+1} = D_{l+1}$, $C_{l} \subset D_{l}$, and $H_{l}(C)\neq 0$ and recalling that $H_{l}(C) = $ ker $\partial_{C,l} / $im $\partial_{C,l+1}$.
Notice that $\mathrm{im} \ \partial_{C,l+1}=\mathrm{im} \ \partial_{D,l+1}$ and ker $\partial_{C,l}\subset$ ker $\partial_{D,l}$.
Therefore, if $H_{l}(D) = $ ker $\partial_{D,l} / $im $\partial_{D,l+1}=0$ then $H_{l}(C) = $ ker $\partial_{C,l} / $im $\partial_{D,l+1}=0$ which is a contradiction.
\end{proof}

\begin{lem}
\label{lem:remove}
Let $K'$ be a subcomplex with unmatched faces in a single dimension $d$. 
Let $X$ be a subspace of $K'$ whose diagram is obtained from the diagram of $K'$ by unshading a single square denoted by $Q$.
The space $X$ is a subcomplex if and only if $Q$ is one of the following squares:
\begin{enumerate}
  \item [\rm(i)] the highest square in the rightmost completely shaded column;
  \item [\rm(ii)] the highest shaded square in each column strictly between the rightmost completely shaded column and the column containing $M$; or
  \item [\rm(iii)] the rightmost square in the highest completely shaded row.
\end{enumerate}
\end{lem}
\begin{proof}
Consider the diagram of $K'$. Notice by Remark \ref{rem:downleft} that we can unshade a square $Q$ and end up with another subcomplex if and only if the square above $Q$ and the square to the right of $Q$ are not shaded.

If $Q$ is in any of the fully shaded columns then the only square that does not have either of the squares to the right of it or above it shaded is the top right of these.

Next suppose that $Q$ is in one of the columns strictly between the rightmost completely shaded column and the column containing $M$. The only shaded square in such a column that does not have the square above it shaded is the highest shaded square. 
So consider one of these squares; denote it by $Q'$ and the square to its right by $Q'_R$. 
Since $Q'$ is the highest shaded square in its column, it contains unmatched faces by Lemma \ref{lem:unmatch}.
This implies that $Q'$ contains faces of dimension $d$ and so $Q'_R$ contains faces of dimension $d+1$.
So if $Q'_R$ is shaded, then $Q'_R$ is the highest shaded square in its column by Remark \ref{rem:downleft} and so $K'$ contains unmatched faces of dimension $d+1$ which is a contradiction.
Therefore the square above $Q'$ is not shaded and the square to the right of $Q'$ is not shaded.

Finally, suppose that $Q$ is in one of the columns including or to the right of the column containing $M$. If $K'$ has no completely shaded rows, then $M$ is in the top right corner and there are no shaded squares in these columns. If $K'$ does have a completely shaded row, then by Definition \ref{defn:canon} the highest shaded square in each of these columns is in the same row as $M$. However, the only square in the row containing $M$ that does not have the square to the right of it shaded is the rightmost of them.
\end{proof}

\begin{defn}
\label{defn:Y}
Let $X$ be a subcomplex with unmatched faces in exactly two dimensions $d$ and $d+1$. 
Define $R_X$ to be the set of the highest shaded squares in each column strictly between the rightmost completely shaded column and the column containing $M_X$. 
Define $Y_X$ to be the space whose diagram is obtained from the diagram of $X$ by unshading each square in $R_X$ that contains faces of dimension $d+1$ except the rightmost of them.
Define $Y_{0,X}$ to be the space whose diagram is obtained from the diagram of $Y_X$ by unshading the remaining square in $R_X$ that contains faces of dimension $d+1$.
When the subcomplex $X$ is clear from context, we will refer to $R_X$, $Y_X$, and $Y_{0,X}$ simply as $R$, $Y$, and $Y_{0}$ respectively.
\end{defn}

\begin{eg}
Consider the subcomplexes $X$, $Y$, and $Y_0$ shown in Figure \ref{fig:XY} below. 
Notice how $X$ has unmatched faces in dimensions 7 and 8 and that $X$ does not have a single row that is completely shaded in, so $M$ is in the top right corner.
In these diagrams the squares marked with an $R$ denote every square that is in a column strictly between the rightmost completely shaded column and the column containing $M$ such that this square is the highest shaded square in its column and contains faces of dimension 8. 
To obtain $Y$ from $X$, we unshade each of these except the rightmost.
To obtain $Y_0$ from $X$, we also unshade the rightmost square in $R$.
\end{eg}

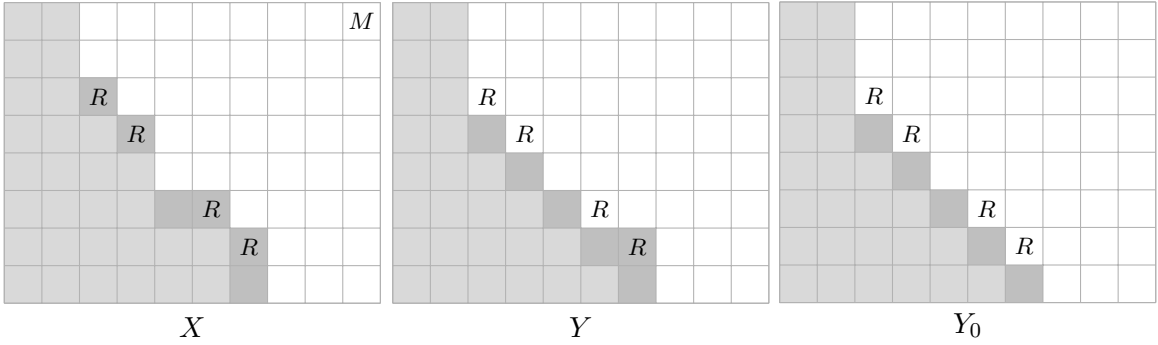
\begin{figure}[htbp]
\def\sca{.5}
\begin{center}
\begin{tikzpicture}[scale=\sca]
\def \length{10}
\def \height{8}
\def \lengthb{9} 
\def \heightb{7} 
\fill [shade] (0,0) rectangle (2,8);
\fill [shade] (2,0) rectangle (3,6);
\fill [shade] (3,0) rectangle (4,5);
\fill [shade] (4,0) rectangle (5,3);
\fill [shade] (5,0) rectangle (6,3);
\fill [shade] (6,0) rectangle (7,2);
\fill [shade b] (2,5) rectangle (3,6);
\fill [shade b] (3,4) rectangle (4,5);
\fill [shade b] (4,2) rectangle (5,3);
\fill [shade b] (5,2) rectangle (6,3);
\fill [shade b] (6,1) rectangle (7,2);
\fill [shade b] (6,0) rectangle (7,1);
\draw [step=1,thin,linee] (0,0) grid (\length, \height);
\draw [linee] (0,0) -- (0,\height);
\draw [linee] (0,0) -- (\length,0); 
\draw (5,-.1) node [below] {$X$};
\draw [linee] (\length, \height) -- (0,\height);
\draw [linee] (\length, \height) -- (\length,0);
\footnotesize
\draw (2.5,5.5) node {$R$}; 
\draw (3.5,4.5) node {$R$};
\draw (5.5,2.5) node {$R$};
\draw (6.5,1.5) node {$R$};
\draw (9.5,7.5) node {$M$};
\end{tikzpicture}
%
%
\begin{tikzpicture}[scale=\sca]
\def \length{10}
\def \height{8}
\def \lengthb{9} 
\def \heightb{7} 
\fill [shade] (0,0) rectangle (2,8);
\fill [shade] (2,0) rectangle (3,5);
\fill [shade] (3,0) rectangle (4,4);
\fill [shade] (4,0) rectangle (5,3);
\fill [shade] (5,0) rectangle (6,2);
\fill [shade] (6,0) rectangle (7,2);
\fill [shade b] (2,4) rectangle (3,5);
\fill [shade b] (3,3) rectangle (4,4);
\fill [shade b] (4,2) rectangle (5,3);
\fill [shade b] (5,1) rectangle (6,2);
\fill [shade b] (6,1) rectangle (7,2);
\fill [shade b] (6,0) rectangle (7,1);
\draw [step=1,thin,linee] (0,0) grid (\length, \height);
\draw [linee] (0,0) -- (0,\height);
\draw [linee] (0,0) -- (\length,0); 
\draw (5,-.1) node [below] {$Y$};
\draw [linee] (\length, \height) -- (0,\height);
\draw [linee] (\length, \height) -- (\length,0);
\footnotesize
\draw (2.5,5.5) node {$R$}; 
\draw (3.5,4.5) node {$R$};
\draw (5.5,2.5) node {$R$};
\draw (6.5,1.5) node {$R$};
\end{tikzpicture}
%
%
\begin{tikzpicture}[scale=\sca]
\def \length{10}
\def \height{8}
\def \lengthb{9} 
\def \heightb{7} 
\fill [shade] (0,0) rectangle (2,8);
\fill [shade] (2,0) rectangle (3,5);
\fill [shade] (3,0) rectangle (4,4);
\fill [shade] (4,0) rectangle (5,3);
\fill [shade] (5,0) rectangle (6,2);
\fill [shade] (6,0) rectangle (7,1);
\fill [shade b] (2,4) rectangle (3,5);
\fill [shade b] (3,3) rectangle (4,4);
\fill [shade b] (4,2) rectangle (5,3);
\fill [shade b] (5,1) rectangle (6,2);
\fill [shade b] (6,0) rectangle (7,1);
\draw [step=1,thin,linee] (0,0) grid (\length, \height);
\draw [linee] (0,0) -- (0,\height);
\draw [linee] (0,0) -- (\length,0); 
\draw (5,0) node [below] {$Y_0$};
\draw [linee] (\length, \height) -- (0,\height);
\draw [linee] (\length, \height) -- (\length,0);
\footnotesize
\draw (2.5,5.5) node {$R$}; 
\draw (3.5,4.5) node {$R$};
\draw (5.5,2.5) node {$R$};
\draw (6.5,1.5) node {$R$};
\end{tikzpicture}
\caption{Diagrams for $Y$ and $Y_0$.}
\label{fig:XY}
\end{center}
\end{figure}

In Lemmas \ref{lem:Ysub}--\ref{lem:Y2} we will show that $Y$ and $Y_{0}$ are each subcomplexes and that $\widetilde{H}_{d}(Y)\neq 0$ and $\widetilde{H}_{d+1}(Y)\neq0$.

\begin{lem}
\label{lem:Ysub}
If $X$ is a subcomplex with unmatched faces in exactly two dimensions $d$ and $d+1$, then $Y$ and $Y_0$ are both subcomplexes.
\end{lem}
\begin{proof}
This is immediate from Remark \ref{rem:downleft}.
\end{proof}

\begin{lem}
\label{lem:Y1}
Let $X$ be a subcomplex with unmatched faces in exactly two dimensions, $d$ and $d+1$. 
If the diagram of $X$ does not have any completely shaded rows, then $\widetilde{H}_{d}(Y)\neq 0$, and $\widetilde{H}_{d+1}(Y)\neq0$.
\end{lem}
\begin{proof}
Notice that $M$ is in the top right corner since the diagram of $X$ does not have any completely shaded rows.
Let $Q$ denote the rightmost square in $R$ that contains faces of dimension $d+1$.
Since $X$ does not have any completely shaded rows, Lemma \ref{lem:unmatch} shows that the only squares containing unmatched faces are those in $R$ and the rightmost shaded square in the bottom row, which might also be in $R$.
So every square in $R$ and the rightmost shaded square in the bottom row contain faces of dimension $d$ or $d+1$.
However, every square in $R$ that contains faces of dimension $d+1$ is unshaded in the diagram of $Y_0$, so the highest shaded square in every column to the right of the rightmost completely shaded column in $Y_0$ contains faces of dimension $d$. 

Denote the rightmost shaded square in the bottom row of $X$ by $Q_B$.
Suppose first that $Q_B$ contains faces of dimension $d+1$.
If the entire column containing $Q_B$ is shaded in, then the diagram of $X$ is made up of several completely shaded columns and nothing else.
So by Proposition \ref{prop:subtypes}, $X$ contains unmatched faces in a single dimension, which is a contradiction.
If the column containing $Q_B$ is not entirely shaded in and $Q_B$ is not the highest shaded square in this column, then by Lemma \ref{lem:unmatch} $X$ contains unmatched faces in a dimension strictly greater than $d+1$, which is also a contradiction.
Therefore $Q_B$ must be the highest shaded square in its column and hence $Q_B=Q$.
So in this case, $Q$ is both the rightmost shaded square in the bottom row and the highest shaded square in its column. 
This implies that there exists a square $Q'$ in $R$ that is in a column to the left of $Q$ and contains faces of dimension $d$ since $X$ has unmatched faces in dimensions $d$ and $d+1$. 
The column containing $Q'$ remains unchanged while going from $X$ to $Y$, and since $Y$ contains both $Q$ and $Q'$, $Y$ contains unmatched faces of dimension $d$ and $d+1$.
Also in this case, the rightmost shaded square in the bottom row of $Y_0$ is the square to the left of $Q$, which contains faces of dimension $d$. 

Suppose now that $Q_B$ contains unmatched faces of dimension $d$, then $Q_B$ is also the rightmost shaded square in the bottom row of $Y$ and $Y_0$, since we only unshade squares containing faces of dimension $d+1$ to get from $X$ to either of them. 
This implies that $Y$ contains unmatched faces of dimension $d$, specifically in $Q_B$, and contains unmatched faces of dimension $d+1$ in $Q$.

Either way, $Y$ contains unmatched faces in exactly two dimensions, $d$ and $d+1$, $Y_0$ contains unmatched faces in exactly one dimension, $d$, and the diagram of $Y$ can be obtained from the diagram of $Y_0$ by shading the single square $Q$.
The assertion now follows by applying Lemma \ref{lem:Xnotred}.
\end{proof}

\begin{lem}
\label{lem:Y2}
Let $X$ be a subcomplex with unmatched faces in exactly two dimensions $d$ and $d+1$. 
If the diagram of $X$ has at least one completely shaded row, then $\widetilde{H}_{d}(Y)\neq 0$, and $\widetilde{H}_{d+1}(Y)\neq0$.
\end{lem}
\begin{proof}
Let $Q$ denote the rightmost square in $R$ that contains faces of dimension $d+1$.
Since $X$ has at least one shaded row, the only squares containing unmatched faces are those in $R$ and the square $M_X$ by Lemma \ref{lem:unmatch}.
So every square in $R$ and $M_X$ contains faces of dimension $d$ or $d+1$.
However, every square in $R$ that contains faces of dimension $d+1$ becomes unshaded in the diagram of $Y_0$ and so the highest shaded square in each column strictly between the rightmost completely shaded column and the column containing $M_X$ in $Y_0$ contains faces of dimension $d$. 

Denote the square to the left of $M_X$ by $M_L$ and the square diagonally above and to the left of $M_X$ by $M_{UL}$.
Recall that the square above $M_X$ is not shaded, but $M_L$ is by Definition \ref{defn:canon}.
If $M_X$ contains faces of dimension $d+1$ then $M_{UL}$ also contains faces of dimension $d+1$ and so $M_{UL}$ must be $Q$.
This implies that there exists a square $Q'$ in $R$ that is in a column to the left of $Q$ and contains faces of dimension $d$, since $X$ has unmatched faces in dimensions $d$ and $d+1$.
The column containing $Q'$ remains unchanged while going from $X$ to $Y$ and since $Y$ contains both $Q$ and $Q'$, $Y$ contains unmatched faces of dimension $d$ and $d+1$.
Also in this case, unshading $Q$ means that $M_{Y_0}$ is $M_L$, which contains faces of dimension $d$. 

If $M_X$ contains faces of dimension $d$, then $M_{UL}$ contains faces in dimension $d$ and so it remains shaded in $Y$ and $Y_0$.
Therefore we have $M_{Y}=M_{Y_0}=M_X$ in this case.
This implies that $Y$ contains unmatched faces of dimension $d$, specifically in $M_Y$, and contains unmatched faces of dimension $d+1$ in $Q$.

Either way, $Y$ contains unmatched faces in exactly two dimensions, $d$ and $d+1$, $Y_0$ contains unmatched faces in exactly one dimension, $d$, and the diagram of $Y$ can be obtained from the diagram of $Y_0$ by shading the single square $Q$.
The assertion now follows by applying Lemma \ref{lem:Xnotred}.
\end{proof}

\begin{defn}
\label{defn:Y'}
Let $X$ be a subcomplex with unmatched faces in exactly two dimensions $d$ and $d+1$. 
Let $Q$ denote the leftmost square in the diagram of $X$ containing unmatched faces of dimension $d$.
If the square above $Q$ is shaded, then define $Y'_X$ to be the space whose diagram is obtained from the diagram of $X$ by completely shading in every column left of and including the column containing $Q$.
If the square above $Q$ is not shaded, then define $Y'_X$ to be the space whose diagram is obtained from the diagram of $X$ by completely shading in every column strictly to the left of the column containing $Q$ and every row below and including the row containing $Q$.
When the subcomplex $X$ is clear from context, we will refer to $Y'_X$ simply as $Y'$.
\end{defn}

\begin{rem}
\label{rem:nottop}
If a shaded square is in the top row and contains unmatched faces, then by Lemma \ref{lem:unmatch} the top row must be the same as the bottom row. In this case the columns are only one square tall. It follows that the only subcomplexes that can occur are those in which several columns are completely shaded in. These diagrams all represent subcomplexes whose unmatched faces are in a single dimension by Proposition \ref{prop:subtypes}.
Therefore any subcomplex with unmatched faces in two or more dimensions cannot have a square in the top row containing unmatched faces.
\end{rem}

\begin{eg}
Consider the subcomplexes shown in Figure \ref{fig:XY'} below. 
Notice how $X$ and $X'$ have unmatched faces in dimensions 7 and 8.
The square denoted by $Q$ is the leftmost square in $X$ and $X'$ that contains unmatched faces of dimension 7. 
To obtain $Y'_{X}$ from $X$, we shade in every column strictly left of the column containing $Q$ and every row below and including the row containing $Q$.
To obtain $Y'_{X'}$ from $X'$, we shade in every column left of and including the column containing $Q$.
\end{eg}

\begin{figure}[htbp]
\def\sca{.49}
\begin{center}
\begin{tikzpicture}[scale=\sca]
\def \length{10}
\def \height{8}
\def \lengthb{9} 
\def \heightb{7} 
\fill [shade] (0,0) rectangle (2,8);
\fill [shade] (2,0) rectangle (3,6);
\fill [shade] (3,0) rectangle (4,5);
\fill [shade] (4,0) rectangle (5,3);
\fill [shade] (5,0) rectangle (6,3);
\fill [shade] (6,0) rectangle (7,2);
\fill [shade b] (2,5) rectangle (3,6);
\fill [shade b] (3,4) rectangle (4,5);
\fill [shade b] (4,2) rectangle (5,3);
\fill [shade b] (5,2) rectangle (6,3);
\fill [shade b] (6,1) rectangle (7,2);
\fill [shade b] (6,0) rectangle (7,1);
\draw [step=1,thin,linee] (0,0) grid (\length, \height);
\draw [linee] (0,0) -- (0,\height);
\draw [linee] (0,0) -- (\length,0); 
\draw (5,-.25) node [below] {$X$};
\draw [linee] (\length, \height) -- (0,\height);
\draw [linee] (\length, \height) -- (\length,0);
\footnotesize
\draw (4.5,2.5) node {$Q$};
\end{tikzpicture}
\hspace{.4in}
\begin{tikzpicture}[scale=\sca]
\def \length{10}
\def \height{8}
\def \lengthb{9} 
\def \heightb{7} 
\fill [shade] (0,0) rectangle (4,8);
\fill [shade] (0,0) rectangle (10,3);
\fill [shade b] (4,2) rectangle (5,3);
\draw [step=1,thin,linee] (0,0) grid (\length, \height);
\draw [linee] (0,0) -- (0,\height);
\draw [linee] (0,0) -- (\length,0); 
\draw (5,0) node [below] {$Y'_X$};
\draw [linee] (\length, \height) -- (0,\height);
\draw [linee] (\length, \height) -- (\length,0);
\footnotesize
\draw (4.5,2.5) node {$Q$};
\end{tikzpicture}

\vspace{.2in}

\begin{tikzpicture}[scale=\sca]
\def \length{10}
\def \height{8}
\def \lengthb{9} 
\def \heightb{7} 
\fill [shade] (0,0) rectangle (2,8);
\fill [shade] (2,0) rectangle (3,6);
\fill [shade] (3,0) rectangle (4,5);
\fill [shade] (4,0) rectangle (5,4);
\fill [shade] (5,0) rectangle (6,3);
\fill [shade] (6,0) rectangle (7,2);
\fill [shade b] (2,5) rectangle (3,6);
\fill [shade b] (3,4) rectangle (4,5);
\fill [shade b] (4,3) rectangle (5,4);
\fill [shade b] (5,2) rectangle (6,3);
\fill [shade b] (6,1) rectangle (7,2);
\fill [shade b] (6,0) rectangle (7,1);
\draw [step=1,thin,linee] (0,0) grid (\length, \height);
\draw [linee] (0,0) -- (0,\height);
\draw [linee] (0,0) -- (\length,0); 
\draw (5,-.25) node [below] {$X'$};
\draw [linee] (\length, \height) -- (0,\height);
\draw [linee] (\length, \height) -- (\length,0);
\footnotesize
\draw (6.5,0.5) node {$Q$};
\end{tikzpicture}
\hspace{.4in}
\begin{tikzpicture}[scale=\sca]
\def \length{10}
\def \height{8}
\def \lengthb{9} 
\def \heightb{7} 
\fill [shade] (0,0) rectangle (7,8);
\fill [shade b] (6,0) rectangle (7,1);
\draw [step=1,thin,linee] (0,0) grid (\length, \height);
\draw [linee] (0,0) -- (0,\height);
\draw [linee] (0,0) -- (\length,0); 
\draw (5,0) node [below] {$Y'_{X'}$};
\draw [linee] (\length, \height) -- (0,\height);
\draw [linee] (\length, \height) -- (\length,0);
\footnotesize
\draw (6.5,0.5) node {$Q$};
\end{tikzpicture}
\caption{Diagrams for $Y'$.}
\label{fig:XY'}
\end{center}
\end{figure}
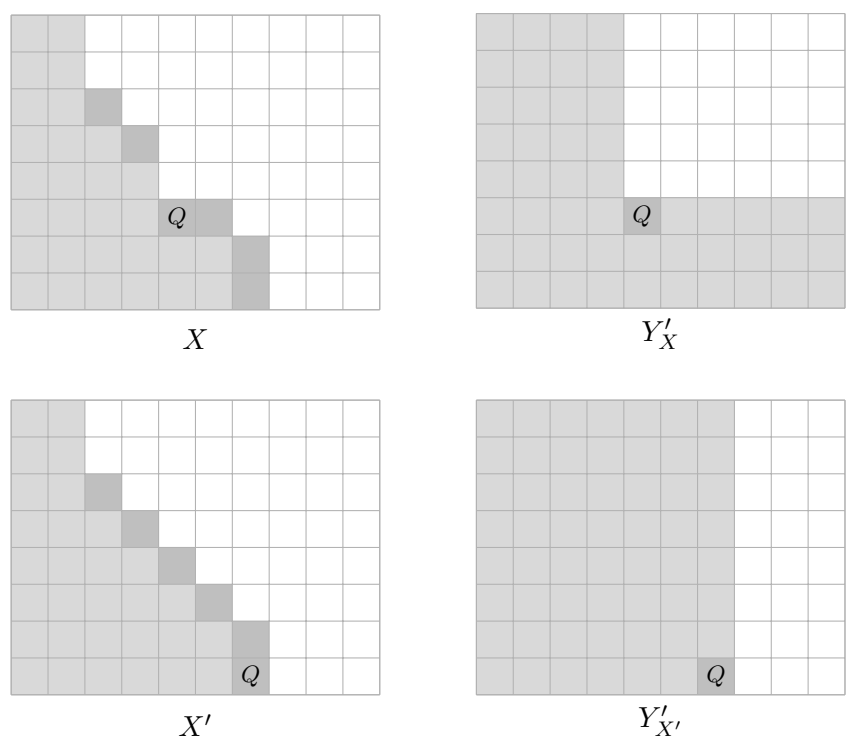

\begin{lem}
\label{lem:Y'1}
Let $X$ be a subcomplex with unmatched faces in exactly two dimensions $d$ and $d+1$. 
Let $Q$ be the leftmost square in the diagram of $X$ containing unmatched faces of dimension $d$.
If $Q$ is in the bottom row and the square above it is shaded, then $Y'$ is a subcomplex that only contains unmatched faces in dimension $d$ and $\widetilde{H}_d(Y')\neq0$. 
\end{lem}
\begin{proof}
If $Q$ is in the rightmost column or the square to its right is shaded then $Q$ does not contain any unmatched faces by Lemma \ref{lem:unmatch}, since the square above $Q$ is shaded by assumption. 
So $Q$ is the rightmost shaded square in the bottom row of $X$ and is not in the rightmost column. 
This implies that there are no shaded squares in any column to the right of the column containing $Q$.
Therefore, in the diagram of $Y'$, every square in a column left of and including the column containing $Q$ is shaded in and they are the only squares that are shaded in.
This implies that $Y'$ is a subcomplex by Remark \ref{rem:downleft}. 

Since $Q$ is still the rightmost shaded square in the bottom row of $Y$, $Q$ contains unmatched faces in $Y'$.
Also, $Y'$ is one of the subcomplexes described in Proposition \ref{prop:subtypes}, so $Y'$ contains unmatched faces in a single dimension $d$. Therefore the reduced homology of $Y'$ is concentrated in degree $d$ by Proposition \ref{prop:K'reduced}.
\end{proof}

\begin{lem}
\label{lem:Y'2}
Let $X$ be a subcomplex with unmatched faces in exactly two dimensions $d$ and $d+1$. 
Let $Q$ be the leftmost square in the diagram of $X$ containing unmatched faces of dimension $d$.
If the square above $Q$ is not shaded, then $Y'$ is a subcomplex that only contains unmatched faces in dimension $d$ and $\widetilde{H}_d(Y')\neq0$. 
\end{lem}
\begin{proof}
Consider the diagram of $Y'$.
By definition, every column strictly to the left of the column containing $Q$ and every row below and including the row containing $Q$ is shaded in.
Also, these are the only squares that are shaded in since the square above $Q$ is not shaded.
This implies that $M_{Y'}=Q$.
It follows that $Y'$ is one of the subcomplexes described in Proposition \ref{prop:subtypes} and that $Y'$ contains unmatched faces in a single dimension $d$.
Therefore the reduced homology of $Y'$ is concentrated in degree $d$ by Proposition \ref{prop:K'reduced}.
\end{proof}

\begin{lem}
\label{lem:alld}
Let $X$ be any subcomplex containing unmatched faces in at least one dimension. If $d$ is the lowest dimension in which $X$ contains unmatched faces, then $X$ contains every face of dimension $d$.
\end{lem}
\begin{proof}
Suppose there exists a square $Q'$ containing faces of dimension $d$ that is not in $X$. 
If there are no shaded squares in the column containing $Q'$, then the bottom row is not completely shaded.
So by Lemma \ref{lem:unmatch}, the rightmost shaded square in the bottom row contains unmatched faces.
However, this square contains faces of dimension strictly less than $d$, which is a contradiction.
So from now on, assume that there are shaded squares in the column containing $Q'$ and let $Q_0$ be the highest of them. 

If the column containing $Q'$ is to the left of or is the column containing $M$, then $Q_0$ contains unmatched faces by Lemma \ref{lem:unmatch}. However $Q_0$ is below $Q'$ and so $Q_0$ contains faces of dimension strictly less than $d$, which is a contradiction. 

If the column containing $Q'$ is to the right of the column containing $M$, then $M$ is not in the rightmost column. 
It follows from Definition \ref{defn:canon} that there is at least one completely shaded row in $X$ and that $Q_0$ is in the same row as $M$. 
Also, $M$ contains unmatched faces by Lemma \ref{lem:unmatch}.
However $M$ would then be strictly left of and below $Q'$ and so $X$ would contain unmatched faces of dimension strictly less than $d$, which is a contradiction.
\end{proof}

\begin{eg}
Consider the subcomplexes $X_1$ and $X_2$ shown in Figure \ref{fig:XY'} again. 
It can be seen that 7 is the lowest dimension in which either of these diagrams contains unmatched faces and that they both contain every face of dimension 7.
\end{eg}

\begin{lem}
\label{lem:X2}
Let $X$ be a subcomplex of $K$ with unmatched faces in exactly two dimensions $d_1$ and $d_2$. Then the reduced homology of $X$ is not concentrated in a single degree.
Specifically, we have $\widetilde{H}_{d_1}(X)\neq 0$ and $\widetilde{H}_{d_2}(X)\neq0$.
\end{lem}
\begin{proof}
First note that if the unmatched faces of $X$ are not in consecutive dimensions then the statement is true by Theorem \ref{thm:Forman} (ii). From now on assume that the unmatched faces of $X$ are in dimensions $d$ and $d+1$. 

Consider the diagram of $X$ and let $R$ denote the set of the highest shaded squares in each column strictly between the rightmost completely shaded column and the column containing $M_X$.
Recall from Definition \ref{defn:Y} that $Y$ is the subcomplex whose diagram is obtained from the diagram of $X$ by unshading the squares in $R$ containing unmatched faces of dimension $d+1$ except the rightmost of them.
By Lemmas \ref{lem:Y1} and \ref{lem:Y2}, we have $\widetilde{H}_{d+1}(Y)\neq0$.
It follows that (a) $Y$ is a subcomplex of $X$, (b) $X$ and $Y$ contain the same faces of dimension $d+2$ (since we only unshaded squares containing faces of dimension $d+1$), and (c) every face of dimension $d+1$ in $Y$ is also in $X$.
Therefore we can apply Lemma \ref{lem:HomContain} (ii) with $C=Y$, $D=X$, and $l=d+1$ to get that $\widetilde{H}_{d+1}(X)\neq0$ as well.

Let $Q$ be the leftmost square in the diagram of $X$ containing unmatched faces of dimension $d$.
Let $Y'$ be the subcomplex whose diagram is obtained from the diagram of $X$ as in Definition \ref{defn:Y'}.
By Lemmas \ref{lem:Y'1} and \ref{lem:Y'2}, $Y'$ only contains unmatched faces in dimension $d$ and $\widetilde{H}_{d}(Y')\neq0$.
Since $X$ is a subcomplex of $Y'$, every face of dimension $d+1$ in $X$ is also in $Y'$, and by Lemma \ref{lem:alld} both $X$ and $Y'$ contain every face of dimension $d$.
Therefore we can apply Lemma \ref{lem:HomContain} (i)  with $C=X$, $D=Y'$, and $l=d+1$ to get that $\widetilde{H}_{d}(X)\neq0$ as well.
\end{proof}

\begin{defn}
\label{defn:Xi}
Let $X$ be a subcomplex of $K$ with unmatched faces in at least three dimensions.
Let the smallest and largest dimensions in which $X$ contains unmatched faces be denoted by $d_1$ and $d_2$ respectively. 
Define $X_1$ to be the subcomplex whose diagram is obtained from the diagram of $X$ by unshading every square containing faces of dimension strictly greater than $d_1+1$.
Define $X_2$ to be the subcomplex whose diagram is obtained from the diagram of $X$ by shading every square containing faces of dimension strictly less than $d_2$.
\end{defn}

\begin{eg}
Consider the subcomplex $X$ shown in Figure \ref{fig:X3}.
Notice that $X$ has unmatched faces in dimensions 5--8 and so $d_1=5$ and $d_2=8$.
To obtain $X_1$ from $X$, we unshade every square containing faces of dimension strictly greater than 6.
To obtain $X_2$ from $X$, we shade every square containing faces of dimension strictly less than 8.
The outline in the diagrams of $X_1$ and $X_2$ denotes $X$.
\end{eg}

\begin{figure}[htbp]
\def\sca{.5}
\begin{center}
\begin{tikzpicture}[scale=\sca]
\def \length{10}
\def \height{8}
\def \lengthb{9} 
\def \heightb{7} 
\fill [shade] (0,0) rectangle (2,7);
\fill [shade] (2,0) rectangle (3,5);
\fill [shade] (3,0) rectangle (4,3);
\fill [shade] (4,0) rectangle (10,1);
\fill [shade b] (0,6) rectangle (1,7);
\fill [shade b] (1,6) rectangle (2,7);
\fill [shade b] (2,4) rectangle (3,5);
\fill [shade b] (3,2) rectangle (4,3);
\fill [shade b] (4,0) rectangle (5,1);
\draw [step=1,thin,linee] (0,0) grid (\length, \height);
\draw [linee] (0,0) -- (0,\height);
\draw [linee] (0,0) -- (\length,0); 
\draw (5,-.12) node [below] {$X$};
\draw [linee] (\length, \height) -- (0,\height);
\draw [linee] (\length, \height) -- (\length,0);
\end{tikzpicture}
%
%
\begin{tikzpicture}[scale=\sca]
\def \length{10}
\def \height{8}
\def \lengthb{9} 
\def \heightb{7} 
\fill [shade] (0,0) rectangle (1,6);
\fill [shade] (1,0) rectangle (2,5);
\fill [shade] (2,0) rectangle (3,4);
\fill [shade] (3,0) rectangle (4,3);
\fill [shade] (4,0) rectangle (6,1);
\fill [shade b] (0,5) rectangle (1,6);
\fill [shade b] (1,4) rectangle (2,5);
\fill [shade b] (2,3) rectangle (3,4);
\fill [shade b] (3,2) rectangle (4,3);
\fill [shade b] (4,0) rectangle (5,1);
\fill [shade b] (5,0) rectangle (6,1);
\draw [step=1,thin,linee] (0,0) grid (\length, \height);
\draw [linee] (0,0) -- (0,\height);
\draw [linee] (0,0) -- (\length,0); 
\draw (5,0) node [below] {$X_1$};
\draw [linee] (\length, \height) -- (0,\height);
\draw [linee] (\length, \height) -- (\length,0);
\draw [thick] (0,0)--(0,7)--(2,7)--(2,5)--(3,5)--(3,3)--(4,3)--(4,1)--(10,1)--(10,0)--(0,0);
\end{tikzpicture}
%
%
\begin{tikzpicture}[scale=\sca]
\def \length{10}
\def \height{8}
\def \lengthb{9} 
\def \heightb{7} 
\fill [shade] (0,0) rectangle (2,7);
\fill [shade] (2,0) rectangle (3,5);
\fill [shade] (3,0) rectangle (4,4);
\fill [shade] (4,0) rectangle (5,3);
\fill [shade] (5,0) rectangle (6,2);
\fill [shade] (6,0) rectangle (10,1);
\fill [shade b] (0,6) rectangle (1,7);
\fill [shade b] (1,6) rectangle (2,7);
\fill [shade b] (2,4) rectangle (3,5);
\fill [shade b] (3,3) rectangle (4,4);
\fill [shade b] (4,2) rectangle (5,3);
\fill [shade b] (5,1) rectangle (6,2);
\fill [shade b] (6,0) rectangle (7,1);
\draw [step=1,thin,linee] (0,0) grid (\length, \height);
\draw [linee] (0,0) -- (0,\height);
\draw [linee] (0,0) -- (\length,0); 
\draw (5,0) node [below] {$X_2$};
\draw [linee] (\length, \height) -- (0,\height);
\draw [linee] (\length, \height) -- (\length,0);
\draw [thick] (0,0)--(0,7)--(2,7)--(2,5)--(3,5)--(3,3)--(4,3)--(4,1)--(10,1)--(10,0)--(0,0);
\end{tikzpicture}
\caption{Diagrams for $X_1$ and $X_2$.}
\label{fig:X3}
\end{center}
\end{figure}
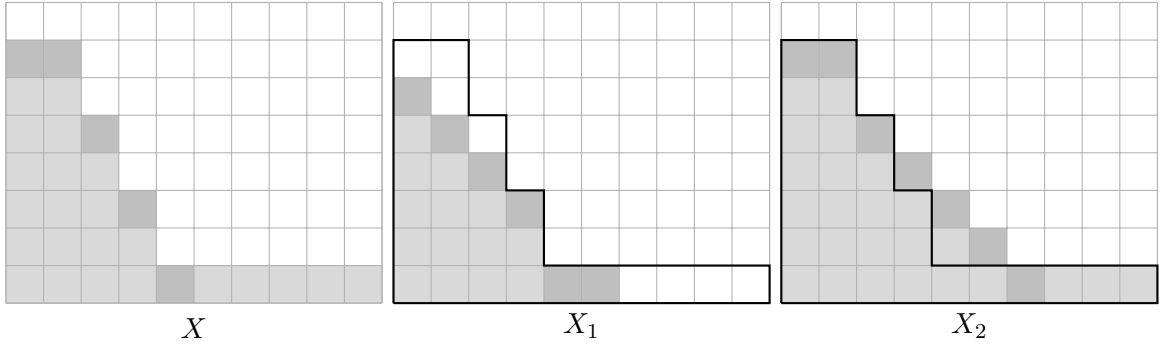

\begin{lem}
\label{lem:X3}
Let $X$ be a subcomplex of $K$ with unmatched faces in at least three dimensions. The reduced homology of $X$ is not concentrated in a single degree.
\end{lem}
\begin{proof}
Let the smallest and largest dimensions in which $X$ contains unmatched faces be denoted by $d_1$ and $d_2$ respectively. 
First consider the subcomplex $X_1$ of Definition \ref{defn:Xi} and let $Q_1$ be a square in $X$ containing unmatched faces in dimension $d_1$. 
By Lemma \ref{lem:unmatch}, either $Q_1$ is the highest shaded square in its column and not in the top row or $Q_1$ is the rightmost shaded square in the bottom row and not in the rightmost column.

Assume that $Q_1$ is the highest shaded square in its column and not in the top row.
This implies that $Q_1$ is still the highest shaded square in its column in $X_1$ since we only unshaded squares containing faces of dimension strictly greater than $d_1+1$ to get $X_1$.
Notice that $X$ contains every face of dimension $d_1$ by Lemma \ref{lem:alld} and so the square above and to the left of $Q_1$ is shaded in both $X$ and $X_1$. 
Therefore $M_{X_1}$ cannot be in a column strictly to the left of $Q_1$ by Definition \ref{defn:canon}, and $Q_1$ contains unmatched faces in $X_1$ by Lemma \ref{lem:unmatch}. 

Now assume that $Q_1$ is the rightmost shaded square in the bottom row of $X$ but not in the rightmost column. 
This implies that $Q_1$ is still the rightmost shaded square in the bottom row of $X_1$, since we only unshaded squares containing faces of dimension strictly greater than $d_1+1$ to get $X_1$.
By Lemma \ref{lem:unmatch}, $Q_1$ contains unmatched faces in $X_1$.

So either way, $Q_1$ contains unmatched faces in $X_1$.
Since $X$ contains every face of dimension $d_1$, $X_1$ does as well and so $X_1$ cannot have any unmatched faces of dimension strictly less than $d_1$.
Also, by definition, $X_1$ does not have any faces of dimension strictly greater than $d_1+1$.
This implies that either $X_1$ contains unmatched faces only in dimension $d_1$ and $\widetilde{H}_{d_1}(X_1)\neq0$ by Proposition \ref{prop:K'reduced},
or $X_1$ contains unmatched faces in both dimensions $d_1$ and $d_1+1$ and $\widetilde{H}_{d_1}(X_1)\neq0$ by Lemma \ref{lem:X2}.
So now we can apply Lemma \ref{lem:HomContain} (ii) with $C=X_1$, $D=X$, and $l=d_1$ to see that $\widetilde{H}_{d_1}(X)\neq0$ as well.

Now consider the subcomplex $X_2$ of Definition \ref{defn:Xi}.
Let $Q_2$ be a square containing unmatched faces in $X_2$ of dimension strictly greater than $d_2-1$ so that $Q_2$ is in $X$. Since $Q_2$ contains faces of dimension strictly greater than $d_2-1$, then so does every other square above and/or to the right of $Q_2$.
It follows that each of these squares is shaded in $X_2$ if and only if it is shaded in $X$. 
This implies that $Q_2$ is the highest shaded square in its column in $X_2$ if and only if it has the same property in $X$. 
The same is true if $Q_2$ is in a completely shaded row or if $Q_2$ is the rightmost shaded square in the bottom row.
Therefore by Lemma \ref{lem:unmatch}, $Q_2$ is unmatched in $X_2$ if and only if it is unmatched in $X$.

This implies that $X_2$ can only contain unmatched faces of dimension $d_2-1$ or $d_2$, and since $X$ has at least one square containing unmatched faces of dimension $d_2$ then so does $X_2$. So either $X_2$ contains unmatched faces only in dimension $d_2$ and $\widetilde{H}_{d_2}(X_2)\neq0$ by Proposition \ref{prop:K'reduced}, or $X_2$ contains unmatched faces in both dimensions $d_2-1$ and $d_2$ and so $\widetilde{H}_{d_2}(X_2)\neq0$ by Lemma \ref{lem:X2}.
Now we can apply Lemma \ref{lem:HomContain} (i) with $C=X$, $D=X_2$, and $l=d_2+1$ to see that $\widetilde{H}_{d_2}(X)\neq0$ as well.
We conclude that $\widetilde{H}_{d_1}(X)\neq0$ and $\widetilde{H}_{d_2}(X)\neq0$ as desired.
\end{proof}

We are now ready to prove the main result of Chapter \ref{chap:class}.
Recall that Proposition \ref{prop:subtypes} says that a subcomplex $K'$ of $K$ using the canonical matching has unmatched faces in a single dimension if and only if $K'$ can be described completely by one or more of the following properties:
\begin{enumerate}
  \item [\rm(i)]$K'$ contains every face $F(S)$ of dimension at most $d$, $0\leq d\leq n-2$;
  \item [\rm(ii)]$K'$ contains every face $F(S)$ such that $S(1)\geq j$, $1\leq j\leq k-1$;
  \item [\rm(iii)]$K'$ contains every face $F(S)$ such that $S(0)\geq i$, $1\leq i\leq n-k-1$. 
\end{enumerate}
The next theorem shows that these subcomplexes are the only ones whose reduced homology groups are concentrated in a single degree.

\begin{thm}
\label{thm:class}
Let $K'$ be an $S_n$-invariant subcomplex of $K$ that uses the canonical matching.
\begin{enumerate}
  \item [\rm(i)] The subcomplex $K'$ has its reduced homology groups concentrated in a single degree if and only if $K'$ has unmatched faces in a single dimension. Furthermore, if $K'$ contains $u_d$ unmatched faces in dimension $d>0$ and has no other unmatched faces then the $d$-th reduced homology group is free abelian of rank $u_d$.
  \item [\rm(ii)] If $K'$ has $u_d$ unmatched faces of dimension $d>0$ and zero unmatched faces in every other dimension, then $K'$ is homotopy equivalent to a wedge of $u_d$ $d$-spheres.
\end{enumerate}
\end{thm}
\begin{proof}
To prove part (i), notice that if $K'$ has unmatched faces in a single dimension, then this follows from Proposition \ref{prop:K'reduced}.
If $K'$ has unmatched faces exactly two dimensions, then this follows from Lemma \ref{lem:X2}.
Finally, if $K'$ has unmatched faces in at least three dimensions, then this follows from Lemma \ref{lem:X3}.

To prove part (ii), notice that $K'$ is homotopic to a CW complex with exactly one cell of dimension zero, $u_d$ cells of dimension $d$, and zero cells in every other dimension by Theorem \ref{thm:Forman} (ii). The result now follows from Proposition \ref{prop:wedge}.
\end{proof}

\chapter{Homology Representations}

\section{Betti Numbers}

Consider a $S_n$-invariant subcomplex $K'$ that has unmatched faces in a single dimension $d$.
Recall from Theorem \ref{thm:class} that the $d$-th reduced homology group of $K'$ is free abelian of rank equal to the number of unmatched faces.
The first goal of this chapter is to compute the rank of $H_d(K')$. 
This rank is called the \textbf{$d$-th Betti number} of $K'$. 

Each of the homology groups obtained in this way supports a natural action of the symmetric group $S_n$ because $S_n$ acts by continuous transformations. Therefore the $d$-th Betti number of these subcomplexes is also the dimension of a corresponding homology representation.

\begin{defn}
Let the triple $(d, i, j)$, where $0<d<n-1$, $0<i\leq n-k$, and $0<j\leq k$, be the subcomplex $K'$ containing the faces of dimension at most $d$, every face $F(S)$ such that $S(1)\geq j$ or $S(0)\geq i$, and no other faces.
If there is more than one triple that describes $K'$, then we will only consider the triple in which $d$ is maximal and $i\ \mathrm{and} \ j$ are minimal.
\end{defn}

\begin{eg}
Consider the subcomplex of $J(18, 8)$ given by the diagram in Figure \ref{fig:triple} below. 
Recall that the square labeled by $(i,j)$ contains faces of dimension $18-i-j-1$ and notice every square labeled by $(i,j)$ in which $i+j\geq7$ is contained in $K'$. 
This implies that every face of dimension $d \leq 18-7-1=10$ is contained in $K'$.
However the diagonal of squares containing every face of dimension $11$ is not entirely in $K'$ since, for example, the square labeled by $(6,0)$ is not in $K'$. 
We can also see that every square labeled by $(i,j)$ where $i\geq 7$ or $j\geq 4$ is contained in $K'$ and therefore $K'=(10,7,4)$.
Notice that $(10,8,4)$ also describes this subcomplex, but by convention we will only associate $(10,7,4)$ with $K'$.
\end{eg}

\begin{figure}[htbp]
\def\sca{1}
\def\split{.4}
\begin{center}
\begin{tikzpicture}[scale=\sca]
\def \length{10}
\def \height{8}
\def \lengthb{9} 
\def \heightb{7} 
\fill [shade] (0,0) rectangle (3,8);
\fill [shade] (3,7) rectangle (4,0);
\fill [shade] (4,6) rectangle (5,0);
\fill [shade] (5,5) rectangle (6,0);
\fill [shade] (6,0) rectangle (10,4);
\fill [shade b] (3,6) rectangle (4,7);
\fill [shade b] (4,5) rectangle (5,6);
\fill [shade b] (5,4) rectangle (6,5);
\fill [shade b] (6,3) rectangle (7,4);
\foreach \x in {0,...,\lengthb}
\foreach \y in {0,...,\heightb}
{
\draw [xshift=.5cm, yshift=.5cm] ($(\length,\height) - (\x,\y) -(1,1)$) node {\x,\y}; 
}
\normalsize
\draw [step=1,thin,gray!40] (0,0) grid (\length, \height);
\draw [linee] (0,0) -- (0,\height);
\draw [linee] (0,0) -- (\length,0); 
\draw [linee] (\length, \height) -- (0,\height);
\draw [linee] (\length, \height) -- (\length,0);
\end{tikzpicture}
\caption{Example of the subcomplex $(10,7,4)$.}
\label{fig:triple}
\end{center}
\end{figure}

\begin{lem}
\label{lem:donly}
The subcomplex given by $(d,i,j)$ has unmatched faces in dimension $d$ only.
\end{lem}
\begin{proof}
We know that $(d,i,j)$ has unmatched faces in a single dimension $d_0$ by Proposition \ref{prop:subtypes}, so we just need to show that $d=d_0$.

Suppose first that there are no rows in the diagram of $(d,i,j)$ that are completely shaded. By Lemma \ref{lem:unmatch}, the only unmatched faces are in the highest shaded squares in each column, unless that square is in the top row, and in the rightmost shaded square in the bottom row. 
Let $Q$ denote the rightmost shaded square in the bottom row.
Since $Q$ contains unmatched faces, it must contain faces of dimension $d_0$.
Also notice that the square to the right of $Q$ is not shaded and so $(d,i,j)$ does not contain every face of dimension $d_0+1$. However $(d,i,j)$ does contain every face of dimension $d_0$ since the highest shaded square in the column containing $Q$ and any column to its left is either in the top row or contains faces of dimension $d_0$. Therefore $d=d_0$ as desired.

The case when there are rows in the diagram that are completely shaded in is proven similarly except that we let $Q$ denote the square containing the faces $F(S)$ such that $S(0)=m_0$ and $S(1)=m_1$. 
\end{proof}

\begin{lem}
\label{lem:label}
Given a subcomplex $(d,i,j)$, let $a_0=\mathrm{max}\{n-d-k,\ n-d-1-j\}$ and $b_0=n-d-1-a_0$. The set of labels of the squares containing unmatched faces is given by
$$\{(a_0,b_0)\} \cup \{(a,b) \ | \ a+b=n-d-1,\ a_0+1\leq a\leq i-1\}.$$
In particular, the square labeled by $(a_0, b_0)$ is the rightmost square containing unmatched faces.
\end{lem}
\begin{proof}
Note that $n-a_0-b_0-1=d$, and so the square labeled by $(a_0,b_0)$ contains faces of dimension $d$.
Consider the subcomplex diagram for $(d,i,j)$ and suppose first that there are no rows that are completely shaded.
In this case, we have $j=k$, which implies that $a_0=n-d-k$ and $b_0=k-1$. 
By Lemma \ref{lem:unmatch}, the rightmost shaded square in the bottom row contains unmatched faces, the highest shaded square in each column that is not in the top row (if there exist any) contains unmatched faces, and no other square contains unmatched faces. 
Let $Q$ denote the rightmost shaded square in the bottom row.
If $Q$ is labeled by $(a,b)$ then $b=k-1$ since $Q$ is in the bottom row.
By Remark \ref{rem:subcomp} we have $n-a-b-1=d$ and thus $a=n-d-k$. 
Therefore $Q$ is the square labeled by $(a_0,b_0)$.

Suppose next that there are rows that are completely shaded.
In this case, the only squares containing unmatched faces are the highest shaded square in each column left of and including the column containing $M$ unless that square is in the top row.
By Definition \ref{defn:canon}, $M$ is in the highest completely shaded row and has label $(m_0,m_1)$.  
Note that $0<j$ by definition, so the top row is not completely shaded in and $0<m_1$.
It follows that $j=m_1<k$.
By Lemma \ref{lem:donly}, $M$ contains unmatched faces in dimension $d$.
This implies that $n-m_0-m_1-1 = d$ or that $m_0 = n-d-1-j$.
Therefore $a_0 = m_0$ and $b_0=m_1$, or in other words, $M$ is the square labeled by $(a_0,b_0)$.

Now suppose we are in either case and let $(a,b)$ be the label of the highest shaded square in some column that is not in the top row. 
Recall that the square labeled by $(a,b)$ contains faces of dimension $n-a-b-1=d$. 
This implies that $a+b=n-d-1$.
Also note that if the square labeled by $(a,b)$ is in the same column as the square labeled by $(a_0,b_0)$ then we must have $(a,b)=(a_0,b_0)$ since otherwise they would contain faces of different dimensions.
The remaining squares containing unmatched faces are precisely those that are both left of the column containing the square labeled by $(a_0,b_0)$, which means that $a_0+1\leq a$, and strictly right of the rightmost completely shaded column, which means that $a\leq i-1$, as desired.
\end{proof}

\begin{eg}
Consider the subcomplex $(3,7,8)$ shown in Figure \ref{fig:label} below. In this case we have $n-d-k=18-3-8=7$ and $n-d-1-j=18-3-1-8=6$. This implies that $a_0=7$ and $b_0=7$. Also note that $a_0+1=8$ and $i-1=6$, and so there is no such $a$ that satisfies $a_0+1\leq a\leq i-1$. Therefore the square labeled by $(a_0, b_0)$ is the only square containing unmatched faces. This example shows that the set
$$\{(a,b) \ | \ a+b=n-d-1,\ a_0\leq a\leq i-1\}$$
does not contain all of the labels of squares containing unmatched faces.
\end{eg}

\begin{figure}[htbp]
\def\sca{1}
\def\split{.4}
\begin{center}
\begin{tikzpicture}[scale=\sca]
\def \length{10}
\def \height{8}
\def \lengthb{9} 
\def \heightb{7} 
\fill [shade] (0,0) rectangle (3,8);
\fill [shade b] (2,0) rectangle (3,1);
\foreach \x in {0,...,\lengthb}
\foreach \y in {0,...,\heightb}
{
\draw [xshift=.5cm, yshift=.5cm] ($(\length,\height) - (\x,\y) -(1,1)$) node {\x,\y}; 
}
\normalsize
\draw [step=1,thin,gray!40] (0,0) grid (\length, \height);
\draw [linee] (0,0) -- (0,\height);
\draw [linee] (0,0) -- (\length,0); 
\draw [linee] (\length, \height) -- (0,\height);
\draw [linee] (\length, \height) -- (\length,0);
\end{tikzpicture}
\caption{Unmatched faces of $(3,7,8)$.}
\label{fig:label}
\end{center}
\end{figure}

\begin{lem}
\label{lem:c1}
Consider a subcomplex $(d,i,j)$. If $j=k$, then:
\begin{enumerate}
  \item [\rm(i)] $a_0>0=m_0$;
  \item [\rm(ii)] the square labeled by $(a_0,b_0)$ has unmatched faces of type 3 and there are $$\displaystyle\sum_{c=0}^{k-1}\binom{n-1-c}{a_0-1}\binom{d+k-1-c}{d}$$ of them;
  \item [\rm(iii)] if $a_0<i$, then square labeled by $(a_0,b_0)$ also has unmatched faces of type 1 and there are
$$\binom{n}{a_0}\binom{n-a_0-1}{b_0-1}$$ of them;
  \item [\rm(iv)] the faces in parts (ii) and (iii) are all of the unmatched faces in the square labeled by $(a_0,b_0)$.
\end{enumerate}
\end{lem}
\begin{proof}
Let $Q$ denote the square labeled by $(a_0,b_0)$ and let $F(S)$ be any face in $Q$.
We know that $Q$ is the rightmost square containing unmatched faces by Lemma \ref{lem:label}.
Since $j=k$, the diagram of $(d,i,j)$ does not have any rows that are completely shaded in. 
By Lemma \ref{lem:unmatch}, the rightmost shaded square in the bottom row contains unmatched faces, it is in the rightmost column containing any shaded squares, and no other square in this column can contain unmatched faces since that would contradict $(d,i,j)$ having unmatched faces in a single dimension.
So the rightmost shaded square in the bottom row is $Q$.

Recall that the only types of faces that are matched with faces of higher dimension, and hence the only types of faces in $Q$ that can be unmatched, are types 1, 3, 5, 7, and 9. 
Since $j=k$, we have $a_0=n-d-k$ and $b_0=k-1$ by definition, which implies that $S(0)=n-d-k\leq n-k-1$ and $S(1)=k-1$.
Since there are no rows that are completely shaded in, we have $m_0=m_1=0$ and $a_0>0$, and so $S(0)=a_0>0=m_0$, proving (i). 
Therefore if $Q$ contains unmatched faces, then they must be of types 1 or 3.
The square $Q$ will always contain unmatched faces of type 3 since the square to its right is not shaded.
However, $Q$ will only contain unmatched faces of type 1 when the column containing $Q$ is not completely shaded in, which means that $a_0<i$.
This proves (iv).

Counting the number of faces of type 3 in $Q$ is equivalent to counting the number of sequences of length $n$ made up of $a_0$ 0's, $b_0$ 1's, and $d+1$ *'s (since $d=n-a_0-b_0-1$ by Remark \ref{rem:subcomp}) such that there is not a 1 to the right of the rightmost *, but there is a 0 to the left of the leftmost *.
This implies that the sequence must start with $1\cdots 10$ where there are $c$ 1's and $0\leq c\leq b_0=k-1$.
At this point we can place the remaining $a_0-1$ 0's into any of the remaining $n-1-c$ spots in the sequence.
Now we are forced to put a * into the rightmost available spot in the sequence since otherwise there would be a 1 to the right of the rightmost *.
All that is left to do is place the remaining $d$ *'s into any of the remaining $n-1-c-(a_0-1)-1=d+k-1-c$ spots in the sequence since this fixes where all the 1's must be. Summing over all possible $c$'s proves (ii).

Counting the faces of type 1, if any exist, is done similarly except that we only require a 1 to the right of the rightmost *.
First we place all of the $a_0$ 0's into any of the $n$ spots in the sequence.
Now we are forced to put a 1 in the rightmost available spot in the sequence since otherwise there would not be a 1 to the right of the rightmost *.
All that is left to do is place the remaining $(b_0-1)$ 1's into any of the remaining $n-a_0-1$ spots in the sequence since this fixes where all the *'s must be, proving (iii). 
\end{proof}

\begin{lem}
\label{lem:c2}
Consider a subcomplex $(d,i,j)$. If $j<k$, then $a_0=m_0$, the unmatched faces in the square labeled by $(a_0,b_0)$ are all of type 1, and there are
$$\displaystyle\sum_{c=0}^{a_0}\binom{n-1-c}{j-1}\binom{d+a_0-c}{d}$$
of them.
\end{lem}
\begin{proof}
Let $Q$ denote the square labeled by $(a_0,b_0)$ and let $F(S)$ be any face in $Q$.
We know that $Q$ is the rightmost square containing unmatched faces by Lemma \ref{lem:label}, since if $(a,b)$ is the label of any other square containing unmatched faces, then $a_0<a$. 
Since $j<k$, the diagram of $(d,i,j)$ does have some rows that are completely shaded in. 
By Lemma \ref{lem:unmatch}, the only squares containing unmatched faces are the highest shaded square in each column left of and including the column containing $M$ unless that square is in the top row.
This implies that $Q$ must be $M$.

If $Q$ contains unmatched faces, then they must be matched with the square above $Q$ since by Definition \ref{defn:canon}, $M$ is either in the rightmost column or the square to the right of $M$ is shaded.
However only faces of type 1 are matched with the square above and more specifically, the unmatched faces are all of type 1(c) since $Q=M$.

Counting the number of faces of type 1(c) in $M$ is equivalent to counting the number of sequences of length $n$ made up of $a_0$ 0's, $b_0$ 1's, and $d+1$ *'s (since $d=n-a_0-b_0-1$ by Remark \ref{rem:subcomp}) such that there is a 1 to the right of the rightmost *, but there is not a 0 to the left of the leftmost *.
This implies that the sequence ends with $10\cdots0$ where there are $c$ 0's and $0\leq c \leq a_0$.
Since $j<k$, we have $a_0=n-d-1-j$ and $b_0=j$.
At this point we can place the remaining $(j-1)$ 1's into any of the remaining $n-1-c$ spots in the sequence.
Now we are forced to put a * into the leftmost available spot in the sequence since otherwise there would be a 0 to the left of the leftmost *.
All that is left to do is place the remaining $d$ *'s into any of the remaining $n-1-c-(j-1)-1=d+a_0-c$ spots in the sequence since this fixes where all the 0's must be. Summing over all possible $c$'s gives the result.
\end{proof}

\begin{lem}
\label{lem:c3}
Consider a subcomplex $(d,i,j)$. Let $(a,b)$ be the label of a square containing unmatched faces in $(d,i,j)$ such that $(a,b)\neq(a_0,b_0)$. Then the unmatched faces in this square are all of type 1 and there are
$$\binom{n}{a}\binom{n-a-1}{b-1}$$
of them.
\end{lem}
\begin{proof}
Let $Q$ denote a square containing unmatched faces and labeled by $(a,b)$ such that $(a,b)\neq(a_0,b_0)$, and let $F(S)$ be any face in $Q$.
Notice that $a>a_0$ by Lemma \ref{lem:label}, and that $a_0\geq m_0$ by Lemmas \ref{lem:c1} and \ref{lem:c2}.
It follows that $S(0)>m_0$.
We also have $$b = n-d-1-a < n-d-1-a_0 \leq n-d-1-(n-d-k) = k-1,$$ and so $S(1) < k-1$.
This implies that $Q$ can only contain faces of types 1 and 2, but faces of type 2 are matched with faces of lower dimension and so the unmatched faces of $Q$ are all of type 1 as desired.

Counting the number of faces of type 1 in $Q=M$ is equivalent to counting the number of sequences of length $n$ made up of $a$ 0's, $b$ 1's, and $d+1$ *'s (since $d=n-a-b-1$ by Remark \ref{rem:subcomp}) such that there is a 1 to the right of the rightmost *. 
First we place all of the $a$ 0's into any of the $n$ spots in the sequence.
Now we are forced to put a 1 in the rightmost available spot in the sequence since otherwise there would not be a 1 to the right of the rightmost *.
All that is left to do is place the remaining $(b-1)$ 1's into any of the remaining $n-a-1$ spots in the sequence since this fixes where all the *'s must be.
\end{proof}

\begin{thm}
\label{thm:Betti}
Let $(d,i,j)$ be a subcomplex of $K$ using the canonical matching. 
Let $|(d,i,j)|$ denote the number of unmatched faces in dimension $d$ and let $b_a=n-d-a-1$.
\begin{enumerate}
  \item [\rm(i)] If $j=k$, then $$|(d,i,j)|=\displaystyle\sum_{a=a_0}^{i-1}\binom{n}{a}\binom{n-a-1}{b_a-1} + \displaystyle\sum_{c=0}^{k-1}\binom{n-1-c}{a_0-1}\binom{d+k-1-c}{d}.$$
  \item [\rm(ii)] If $j<k$, then $$|(d,i,j)|=\displaystyle\sum_{a=a_0+1}^{i-1}\binom{n}{a}\binom{n-a-1}{b_a-1} + \displaystyle\sum_{c=0}^{a_0}\binom{n-1-c}{j-1}\binom{d+a_0-c}{d}.$$
\end{enumerate}
\end{thm}
\begin{proof}
By Lemma \ref{lem:label} the labels of the squares containing unmatched faces are given by the set
$$\{(a_0,b_0)\} \cup \{(a,b) \ | \ a+b=n-d-1,\ a_0+1\leq a\leq i-1\}.$$
By Lemma \ref{lem:c3}, the squares labeled by the subset $\{(a,b) \ | \ a+b=n-d-1,\ a_0+1\leq a\leq i-1\}$
each contribute
$$\binom{n}{a}\binom{n-a-1}{b_a-1}$$
unmatched faces, for a total of
$$\displaystyle\sum_{a=a_0+1}^{i-1}\binom{n}{a}\binom{n-a-1}{b_a-1}$$
unmatched faces. 
This accounts for the first sum in (i), except for the first term, and the entire first sum in (ii). 

If $j=k$, then by Lemma \ref{lem:c1}, the square labeled by $(a_0,b_0)$ contributes  
$$\displaystyle\sum_{c=0}^{k-1}\binom{n-1-c}{a_0-1}\binom{d+k-1-c}{d}$$
more unmatched faces, plus another 
$$\binom{n}{a_0}\binom{n-a_0-1}{b_0-1}$$
if $a_0<i$. 
This accounts for the second sum in (i) and the remaining term of the first sum in (i) respectively.

If $j<k$, then by Lemma \ref{lem:c2}, the square labeled by $(a_0,b_0)$ contributes  
$$\displaystyle\sum_{c=0}^{a_0}\binom{n-1-c}{j-1}\binom{d+a_0-c}{d}$$
more unmatched faces.
This accounts for the second sum in (ii).
\end{proof}

\begin{eg}
Consider the subcomplex $(10, 7, 4)$ of $J(18,8)$.
Notice that $j=4<8=k$, $$a_0 = \ \mathrm{max}\{n-d-k,\ n-d-1-j\}= \ \mathrm{max}\{18-10-8,\ 18-10-1-4\}= 3,$$ and $b_a = n-d-1-a=18-10-1-a= 7-a$.
We can apply Theorem \ref{thm:Betti} to see that 
$$|(d,i,j)|=\displaystyle\sum_{a=4}^{6}\binom{18}{a}\binom{17-a}{6-a} + \displaystyle\sum_{c=0}^{3}\binom{17-c}{3}\binom{13-c}{10}=596869.$$
\end{eg}

%


\section{Characters Arising from the Subcomplexes}

Before we identify the irreducible characters of the homology representations, it will be helpful to describe the character of the following representation.

\begin{defn}
Consider the CW complex $K$ associated with $J(n,k)$. We define $Y(a,b)$ to be the $\mathbb{C}S_n$-submodule of one of its chain groups equal to the span of the faces $F(S)$ such that $S$ has exactly $a$ 0's and $b$ 1's.
\end{defn}

\begin{lem}
\label{lem:neg}
Let $K'$ be the subcomplex of $K$ obtained by removing the unique $(n-1)$-face and let $s$ be a Coxeter generator of the group $S_n$. The induced action of $s$ on $H_{n-2}(K')$ and on the $(n-1)$-th chain group of $K$ is negation.
\end{lem}
\begin{proof}
While not exactly the same, this proof follows from some trivial modifications to the proof of \cite[Lemma 3.4]{Green10}. This lemma states the exact same result except for the polytope known as the half cube. It relies on the property that if we remove the unique $(n-1)$-dimensional face of the $(n-1)$-dimensional half cube, then it is homeomorphic to $S^{n-2}$, which is also true of any polytope.
\end{proof}

\begin{lem}
\label{lem:cong}
Let $W=S_n$ and let $W_I$ denote the parabolic subgroup generated by the set $I=\{s_1,s_2,\ldots,s_{n-a-b-1}\}\cup\{s_{n-a-b+1},\ldots,s_{n-b-1}\}\cup\{s_{n-b+1},\ldots,s_{n-1}\}$. Regarding $Y(a,b)$ as a $\mathbb{C}S_n$-module by extension of scalars, we have $Y(a,b)\cong (\mathrm{sgn}_{n-a-b} \otimes \mathrm{id}_{a} \otimes \mathrm{id}_{b}) \uparrow^W_{W_I}$.
\end{lem}
\begin{proof}
Let $e$ be the cell of the $CW$ complex $K$ corresponding to $F(S)\in Y(a,b)$ where $S=(*_1,*_2,\ldots,*_{n-a-b},0_1,\ldots,0_a,1_1,\ldots,1_b)$. Notice that the elements of the set $\{s_1,s_2,\ldots,s_{n-a-b-1}\}$ act on the sequence $S$ by permuting the *'s. 
In other words, each element of $\{s_1,s_2,\ldots,s_{n-a-b-1}\}$ fixes $F(S)$ setwise but not pointwise, and so by Lemma \ref{lem:neg} each one of these generators sends $e$ to $-e$. 
However the elements of the sets $\{s_{n-a-b+1},\ldots,s_{n-b-1}\}$ and $\{s_{n-b+1},\ldots,s_{n-1}\}$ act on $S$ by permuting the 0's and 1's respectively and hence each of these generators fixes the set of vertices pointwise, and fixes $e$.

We see that the subgroup $W_I$ stabilizes $F(S)$ and has order $(n-a-b)!a!b!$ and index $\binom{n}{n-a-b}\binom{a+b}{b}$ in $W$. Therefore by Proposition \ref{prop:faces}, $W_I$ is the full set stabilizer of $F(S)$. Now by Lemma \ref{lem:Coxeter} we see that every element of $Y(a,b)$ can be written as a linear combination of elements of the form $w^IF(S)$ for some $w^I \in W^I$. On the other hand, every element of $(\mathrm{sgn}_{n-a-b} \otimes \mathrm{id}_{a} \otimes \mathrm{id}_{b}) \uparrow^W_{W_I}$ can be written as a linear combination of elements of the form $w^I \otimes_{\mathbb{C}W_I} v$ where $v$ is the single basis element of the one dimensional module $(\mathrm{sgn}_{n-a-b} \otimes \mathrm{id}_{a} \otimes \mathrm{id}_{b})$. It now follows that the linear map $\phi:Y(a,b)\longrightarrow (\mathrm{sgn}_{n-a-b} \otimes \mathrm{id}_{a} \otimes \mathrm{id}_{b}) \uparrow^W_{W_I}$ such that $\phi(w^IF(S)) = w^I \otimes_{\mathbb{C}W_I} v$ is an isomorphism of $\mathbb{C}S_n$-modules.
\end{proof}

\begin{thm}[\textbf{The Pieri Rules}]
\label{thm:Pieri}
Given a partition $\mu$ of $n$ we can decompose the characters below into irreducibles in the following way:

\begin{enumerate}
\item [\rm(i)] $(\chi^\mu \times \chi^{[m]}) \uparrow^{S_{n+m}}_{S_n\times S_m} = \displaystyle\sum_{\lambda\in C_m(\mu)}\chi^\lambda,$
\item [\rm(ii)] $(\chi^\mu \times \chi^{[1^m]}) \uparrow^{S_{n+m}}_{S_n\times S_m} = \displaystyle\sum_{\lambda\in R_m(\mu)}\chi^\lambda,$
\end{enumerate}
where $C_m(\mu)$ (respectively, $R_m(\mu)$) denotes the set of partitions that can be obtained from $\mu$ by adding $m$ boxes to the Young diagram of $\mu$ so that there is at most one new box per column (respectively, row).
\end{thm}

\begin{proof}
This can be found in \cite[\S 2.2]{Fulton}.
\end{proof}

\begin{eg}
In Figure \ref{fig:Pieri} below we see the Young diagram for the partition $\lambda=[3,2,2,1]$ at the top. The remaining six Young diagrams are the elements of $C_2(\lambda)$ where the shaded boxes represent the two boxes added to $[3,2,2,1]$. Therefore by the Pieri rules, we have  $$(\chi^{[3,2,2,1]} \times \chi^{[2]}) \uparrow^{S_{10}}_{S_8\times S_2} = \chi^{[4,3,2,1]} + \chi^{[4,2,2,2]} + \chi^{[4,2,2,1,1]} + \chi^{[3,3,2,2]} + \chi^{[3,3,2,1,1]} + \chi^{[3,2,2,2,1]}.$$
\end{eg}

\def \sscale{.7}
\def \vvspace{.5}
\begin{figure}[htbp]
\begin{center}
\def \labelh{-4}
\def \split{4}
\begin{tikzpicture}[scale=\sscale]

\draw [step=1,thin] (0,0) grid (3, 1);
\draw [step=1,thin] (0,-1) grid (2, 0);
\draw [step=1,thin] (0,-2) grid (2, -1);
\draw [step=1,thin] (0,-3) grid (1, -2);
\draw node at (1.5,\labelh) {$[3,2,2,1]$};

\end{tikzpicture}

\vspace{\vvspace in}
\begin{tikzpicture}[scale=\sscale]
\def \split{4}

\fill [shade] (3,0) rectangle (4,1);
\fill [shade] (2,-1) rectangle (3,0);
\draw [step=1,thin] (0,0) grid (3, 1);
\draw [step=1,thin] (0,-1) grid (2, 0);
\draw [step=1,thin] (0,-2) grid (2, -1);
\draw [step=1,thin] (0,-3) grid (1, -2);
\draw [step=1,thin] (3,0) grid (4,1);
\draw [step=1,thin] (2,-1) grid (3,0);
\draw node at (2,\labelh) {$[4,3,2,1]$};

\fill [shade] (6+\split+3,0) rectangle (6+\split+4,1);
\fill [shade] (6+\split+1,-3) rectangle (6+\split+2,-2);
\draw [step=1,thin] (6+\split,0) grid (9+\split, 1);
\draw [step=1,thin] (6+\split,-1) grid (8+\split, 0);
\draw [step=1,thin] (6+\split,-2) grid (8+\split, -1);
\draw [step=1,thin] (6+\split,-3) grid (7+\split, -2);
\draw [step=1,thin] (6+\split+3,0) grid (6+\split+4,1);
\draw [step=1,thin] (6+\split+1,-3) grid (6+\split+2,-2);
\draw node at (2+6+\split,\labelh) {$[4,2,2,2]$};

\end{tikzpicture}

\vspace{\vvspace in}
\begin{tikzpicture}[scale=\sscale]
\def \split{4}

\fill [shade] (3,0) rectangle (4,1);
\fill [shade] (0,-4) rectangle (1,-3);
\draw [step=1,thin] (0,0) grid (3, 1);
\draw [step=1,thin] (0,-1) grid (2, 0);
\draw [step=1,thin] (0,-2) grid (2, -1);
\draw [step=1,thin] (0,-3) grid (1, -2);
\draw [step=1,thin] (3,0) grid (4,1);
\draw [step=1,thin] (0,-4) grid (1,-3);
\draw node at (2,\labelh-1) {$[4,2,2,1,1]$};

\fill [shade] (6+\split+2,-1) rectangle (6+\split+3,0);
\fill [shade] (6+\split+1,-3) rectangle (6+\split+2,-2);
\draw [step=1,thin] (6+\split,0) grid (9+\split, 1);
\draw [step=1,thin] (6+\split,-1) grid (8+\split, 0);
\draw [step=1,thin] (6+\split,-2) grid (8+\split, -1);
\draw [step=1,thin] (6+\split,-3) grid (7+\split, -2);
\draw [step=1,thin] (6+\split+2,-1) grid (6+\split+3,0);
\draw [step=1,thin] (6+\split+1,-3) grid (6+\split+2,-2);
\draw node at (6+\split+4,0) {};
\draw node at (2+6+\split,\labelh-1) {$[3,3,2,2]$};

\end{tikzpicture}

\vspace{\vvspace in}
\begin{tikzpicture}[scale=\sscale]
\def \split{4}

\fill [shade] (2,-1) rectangle (3,0);
\fill [shade] (0,-4) rectangle (1,-3);
\draw [step=1,thin] (0,0) grid (3, 1);
\draw [step=1,thin] (0,-1) grid (2, 0);
\draw [step=1,thin] (0,-2) grid (2, -1);
\draw [step=1,thin] (0,-3) grid (1, -2);
\draw [step=1,thin] (2,-1) grid (3,0);
\draw [step=1,thin] (0,-4) grid (1,-3);
\draw node at (2,\labelh-1) {$[3,3,2,1,1]$};

\fill [shade] (6+\split+0,-4) rectangle (6+\split+1,-3);
\fill [shade] (6+\split+1,-3) rectangle (6+\split+2,-2);
\draw [step=1,thin] (6+\split,0) grid (9+\split, 1);
\draw [step=1,thin] (6+\split,-1) grid (8+\split, 0);
\draw [step=1,thin] (6+\split,-2) grid (8+\split, -1);
\draw [step=1,thin] (6+\split,-3) grid (7+\split, -2);
\draw [step=1,thin] (6+\split+0,-4) grid (6+\split+1,-3);
\draw [step=1,thin] (6+\split+1,-3) grid (6+\split+2,-2);
\draw node at (6+\split+4,0) {};
\draw node at (2+6+\split,\labelh-1) {$[3,2,2,2,1]$};

\end{tikzpicture}
\vspace{.3 in}
\caption{The Pieri Rules}
\label{fig:Pieri}
\end{center}
\end{figure}
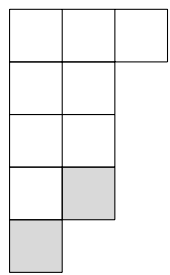

Here we introduce some more notation in order to simplify some of the following statements.
\begin{defn}
Let $a,b\in \mathbb{Z}$ such that $a\geq b$. Then we define $f_1, f_2, f_3,$ and $f_4$ in the following way:
\begin{enumerate}
  \item [(i)] $f_1(a,b) = \displaystyle\sum_{c=0}^{b-1}\chi^{[a+c,\ b-c,\ 1^{n-a-b}]}$;
  \item [(ii)] $f_2(a,b) = \displaystyle\sum_{c=0}^{b}\chi^{[a+c,\ b+1-c,\ 1^{n-a-b-1}]}$;
  \item [(iii)] $f_3(a,b) = \displaystyle\sum_{c=0}^{b-1}\chi^{[a+1+c,\ b-c,\ 1^{n-a-b-1}]}$;
  \item [(iv)] $f_4(a,b) = \displaystyle\sum_{c=0}^{b}\chi^{[a+1+c,\ b+1-c,\ 1^{n-a-b-2}]}$;
\end{enumerate}
where $\chi^{[a,\ b+1,\ 1^{n-a-b-1}]}=0$ if $a=b$ and $f_1(a,b) = f_3(a,b) = 0$ if $b=0$. Also, if $a<b$ then we define $f_i(a,b) = f_i(b,a)$ for $1\leq i \leq 4$.
\end{defn}

\begin{lem}
\label{lem:charY}
The character of the module $Y(a,b)$ is given by $f_1(a,b)+f_2(a,b)+f_3(a,b)+f_4(a,b)$. 
\end{lem}
\begin{proof}
Suppose first that $a\geq b$.
By transitivity of induction and Lemma \ref{lem:cong}, we have
\begin{center}
$Y(a,b)\cong ((\mathrm{sgn}_{n-a-b} \otimes \mathrm{id}_{a} \otimes \mathrm{id}_{b}) \uparrow^{S_{n-a}\times S_b}_{S_{n-a-b}\times S_a\times S_b})\uparrow^{S_{n}}_{S_{n-a}\times S_b}$.
\end{center}
At this point we need to utilize the Pieri rule from Theorem \ref{thm:Pieri} (i) twice. We use it the first time to see that
\begin{center}
$(\chi^{[1^{n-a-b}]} \times \chi^{[a]} \times \chi^{[b]}) \uparrow^{S_{n-a}\times S_b}_{S_{n-a-b}\times S_a\times S_b} = (\chi^{[a+1,\ 1^{n-a-b-1}]}+\chi^{[a,\ 1^{n-a-b}]})\times \chi^{[b]}$
\end{center}
where $\chi^{[a,\ 1^{n-a-b}]}=0$ if $a=0$.
The assertion then follows from applying the same Pieri rule a second time and noticing that
\begin{center}
$(\chi^{[a,\ 1^{n-a-b}]}\times \chi^{[b]})\uparrow^{S_{n}}_{S_{n-a}\times S_b}=f_1(a,b)+f_2(a,b)$
\end{center}
and
\begin{center}
$(\chi^{[a+1,\ 1^{n-a-b-1}]}\times \chi^{[b]})\uparrow^{S_{n}}_{S_{n-a}\times S_b}=f_3(a,b)+f_4(a,b)$.
\end{center}
If $a<b$, then the assertion follows from the observation that $$(\chi^{[1^{n-a-b}]} \times \chi^{[a]} \times \chi^{[b]}) \uparrow^{S_{n}}_{S_{n-a-b}\times S_a\times S_b} = (\chi^{[1^{n-a-b}]} \times \chi^{[b]} \times \chi^{[a]}) \uparrow^{S_{n}}_{S_{n-a-b}\times S_b\times S_a}.$$
\end{proof}

\begin{lem}
\label{lem:canceldown}
We have
$$
  \displaystyle\sum_{c=0}^{b}(-1)^{b-c}\Big(f_1(a,c)+f_2(a,c)+f_3(a,c)+f_4(a,c)\Big) = \left\{ 
  \begin{array}{l l}
    f_2(a,b)+f_4(a,b) & \quad \text{if $a\geq b$}\\
    f_3(a,b)+f_4(a,b) & \quad \text{if $a<b$.}\\
  \end{array} \right.
$$
\end{lem}
\begin{proof}
We will prove this by induction on $b$. The base case when $b=0$, which implies $a\geq b$, is true since $f_1(a,b)=f_3(a,b)=0$ when $b=0$. To prove the inductive step, we break it into cases.

If $a\geq b$, then $a > b-1$ and by induction the alternating sum is equal to
$$-f_2(a,b-1)-f_4(a,b-1)+f_1(a,b)+f_2(a,b)+f_3(a,b)+f_4(a,b).$$
However when $a\geq b$, $f_2(a,b-1) = f_1(a,b)$ and $f_4(a,b-1) = f_3(a,b)$ as desired.

If $a=b-1$, then by induction the alternating sum is equal to
$$-f_2(a,b-1)-f_4(a,b-1)+f_1(a,b)+f_2(a,b)+f_3(a,b)+f_4(a,b).$$
However, 
$$f_2(a,b-1) = f_2(b-1,b-1) = \displaystyle\sum_{c=0}^{b-1}\chi^{[b-1+c,\ b-c,\ 1^{n-2b+1}]}=\displaystyle\sum_{c=1}^{b-1}\chi^{[b-1+c,\ b-c,\ 1^{n-2b+1}]}$$ 
and 
$$f_1(a,b)=f_1(b-1,b) = f_1(b,b-1) = \displaystyle\sum_{c=0}^{b-2}\chi^{[b+c,\ b-1-c,\ 1^{n-2b+1}]}.$$ 
So after reindexing we see that $f_2(a,b-1)=f_1(a,b)$ in this case.
Also, 
$$f_4(a,b-1) = f_4(b-1,b-1) = \displaystyle\sum_{c=0}^{b-1}\chi^{[b+c,\ b-c,\ 1^{n-2b}]} = f_2(b,b-1) = f_2(a,b)$$
as desired.

If $a<b-1$, then by induction the alternating sum is equal to
$$-f_3(a,b-1)-f_4(a,b-1)+f_1(a,b)+f_2(a,b)+f_3(a,b)+f_4(a,b).$$
However, 
$$f_3(a,b-1) = f_3(b-1,a) = \displaystyle\sum_{c=0}^{a-1}\chi^{[b+c,\ a-c,\ 1^{n-a-b}]} = f_1(b,a) = f_1(a,b)$$
and 
$$f_4(a,b-1) = f_4(b-1,a) = \displaystyle\sum_{c=0}^{a}\chi^{[b+c,\ a+1-c,\ 1^{n-a-b-1}]} = f_2(b,a) = f_2(a,b)$$
when $a<b-1$, as desired.
\end{proof}

\begin{lem}
\label{lem:cancelleft}
We have
$$\displaystyle\sum_{\substack{c=0 \\ c < b}}^{a}(-1)^{a-c}\Big(f_3(c,b)+f_4(c,b)\Big)+\displaystyle\sum_{c=b}^{a}(-1)^{a-c}\Big(f_2(c,b)+f_4(c,b)\Big)=f_4(a,b).$$
\end{lem}
\begin{proof}
We will prove this by induction on $a$. 
First let 
$$LS = \displaystyle\sum_{\substack{c=0 \\ c < b}}^{a}(-1)^{a-c}\Big(f_3(c,b)+f_4(c,b)\Big)$$
and 
$$RS = \displaystyle\sum_{c=b}^{a}(-1)^{a-c}\Big(f_2(c,b)+f_4(c,b)\Big).$$

The base case when $a=0$ has two cases. 
If $b=0$ as well, then $LS=0$
and so the alternating sum becomes $f_2(0,0)+f_4(0,0)$, but $f_2(0,0)=0$ by definition.
If $b>0$, then the alternating sum becomes $f_3(0,0)+f_4(0,0)$, but $f_3(0,0)=0$ by definition.

To prove the inductive step, we break it into cases.
If $a<b$, then $RS =0$ and by induction, $LS = -f_4(a-1,b)+f_3(a,b)+f_4(a,b)$.
However when $a<b$, we have
$$f_4(a-1,b)=f_4(b,a-1)=\displaystyle\sum_{c=0}^{a-1}\chi^{[b+1+c,\ a-c,\ 1^{n-a-b-1}]}= f_3(b,a) =f_3(a,b)$$
by definition.
Therefore $LS + RS = -f_4(a-1,b)+f_3(a,b)+f_4(a,b) + 0 = f_4(a,b)$ as desired.

If $a=b$, then $LS = -f_4(a-1,b)$ by induction, and $RS = f_2(a,b)+f_4(a,b)$.
However when $a=b$,
$$f_4(a-1,b)=f_4(b,b-1)=\displaystyle\sum_{c=0}^{b-1}\chi^{[b+1+c,\ b-c,\ 1^{n-2b-1}]}$$
and 
$$f_2(a,b) =f_2(b,b)=\displaystyle\sum_{c=0}^{b}\chi^{[b+c,\ b+1-c,\ 1^{n-2b-1}]}=\displaystyle\sum_{c=1}^{b}\chi^{[b+c,\ b+1-c,\ 1^{n-2b-1}]}$$
by definition. After reindexing we see that $f_4(a-1,b) = f_2(a,b)$.
Therefore $$LS + RS = -f_4(a-1,b) + f_2(a,b)+f_4(a,b) =f_4(a,b)$$ as desired.

If $a > b$, then by induction $LS + RS = -f_4(a-1,b)+f_2(a,b)+f_4(a,b)$.
However when $a > b$,
$$f_4(a-1,b) = \displaystyle\sum_{c=0}^{b}\chi^{[a+c,\ b+1-c,\ 1^{n-a-b-1}]} = f_2(a,b)$$
by definition.
Therefore $LS + RS = -f_4(a-1,b)+f_2(a,b)+f_4(a,b)=f_4(a,b)$ as desired.
\end{proof}


\section{Characters of the Homology Representations}

In the previous section, a description of the characters that arise from the modules $Y(a,b)$ was given in Lemma \ref{lem:charY}. Our goal now is to use the following theorem to give a similar description for the characters of the homology representations.

\begin{thm}[\textbf{Hopf trace formula}] Let $K$ be a finite complex with (integral) chain groups $C_p(K)$ and homology groups $H_p(K)$. Let $T_p(K)$
be the torsion subgroup of $H_p(K$). Let $\phi : C_p(K) \longrightarrow C_p(K)$ be a chain map, and let $\phi_*$ be the induced map on homology. Then we have
\end{thm}
\begin{center}
$\displaystyle\sum_{p}(-1)^p$ tr$(\phi,C_p(K)) = \displaystyle\sum_{p}(-1)^p$ tr$(\phi_*,H_p(K)/T_p(K)).$
\end{center}
\begin{proof}
See \cite[Theorem 22.1]{Munkres}.
\end{proof}

\begin{lem}
\label{lem:allzero}
Consider the CW complex $K$ associated with $J(n,k)$ and let $C_p$ denote the chain groups of $K$. Let $\phi$ be a chain map of this chain complex. Then we have
\end{lem}
\begin{center}
$\displaystyle\sum_{p\geq -1} (-1)^p$ tr$(\phi, C_p)=0$.
\end{center}
\begin{proof}
This follows after noticing that the hypersimplex $J(n,k)$ is contractible and hence has trivial reduced homology. Therefore the right side of the Hopf trace formula becomes zero as desired.
\end{proof}

\begin{lem}
\label{lem:notin}
Consider any subcomplex $K'$ of $K$ whose reduced homology is concentrated in a single degree $d$. 
Then we have
\end{lem}
\begin{center}
tr$(\phi_*,H_{d}(K')) = \displaystyle\sum_{(a,b)\notin K'} (-1)^{n-d-a-b}$tr$(\phi, Y(a,b))$.
\end{center}
\begin{proof}
By Theorem \ref{thm:class} (i) we know that $H_{d}(K')$ has no torsion subgroup and so applying the Hopf trace formula gives
\begin{center}
$\displaystyle\sum_{(a,b)\in K'}(-1)^{n-a-b-1}$tr$(\phi,Y(a,b)) = \displaystyle\sum_p(-1)^p$tr$(\phi_*,H_p(K'))$.
\end{center}
Since the reduced homology of $K'$ is concentrated in degree $d$, this simplifies to
\begin{center}
$(-1)^d$tr$(\phi_*,H_d(K')) = \displaystyle\sum_{(a,b)\in K'}(-1)^{n-a-b-1}$tr$(\phi,Y(a,b))$.
\end{center}
Next, by Lemma \ref{lem:allzero} we know that
\begin{center}
$\displaystyle\sum_{(a,b)\in K'}(-1)^{n-a-b-1}$tr$(\phi,Y(a,b)) + \displaystyle\sum_{(a,b)\notin K'}(-1)^{n-a-b-1}$tr$(\phi,Y(a,b)) = 0$.
\end{center}
Combining this with the preceding equation then gives us
\begin{center}
$(-1)^d$tr$(\phi_*,H_d(K')) = \displaystyle\sum_{(a,b)\notin K'}(-1)^{n-a-b}$tr$(\phi,Y(a,b))$.
\end{center}
Finally, we multiply each side by $(-1)^{-d}$ to get the result.
\end{proof}

\begin{eg}
Consider the subcomplexes $(10,7,4)$ and $(7,7,8)$ of $J(18, 8)$ shown in Figure \ref{fig:altsum} below. 
In order to identify the character of the homology representation we will use Lemma \ref{lem:notin} along with Lemmas \ref{lem:canceldown} and \ref{lem:cancelleft} in the following way. 
First we will only consider the character of $Y(a,b)$ if the square labeled by $(a,b)$ is not shaded by Lemma \ref{lem:notin}. 
Then we will use Lemma \ref{lem:canceldown} to simplify the alternating sum by canceling as we move down columns as seen in the left diagrams.
Finally, we will use Lemma \ref{lem:cancelleft} to further simply this sum by canceling across the row above the highest completely shaded row as seen in the right diagrams.
\end{eg}

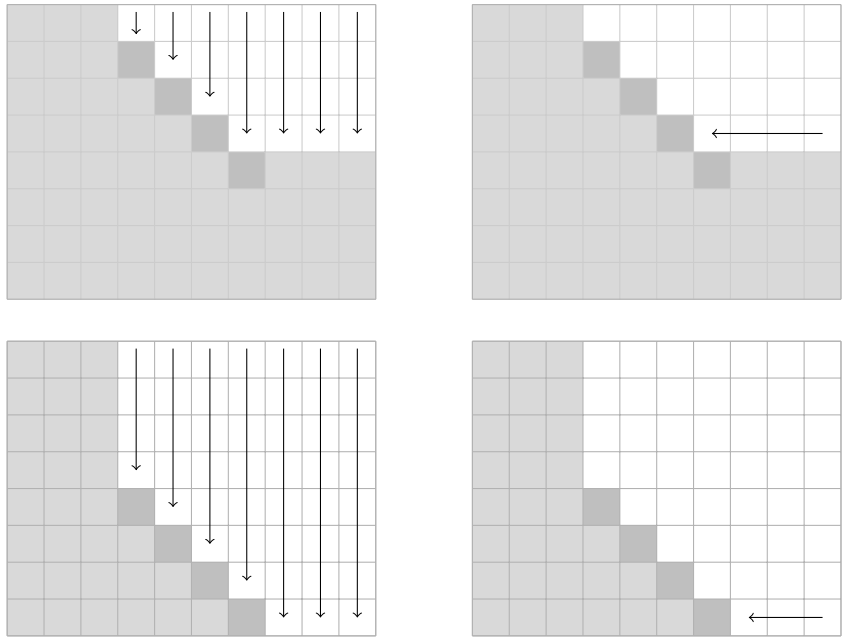
\begin{figure}[htbp]
\def\sca{.49}
\def\split{.4}
\begin{center}
\begin{tikzpicture}[scale=\sca]
\def \length{10}
\def \height{8}
\def \lengthb{9} 
\def \heightb{7} 
\fill [shade] (0,0) rectangle (3,8);
\fill [shade] (3,7) rectangle (4,0);
\fill [shade] (4,6) rectangle (5,0);
\fill [shade] (5,5) rectangle (6,0);
\fill [shade] (6,0) rectangle (10,4);
\fill [shade b] (3,6) rectangle (4,7);
\fill [shade b] (4,5) rectangle (5,6);
\fill [shade b] (5,4) rectangle (6,5);
\fill [shade b] (6,3) rectangle (7,4);
\draw [step=1,thin,gray!40] (0,0) grid (\length, \height);
\draw [linee] (0,0) -- (0,\height);
\draw [linee] (0,0) -- (\length,0); 
\draw [linee] (\length, \height) -- (0,\height);
\draw [linee] (\length, \height) -- (\length,0);
\draw [->] (9.5,7.8) -- (9.5,4.5);
\draw [->] (8.5,7.8) -- (8.5,4.5);
\draw [->] (7.5,7.8) -- (7.5,4.5);
\draw [->] (6.5,7.8) -- (6.5,4.5);
\draw [->] (5.5,7.8) -- (5.5,5.5);
\draw [->] (4.5,7.8) -- (4.5,6.5);
\draw [->] (3.5,7.8) -- (3.5,7.2);
\end{tikzpicture}
\hspace{.4in}
\begin{tikzpicture}[scale=\sca]
\def \length{10}
\def \height{8}
\def \lengthb{9} 
\def \heightb{7} 
\fill [shade] (0,0) rectangle (3,8);
\fill [shade] (3,7) rectangle (4,0);
\fill [shade] (4,6) rectangle (5,0);
\fill [shade] (5,5) rectangle (6,0);
\fill [shade] (6,0) rectangle (10,4);
\fill [shade b] (3,6) rectangle (4,7);
\fill [shade b] (4,5) rectangle (5,6);
\fill [shade b] (5,4) rectangle (6,5);
\fill [shade b] (6,3) rectangle (7,4);
\draw [step=1,thin,gray!40] (0,0) grid (\length, \height);
\draw [linee] (0,0) -- (0,\height);
\draw [linee] (0,0) -- (\length,0); 
\draw [linee] (\length, \height) -- (0,\height);
\draw [linee] (\length, \height) -- (\length,0);
\draw [->] (9.5,4.5) -- (6.5,4.5);
\end{tikzpicture}

\vspace{.2in}

\begin{tikzpicture}[scale=\sca]
\def \length{10}
\def \height{8}
\def \lengthb{9} 
\def \heightb{7} 
\fill [shade] (0,0) rectangle (3,8);
\fill [shade] (3,0) rectangle (4,4);
\fill [shade] (4,0) rectangle (5,3);
\fill [shade] (5,0) rectangle (6,2);
\fill [shade] (6,0) rectangle (7,1);
\fill [shade b] (3,3) rectangle (4,4);
\fill [shade b] (4,2) rectangle (5,3);
\fill [shade b] (5,1) rectangle (6,2);
\fill [shade b] (6,0) rectangle (7,1);
\draw [step=1,thin,linee] (0,0) grid (\length, \height);
\draw [linee] (0,0) -- (0,\height);
\draw [linee] (0,0) -- (\length,0); 
\draw [linee] (\length, \height) -- (0,\height);
\draw [linee] (\length, \height) -- (\length,0);
\draw [->] (9.5,7.8) -- (9.5,0.5);
\draw [->] (8.5,7.8) -- (8.5,0.5);
\draw [->] (7.5,7.8) -- (7.5,0.5);
\draw [->] (6.5,7.8) -- (6.5,1.5);
\draw [->] (5.5,7.8) -- (5.5,2.5);
\draw [->] (4.5,7.8) -- (4.5,3.5);
\draw [->] (3.5,7.8) -- (3.5,4.5);
\end{tikzpicture}
\hspace{.4in}
\begin{tikzpicture}[scale=\sca]
\def \length{10}
\def \height{8}
\def \lengthb{9} 
\def \heightb{7} 
\fill [shade] (0,0) rectangle (3,8);
\fill [shade] (3,0) rectangle (4,4);
\fill [shade] (4,0) rectangle (5,3);
\fill [shade] (5,0) rectangle (6,2);
\fill [shade] (6,0) rectangle (7,1);
\fill [shade b] (3,3) rectangle (4,4);
\fill [shade b] (4,2) rectangle (5,3);
\fill [shade b] (5,1) rectangle (6,2);
\fill [shade b] (6,0) rectangle (7,1);
\draw [step=1,thin,linee] (0,0) grid (\length, \height);
\draw [linee] (0,0) -- (0,\height);
\draw [linee] (0,0) -- (\length,0); 
\draw [linee] (\length, \height) -- (0,\height);
\draw [linee] (\length, \height) -- (\length,0);
\draw [->] (9.5,0.5) -- (7.5,0.5);
\end{tikzpicture}
\caption{Alternating Sum of Characters.}
\label{fig:altsum}
\end{center}
\end{figure}

\begin{lem}
\label{lem:char}
Let $K'=(d,i,j)$ be a subcomplex of $K$.
Let $a_1=n-d-1-j$ and let $b_a=n-d-a-1$.
%
If $j=k$, then the character of the representation of $S_n$ on the $d$-th homology of the complex $K'$, $\chi(H_d(K'))$, is given by
\begin{equation}
\label{eq:1}
f_4(a_1,j-1) + \left(\displaystyle\sum_{\substack{a=a_1+1 \\ a\geq b_a-1}}^{i-1}f_2(a,b_a-1)+f_4(a,b_a-1)\right) + \left(\displaystyle\sum_{\substack{a=a_1+1 \\ a<b_a-1}}^{i-1}f_3(a,b_a-1)+f_4(a,b_a-1)\right).
\end{equation}
\end{lem}
\begin{proof}
By Lemmas \ref{lem:charY} and \ref{lem:notin}, we have
$$\chi(H_d(K')) = \displaystyle\sum_{(a,b)\notin K'} (-1)^{n-d-a-b}\Big(f_1(a,b)+f_2(a,b)+f_3(a,b)+f_4(a,b)\Big).$$ 

Since $j=k$, the diagram of $K'$ has no rows that are completely shaded in and $a_0 = n-d-k$.
Let $Q$ denote the rightmost shaded square in the bottom row. 
The square $Q$ has label $(a_0,b_0)$ by Lemma \ref{lem:label}.
Consider the set of squares in the columns strictly right of $Q$. 
If the top square in one of these columns is labeled by $(a,0)$, then the bottom square in this column is labeled by $(a,k-1)$.
None of these squares are shaded and so we can apply Lemma \ref{lem:canceldown} and see that the contribution of this column to Equation \ref{eq:1} is either 
$\pm\Big(f_2(a,k-1)+f_4(a,k-1)\Big)$, if $a\geq k-1$, or $\pm\Big(f_3(a,k-1)+f_4(a,k-1)\Big)$, if $a<k-1$.
As in Lemma \ref{lem:label}, we have $$a_0=\mathrm{max}\{n-d-k,\ n-d-1-j\},$$ and so the parity is given by $(-1)^{n-d-a-k+1}=(-1)^{a_0-1-a}$.
Summing over all $a$, where $0\leq a\leq a_0-1$, we see that
$$\displaystyle\sum_{\substack{a=0 \\ a < k-1}}^{a_0-1}(-1)^{a_0-1-a}\Big(f_3(a,k-1)+f_4(a,k-1)\Big)+ \displaystyle\sum_{a=k-1}^{a_0-1}(-1)^{a_0-1-a}\Big(f_2(a,k-1)+f_4(a,k-1)\Big)$$
is the contribution of all these columns to Equation \ref{eq:1}.
We can now apply Lemma \ref{lem:cancelleft} with $b=k-1$ to simplify this to $f_4(a_0-1,k-1)$.
Since $j=k$ and $a_0-1=a_1$, we have $f_4(a_0-1,k-1)=f_4(a_1,j-1)$. 

Now consider any column containing shaded squares that is not entirely shaded in. 
If $(a,b)$ is the label of a square above the highest shaded square in that column, then $a_1+1=a_0\leq a \leq i-1$ by Lemma \ref{lem:label}.
We also have $a+b=n-d-2$ since this square must contain faces of dimension $d+1$ by Lemma \ref{lem:donly}.
This implies that $b=b_a-1$.
Applying Lemma \ref{lem:canceldown}, we see that the contribution of each column to Equation \ref{eq:1} is either 
$\pm\Big(f_2(a,b_a-1)+f_4(a,b_a-1)\Big)$, if $a\geq b_a-1$, or $\pm\Big(f_3(a,b_a-1)+f_4(a,b_a-1)\Big)$, if $a<b_a-1$.
The parity is given by $(-1)^{n-d-a-b_a+1}=(-1)^{n-d-a-(n-d-a-2)}=1$.
Summing over all $a$, where $a_1+1\leq a \leq i-1$, we see that
$$\left(\displaystyle\sum_{\substack{a=a_1+1 \\ a\geq b_a-1}}^{i-1}f_2(a,b_a-1)+f_4(a,b_a-1)\right) + \left(\displaystyle\sum_{\substack{a=a_1+1 \\ a<b_a-1}}^{i-1}f_3(a,b_a-1)+f_4(a,b_a-1)\right)$$
is the contribution of all these columns to Equation \ref{eq:1} as desired.
\end{proof}

\begin{lem}
\label{lem:char2}
Let $K'=(d,i,j)$ be a subcomplex of $K$.
Let $a_1=n-d-1-j$ and let $b_a=n-d-a-1$.
%
If $j<k$, then the character of the representation of $S_n$ on the $d$-th homology of the complex $K'$, $\chi(H_d(K'))$, is given by
\begin{equation}
\label{eq:2}
f_4(a_1,j-1) + \left(\displaystyle\sum_{\substack{a=a_1+1 \\ a\geq b_a-1}}^{i-1}f_2(a,b_a-1)+f_4(a,b_a-1)\right) + \left(\displaystyle\sum_{\substack{a=a_1+1 \\ a<b_a-1}}^{i-1}f_3(a,b_a-1)+f_4(a,b_a-1)\right).
\end{equation}
\end{lem}
\begin{proof}
By Lemmas \ref{lem:charY} and \ref{lem:notin}, we have
$$\chi(H_d(K')) = \displaystyle\sum_{(a,b)\notin K'} (-1)^{n-d-a-b}\Big(f_1(a,b)+f_2(a,b)+f_3(a,b)+f_4(a,b)\Big).$$ 

Since $j<k$, the diagram of $K'$ has some rows that are completely shaded in and $a_0 = n-d-1-j = a_1$.
Also since $j<k$, the square $M$ is not in the top right corner and has label $(m_0,j)$ by Definition \ref{defn:canon}.
Next note that $j>0$ by definition, so the top row is not completely shaded in. 
This implies that $M$ is not in the top row.
By Lemma \ref{lem:unmatch}, $M$ contains unmatched faces and so by Lemma \ref{lem:donly}, $M$ must contain faces of dimension $d$.
It follows from Remark \ref{rem:subcomp} that $d=n-m_0-j-1$ or that $m_0 = n-d-1-j = a_1$.

Consider the set of squares that are not in $K'$ and are either in the column containing $M$ or to the right of this column.
If the top square in one of these columns is labeled by $(a,0)$, then the lowest unshaded square in this column is labeled by $(a,j-1)$.
Applying Lemma \ref{lem:canceldown}, we see that the contribution of this column to Equation \ref{eq:2} is either 
$\pm\Big(f_2(a,j-1)+f_4(a,j-1)\Big)$, if $a\geq j-1$, or $\pm\Big(f_3(a,j-1)+f_4(a,j-1)\Big)$, if $a<j-1$.
The parity is given by $(-1)^{n-d-a-j+1}=(-1)^{a_1-a}$.
Summing over all $a$, where $0\leq a\leq a_1$, we see that
$$\displaystyle\sum_{\substack{a=0 \\ a < j-1}}^{a_1}(-1)^{a_1-a}\Big(f_3(a,j-1)+f_4(a,j-1)\Big)+ \displaystyle\sum_{a=j-1}^{a_1}(-1)^{a_1-a}\Big(f_2(a,j-1)+f_4(a,j-1)\Big)$$
is the contribution of all these columns to Equation \ref{eq:2}.
We can now apply Lemma \ref{lem:cancelleft} with $b=j-1$ to simplify this to $f_4(a_1,j-1)$.

Now consider any column strictly left of the column containing $M$ that is not entirely shaded in. 
If $(a,b)$ is the label of a square above the highest shaded square in that column, then $a_1+1\leq a \leq i-1$ by Lemma \ref{lem:label}.
We also have that $a+b=n-d-2$ since this square must contain faces of dimension $d+1$ by Lemma \ref{lem:donly}.
This implies that $b=b_a-1$.
Applying Lemma \ref{lem:canceldown}, we see that the contribution of each column to Equation \ref{eq:2} is either 
$\pm\Big(f_2(a,b_a-1)+f_4(a,b_a-1)\Big)$, if $a\geq b_a-1$, or $\pm\Big(f_3(a,b_a-1)+f_4(a,b_a-1)\Big)$, if $a<b_a-1$.
The parity is given by $(-1)^{n-d-a-b_a+1}=(-1)^{n-d-a-(n-d-a-2)}=1$.
Summing over all $a$, where $a_1+1\leq a \leq i-1$, we see that
$$\left(\displaystyle\sum_{\substack{a=a_1+1 \\ a\geq b_a-1}}^{i-1}f_2(a,b_a-1)+f_4(a,b_a-1)\right) + \left(\displaystyle\sum_{\substack{a=a_1+1 \\ a<b_a-1}}^{i-1}f_3(a,b_a-1)+f_4(a,b_a-1)\right)$$
is the contribution of all these columns to Equation \ref{eq:2} as desired.
\end{proof}

\begin{eg}
Consider again the subcomplex $K'=(10,7,4)$ of $J(18, 8)$ shown in Figure \ref{fig:altsum}. 
Notice that $j=4<8=k$, $a_1=n-d-1-j=18-10-1-4 = 3$, and $$b_a-1=n-d-a-2=18-10-a-2=6-a.$$
So we can apply Lemma \ref{lem:char2} to see that
$$\chi(H_{10}(K')) =f_4(3,3)+f_2(4,2)+f_4(4,2)+f_2(5,1)+f_4(5,1)+f_2(6,0)+f_4(6,0).$$
\end{eg}

\begin{lem}
\label{lem:pieri2}
Let $b_a=n-d-a-1$.
The character corresponding to the representation $$(\displaystyle\bigwedge^{d+1}E_{n-a} \otimes \mathrm{id}_{a})\uparrow^{S_n}_{S_{n-a}\times S_{a}}$$ is given by:
\begin{enumerate}
\item[\rm{(i)}] $f_2(a,b_a-1)+f_4(a,b_a-1)$ if $a\geq b_a-1$; 
\item[\rm{(ii)}] $f_3(b_a-1,a)+f_4(b_a-1,a)$ if $a< b_a-1$.
\end{enumerate}
\end{lem}
\begin{proof}
Since $d+1=n-a-b_a$, Proposition \ref{prop:ext} shows that the character
corresponding to $(\displaystyle\bigwedge^{d+1}E_{n-a} \otimes \mathrm{id}_{a})\uparrow^{S_n}_{S_{n-a}\times S_{a}}$ is given by $(\chi^{[b_a,\ 1^{n-a-b_a}]}\times \chi^{[a]})\uparrow^{S_n}_{S_{n-a}\times S_{a}}$.
At this point we need to utilize the Pieri rule from Theorem \ref{thm:Pieri} (i).

Suppose first that $a\geq b_a-1$. Applying the Pieri rule shows that $$(\chi^{[b_a,\ 1^{n-a-b_a}]}\times \chi^{[a]})\uparrow^{S_n}_{S_{n-a}\times S_{a}}$$ is equal to
$$\displaystyle\sum_{c=0}^{b_a-1}\chi^{[a+c,\ b_a-c,\ 1^{n-a-b_a}]} + 
  \displaystyle\sum_{c=0}^{b_a-1}\chi^{[a+1+c,\ b_a-c,\ 1^{n-a-b_a-1}]},$$
which by definition is equal to $f_2(a,b_a-1)+f_4(a,b_a-1)$.

Suppose next that $a< b_a-1$. Applying the Pieri rule shows that  
$$(\chi^{[b_a,\ 1^{n-a-b_a}]}\times \chi^{[a]})\uparrow^{S_n}_{S_{n-a}\times S_{a}}$$
is equal to
$$\displaystyle\sum_{c=0}^{a-1}\chi^{[b_a+c,\ a-c,\ 1^{n-a-b_a}]} + 
  \displaystyle\sum_{c=0}^{a}\chi^{[b_a+c,\ a+1-c,\ 1^{n-a-b_a-1}]},$$ 
which by definition is equal to $f_3(b_a-1,a) + f_4(b_a-1,a)$. 
\end{proof}

\begin{thm}
\label{thm:ext}
Let $K'=(d,i,j)$ be a subcomplex of $K$ and let $a_1=n-d-1-j$.
Let $V_{f_4(a,b)}$ be the representation corresponding to $f_4(a,b)$. 
Regarded as a $\mathbb{C}S_n$-module, the $d$-th homology group of $K'$ is isomorphic to 
$$
V_{f_4(a_1,j-1)}
+ \displaystyle\bigoplus_{\substack{a=a_1+1}}^{i-1}\left(\displaystyle\bigwedge^{d+1}E_{n-a} \otimes \mathrm{id}_{a}\right)\uparrow^{S_n}_{S_{n-a}\times S_{a}}.
$$ 
\end{thm}
\begin{proof}
Let $b_a=n-d-a-1$.
By Lemmas \ref{lem:char} and \ref{lem:char2}, it is enough to show that the sum of characters 
$$
f_4(a_1,j-1) + \left(\displaystyle\sum_{\substack{a=a_1+1 \\ a\geq b_a-1}}^{i-1}f_2(a,b_a-1)+f_4(a,b_a-1)\right) + \left(\displaystyle\sum_{\substack{a=a_1+1 \\ a<b_a-1}}^{i-1}f_3(a,b_a-1)+f_4(a,b_a-1)\right)
$$
is the same as the character corresponding to 
$$
V_{f_4(a_1,j-1)}
+ \displaystyle\bigoplus_{\substack{a=a_1+1}}^{i-1}\left(\displaystyle\bigwedge^{d+1}E_{n-a} \otimes \mathrm{id}_{a}\right)\uparrow^{S_n}_{S_{n-a}\times S_{a}}.
$$

By Lemma \ref{lem:pieri2}, the character corresponding to the representation $(\displaystyle\bigwedge^{d+1}E_{n-a} \otimes \mathrm{id}_{a})\uparrow^{S_n}_{S_{n-a}\times S_{a}}$ is given by $f_2(a,b_a-1)+f_4(a,b_a-1)$, if $a\geq b_a-1$, and $f_3(b_a-1,a)+f_4(b_a-1,a)$, if $a< b_a-1$.
Summing over all $a$ where $a_1+1\leq a\leq i-1$ gives the result.
\end{proof}

\begin{eg}
Consider the subcomplex $(10, 7, 4)$ of $J(18,8)$.
Notice that $d=10,\ i=7,\ j=4,$ and $a_1 =n-d-1-j=18-10-1-4= 3$.
We can apply Theorem \ref{thm:ext} to see that the $d$-th homology group of $(10, 7, 4)$ is isomorphic as an $S_n$-module to
$$
V_{f_4(3,3)}
+ \displaystyle\bigoplus_{\substack{a=4}}^{6}\left(\displaystyle\bigwedge^{11}E_{18-a} \otimes \mathrm{id}_{a}\right)\uparrow^{S_{18}}_{S_{18-a}\times S_{a}}.$$
\end{eg}


\bibliographystyle{amsplain}

\nocite{*}		

\bibliography{refs}		

\providecommand{\bysame}{\leavevmode\hbox to3em{\hrulefill}\thinspace}
\providecommand{\MR}{\relax\ifhmode\unskip\space\fi MR }
\providecommand{\MRhref}[2]{%
  \href{http://www.ams.org/mathscinet-getitem?mr=#1}{#2}
}
\providecommand{\href}[2]{#2}
\begin{thebibliography}{10}

\bibitem{Bjorner}
A.~Bj\"{o}rner, \emph{Topological {M}ethods}, Handbook of {C}ombinatorics,
  North-Holland, 1995, pp.~1819--1872.

\bibitem{BT}
A.~Borel and J.~Tits, \emph{Groupes r\'eductifs}, Publ. Math. I.H.E.S.
  \textbf{27} (195), 55--151.

\bibitem{Casselman}
W.~Casselman, \emph{Geometric rationality of {S}atake compactifications},
  Algebraic groups and Lie groups, Cambridge, University Press, Cambridge, UK,
  1997.

\bibitem{Forman02}
R.~Forman, \emph{A user's guide to discrete {M}orse theory}, S\'eminaire
  Loth\-aringien.

\bibitem{Forman04}
\bysame, \emph{Topics in combinatorial differential topology and geometry},
  IAS/Park City Mathematics series \textbf{14} (2004), 135--201.

\bibitem{Fulton}
W.~Fulton, \emph{Young {T}ableaux}, Cambridge University Press, Cambridge,
  1997.

\bibitem{CharactersCoxeter}
M.~Geck and G.~Pfeiffer, \emph{Characters of finite {C}oxeter groups and
  {I}wahori--{H}ecke algebras}, Oxford University Press, Oxford, 2000.

\bibitem{Geoghegan}
R.~Geoghegan, \emph{Topological methods in group theory}, Springer-Verlag, New
  York, 2007.

\bibitem{Green09}
R.M. Green, \emph{Homology representations arising from the half cube}, Adv.
  Math. \textbf{222} (2009), 216--239.

\bibitem{Green10}
\bysame, \emph{Homology representations arising from the half cube, {II}},
  Journal of Combinatorial Theory Series A \textbf{117} (2010), 1037--1048.

\bibitem{Hatcher}
A.~Hatcher, \emph{Algebraic {T}opology}, Cambridge University Press, Cambridge,
  UK, 2002.

\bibitem{Humphreys}
J.E. Humphreys, \emph{Reflection {G}roups and {C}oxeter {G}roups}, Cambridge
  University Press, Cambridge, 1990.

\bibitem{Munkres}
J.R. Munkres, \emph{Elements of algebraic topology}, Addison-Wesley, Menlo
  Park, CA, 1984.

\bibitem{Satake}
I.~Satake, \emph{On representations and compactifications of symmetric
  {R}iemannian spaces}, Ann. Math. \textbf{71} (1960), 77--110.

\bibitem{Stembridge}
J.R. Stembridge, \emph{A practical view of $\widehat{W}$}, Notes from the AIM
  workshop, Palo Alto, July 2006. Available online at
  http://liegroups.org/papers.

\bibitem{Ziegler95}
G.M. Ziegler, \emph{Lectures on {P}olytopes}, Springer-Verlag, New York, 1995.

\end{thebibliography}

\end{document}